\documentclass[11pt]{article}%
\usepackage[colorlinks,bookmarks=true]{hyperref}
\usepackage{amsmath}
\usepackage{graphicx}
\usepackage{amsfonts}
\usepackage{amssymb}%

\providecommand{\U}[1]{\protect\rule{.1in}{.1in}}
\newtheorem{theorem}{Theorem} [section]

\newtheorem{corollary}{Corollary}[section]

\newtheorem{definition}{Definition} [section]

\newtheorem{lemma}{Lemma}[section]

\newtheorem{proposition}{Proposition} [section]
\newtheorem{remark}{Remark} [section]

\newenvironment{proof}[1][Proof]{\textbf{#1.} }{\ \rule{1em}{1em}}
\numberwithin{equation}{section}
\newcommand{\p}{\partial}
\newcommand{\ep}{\epsilon}

\newcommand{\BFR}{\mathbf{R}}

\newcommand{\CJ}{\mathcal{J}}

\newcommand{\be}{\begin{equation}}
\newcommand{\ee}{\end{equation}}
\newcommand{\CX}{\mathcal{X}}
\newcommand{\CT}{\mathcal{T}}
\newcommand{\CG}{\mathcal{G}}
\newcommand{\CA}{\mathcal{A}}
\newcommand{\CH}{\mathcal{H}}

\begin{document}

\title{Separable Hamiltonian PDEs and Turning point principle for stability of
gaseous stars}
\author{Zhiwu Lin and Chongchun Zeng\\School of Mathematics\\Georgia Institute of Technology\\Atlanta, GA 30332, USA}
\date{}
\maketitle

\begin{abstract}
We consider stability of non-rotating gaseous stars modeled by the
Euler-Poisson system. Under general assumptions on the equation of states, we
proved a turning point principle (TPP) that the stability of the stars is
entirely determined by the mass-radius curve parameterized by the center
density. In particular, the stability can only change at extrema (i.e. local
maximum or minimum points) of the total mass. For very general equation of
states, TPP implies that for increasing center density the stars are stable up
to the first mass maximum and unstable beyond this point until next mass
extremum (a minimum). Moreover, we get a precise counting of unstable modes
and exponential trichotomy estimates for the linearized Euler-Poisson system.
To prove these results, we develop a general framework of separable
Hamiltonian PDEs. The general approach is flexible and can be used for many
other problems including stability of rotating and magnetic stars,
relativistic stars and galaxies.

\end{abstract}

\section{Introduction}

Consider a self-gravitating gaseous star satisfying the 3D Euler-Poisson
system%
\begin{equation}
\rho_{t}+\nabla\cdot\left(  \rho u\right)  =0,\label{continuity-EP}%
\end{equation}%
\begin{equation}
\rho\left(  v_{t}+u\cdot\nabla u\right)  =-\nabla p-\rho\nabla
V,\label{momentum-EP}%
\end{equation}%
\begin{equation}
\Delta V=4\pi\rho,\ \lim_{\left\vert x\right\vert \rightarrow\infty}V\left(
t,x\right)  =0,\label{Poisson-EP}%
\end{equation}
where $\rho\geq0$ is the density, $u\left(  t,x\right)  \in\mathbf{R}^{3}$ is
the velocity, $p=P\left(  \rho\right)  $ is the pressure and $V$ is the
self-consistent gravitational potential. Assume $P\left(  \rho\right)
\ $satisfies:
\begin{equation}
P\left(  s\right)  \in C^{1}\left(  0,\infty\right)  ,\ P^{\prime
}>0\label{assumption-P1}%
\end{equation}
and there exists $\gamma_{0}\in\left(  \frac{6}{5},2\right)  $ such that
\begin{equation}
\lim_{s\rightarrow0+}s^{1-\gamma_{0}}P^{\prime}\left(  s\right)
=K>0.\ \ \ \label{assumption-P2}%
\end{equation}
The assumptions (\ref{assumption-P2}) implies that the pressure $P\left(
\rho\right)  \thickapprox K\rho^{\gamma_{0}}$ for $\rho$ near $0$.

We consider the stability of non-rotating stars. Throughout the paper, {\it non-rotating stars} referred to  static equilibria of (\ref{momentum-EP}%
)-(\ref{Poisson-EP}) with $u=\vec{0}$. Note that any  traveling
solution of (\ref{momentum-EP})-(\ref{Poisson-EP}) with $u$ to be a
constant vector $\vec{c}\ $becomes static under the Galilean
transformation
\[
\left(  \rho\left(  x,t\right)  ,u\left(  x,t\right)  \right)  \rightarrow
\left(  \left(  \rho\left(  x+\mathbf{\ }\vec{c}t,t\right)  ,u\left(
x+\vec{c}t,t\right)  -\vec{c}\right)  \right)  .
\]
The density function of a compactly supported non-rotating star can be
shown to be radially symmetric (\cite{Gidas-Ni-Nirenberg}).

By Lemma
\ref{lemma-steady-small-density}, there exists $\mu_{\max}\in(0,+\infty]$ such
that for any center density $\rho_{\mu}\left(  0\right)  =\mu\in\left(
0,\mu_{\max}\right)  $, there exists a unique non-rotating star with the
density $\rho_{\mu}\left(  \left\vert x\right\vert \right)  $ supported inside
a ball with radius $R_{\mu}=R\left(  \mu\right)  <\infty$. In particular,
$\mu_{\max}=\infty$ when $\gamma_{0}\geq\frac{4}{3}\ $%
(\cite{uggala-Newtonian-stars}). (See also \cite{Rein-existence}
\cite{rein-rendall-compact} \cite{makino-84} for the proof when $\gamma
_{0}>\frac{4}{3}$). Denote
\[
M\left(  \mu\right)  =\int_{R^{3}}\rho_{\mu}\ dx=\int_{S_{\mu}}\rho_{\mu}\ dx
\]
to be the total mass of the star, where $S_{\mu}=\left\{  \left\vert
x\right\vert <R_{\mu}\right\}  $ is the support of $\rho_{\mu}$. We consider
the linear stability of this family of non-rotating gaseous stars $\rho_{\mu
}\left(  \left\vert x\right\vert \right)  $ for $\mu\in\left(  0,\mu_{\max
}\right)  $. Our main result is the following turning point principle.

\begin{theorem}
\label{Thm2-TPP}\textbf{\ }The linear stability of $\rho_{\mu}$ is fully
determined by the mass-radius curve parameterized by $\mu$. Let $n^{u}\left(
\mu\right)  $ be the number of unstable modes, namely the total algebraic
multiplicities of unstable eigenvalues. For small $\mu$, we have
\begin{equation}
n^{u}\left(  \mu\right)  =%
\begin{cases}
1\text{ (linear instability)} & \text{when }\gamma_{0}\in\left(  \frac{6}%
{5},\frac{4}{3}\right)  \\
0\text{ (linear stability)} & \text{ when }\gamma_{0}\in\left(  \frac{4}%
{3},2\right)
\end{cases}
.\label{formula-unstable-mode-small-mu}%
\end{equation}
The number $n^{u}\left(  \mu\right)  \ $can only change at mass extrema. For
increasing $\mu$, at a mass extrema point where $M^{\prime}(\mu)$ changes
sign, $n^{u}\left(  \mu\right)  $ increases by $1$ if $M^{\prime}%
(\mu)R^{\prime}\left(  \mu\right)  $ changes from $-$ to $+$ (i.e. the
mass-radius curve bends counterclockwise) and $n^{u}\left(  \mu\right)  $
decreases by $1$ if $M^{\prime}(\mu)R^{\prime}\left(  \mu\right)  $ changes
from $+$ to $-\ $(i.e. the mass-radius curve bends clockwise).
\end{theorem}

Here, the mass-radius curve is oriented in a coordinate plane where the
horizontal and vertical axes correspond to the support radius and mass of the
star respectively. Theorem \ref{Thm2-TPP} shows that the stability of
non-rotating stars and the number of unstable modes are entirely determined by
the mass-radius curve parameterized by the center density $\mu$. In
particular, the stability can only change at a center density with extremal
mass (i.e. maxima or minima of $M\left(  \mu\right)  $). The change of
stability at mass extrema is called turning point principle (TPP) in the
astrophysical literature for both Newtonian and relativistic stars. It was
usually based on heuristic arguments. As an example, we quote the following
arguments in \cite{straumann book} for relativistic stars:\textquotedblleft
Suppose that for a given equilibrium configuration a radial mode changes its
stability property, i.e., the frequency $\omega$ of this mode passes through
zero. This implies that there exist infinitesimally nearby equilibrium
configurations into which the given one can be transformed, without changing
the total mass. Hence if $\omega$ passes trough zero we have $M^{\prime
}\left(  \mu\right)  =0$.\textquotedblright\ Same arguments can also be found
in other astrophysical textbooks such as \cite{weinberg book}
\cite{shapiro-book-85} \cite{glendening-compact-star}. In Theorem
\ref{Thm2-TPP}, we give a rigorous justification of TPP for Newtonian stars.
Moreover, we obtain the precise counting of unstable modes from the
mass-radius curve. For relativistic stars, similar results can also be
obtained (\cite{lin-hadzic-tpp-euler-einstein}).

Besides above stability criteria, we obtain more detailed information about
the spectra of the linearized Euler-Poisson operator and exponential
trichotomy estimates for the linearized Euler-Poisson system, which will be
useful for the future study of nonlinear dynamics near the non-rotating stars.
To state these results, first we introduce some notations. Let $X_{\mu}%
,Y_{\mu}$ be the
weighted spaces $L_{\Phi^{\prime\prime}\left(  \rho_{\mu}\right)  }^{2}\left(
S_{\mu}\right)  $ and $\left(  L_{\rho_{\mu}}^{2}\left(  S_{\mu}\right)
\right)  ^{3}$, where the enthalpy $\Phi\left(  \rho\right)  >0$ is defined
by
\begin{equation}
\Phi\left(  0\right)  =\Phi^{\prime}\left(  0\right)  =0,\ \ \Phi
^{\prime\prime}\left(  \rho\right)  =\frac{P^{\prime}\left(  \rho\right)
}{\rho}.\label{defn-enthalpy}%
\end{equation}
Denote $\mathbf{X}=X_{\mu}\times Y_{\mu}$.
The linearized Euler-Poisson system at $\left(  \rho_{\mu},\vec{0}\right)  $
is
\begin{align}
&\sigma_{t}=-\nabla\cdot\left(  \rho_{\mu}v\right)  ,\label{LEP-1}\\
&v_{t}=-\nabla\left(  \Phi^{\prime\prime}\left(  \rho_{\mu}\right)
\sigma+V\right)  ,\ \ \ \label{LEP2}%
\end{align}
with $\Delta V=4\pi\rho$. Here, $\left(  \sigma,v\right)  \in\mathbf{X}\ $are
the density and velocity perturbations.

Define the operators
\begin{equation}
L_{\mu}=\Phi^{\prime\prime}\left(  \rho_{\mu}\right)  -4\pi\left(
-\Delta\right)  ^{-1}:X_{\mu}\rightarrow X_{\mu}^{\ast},\ \ A_{\mu}=\rho_{\mu
}:Y_{\mu}\rightarrow Y_{\mu}^{\ast}\label{defn-L-A}%
\end{equation}
and
\begin{equation}
B_{\mu}=-\nabla\cdot=-\operatorname{div}:Y_{\mu}^{\ast}\rightarrow X_{\mu
},\ \ \ B_{\mu}^{\prime}=\nabla:X_{\mu}^{\ast}\rightarrow Y_{\mu
}.\label{defn-B}%
\end{equation}
Here, for $\sigma\in X_{\mu}$, we denote
\[
\left(  -\Delta\right)  ^{-1}\sigma=\int_{S_{\mu}}\frac{1}{4\pi\left\vert
x-y\right\vert }\sigma\left(  y\right)  dy\ |_{S_{\mu}}\text{. }%
\]
Then (\ref{LEP-1})-(\ref{LEP2}) can be written in the Hamiltonian form%
\begin{equation}
\partial_{t}\left(
\begin{array}
[c]{c}%
\sigma\\
v
\end{array}
\right)  =\left(
\begin{array}
[c]{cc}%
0 & B_{\mu}\\
-B_{\mu}^{\prime} & 0
\end{array}
\right)  \left(
\begin{array}
[c]{cc}%
L_{\mu} & 0\\
0 & A_{\mu}%
\end{array}
\right)  \left(
\begin{array}
[c]{c}%
\sigma\\
v
\end{array}
\right)  =\mathcal{J}_{\mu}\mathcal{L}_{\mu}\left(
\begin{array}
[c]{c}%
\sigma\\
v
\end{array}
\right)  ,\label{hamiltonian-EP}%
\end{equation}
where the operators
\begin{equation}
\mathcal{J}_{\mu}=\left(
\begin{array}
[c]{cc}%
0 & B_{\mu}\\
-B_{\mu}^{\prime} & 0
\end{array}
\right)  :\mathbf{X}^{\ast}\mathbf{\rightarrow X},\ \mathcal{L}_{\mu}=\left(
\begin{array}
[c]{cc}%
L_{\mu} & 0\\
0 & A_{\mu}%
\end{array}
\right)  :\mathbf{X\rightarrow X}^{\ast},\label{defn-JL-EP}%
\end{equation}
are off-diagonal anti-selfdual and diagonal self-dual respectively. We call
systems like (\ref{hamiltonian-EP}) to be separable Hamiltonian systems.

In the following theorems and throughout this paper, we follow the tradition
in the astrophysics literature that \textquotedblleft non-radial"
perturbations refer to those modes corresponding to non-constant spherical
harmonics. See the more precise Definition \ref{D:non-radial} of the subspaces
$\mathbf{X}_{r}$ and $\mathbf{X}_{nr}$ of radial and non-radial perturbations
in Subsection \ref{SS:non-radial}.

\begin{theorem}
\label{Thm1-EP}(i) The steady state $\rho_{\mu}$, which is parameterized by
the $C^{1}$ parameter $\mu$, is spectrally stable to non-radial perturbations
in $\mathbf{X}_{nr}$ with isolated purely imaginary eigenvalues. The zero
eigenvalue is isolated with an infinite dimensional kernel space
\begin{align*}
\ker(\mathcal{J}_{\mu}\mathcal{L}_{\mu})=  &  \left\{  \left(
\begin{array}
[c]{c}%
0\\
u
\end{array}
\right)  \ |\ \int\rho_{\mu}\left\vert u\right\vert ^{2}dx<\infty
,\ \nabla\cdot\left(  \rho_{\mu}u\right)  =0\right\} \\
&  \oplus span\left\{  \left(
\begin{array}
[c]{c}%
\partial_{x_{i}}\rho_{\mu}\\
0
\end{array}
\right)  ,\ i=1,2,3\right\}  ,
\end{align*}
and the only generalized eigenvectors of $0$ are given by $(0,\partial_{x_{i}%
}\nabla\tilde{\zeta})^{T}$ with
\[
\mathcal{J}_{\mu}\mathcal{L}_{\mu}\left(
\begin{array}
[c]{c}%
0\\
\partial_{x_{i}}\nabla\tilde{\zeta}%
\end{array}
\right)  =\left(
\begin{array}
[c]{c}%
\partial_{x_{i}}\rho_{\mu}\\
0
\end{array}
\right)  ,\quad i=1,2,3,
\]
where $\tilde{\zeta}$ is defined in \eqref{E:G-kernel} and
\eqref{E:G-kernel-1}.\newline(ii) Under radial perturbations in $\mathbf{X}%
_{r}$, the spectra of the linearized system \eqref{LEP-1}-\eqref{LEP2} are isolated
eigenvalues with finite multiplicity,
\[
\ker(\mathcal{J}_{\mu}\mathcal{L}_{\mu})\cap\mathbf{X}_{r}=span\{(\partial
_{\mu}\rho_{\mu},0)^{T}\}
\]
and the steady state $\rho_{\mu}$ is spectrally stable to radial perturbations
if and only if $n^{-}\left(  D_{\mu}^{0}\right)  =1$ and $i_{\mu}=1$. Here,
the self-adjoint operator $D_{\mu}^{0}$ is defined in
(\ref{defn-D-mu-0-general}) and
\begin{equation}
i_{\mu}=%
\begin{cases}
1 & \text{if }M^{\prime}(\mu)\frac{d}{d\mu}\left(  \frac{M\left(  \mu\right)
}{R_{\mu}}\right)  >0\text{ or \ }M^{\prime}(\mu)=0\\
0 & \text{if }M^{\prime}(\mu)\frac{d}{d\mu}\left(  \frac{M\left(  \mu\right)
}{R_{\mu}}\right)  <0\text{ or }\frac{d}{d\mu}\left(  \frac{M\left(
\mu\right)  }{R_{\mu}}\right)  =0.
\end{cases}
. \label{defn-imu}%
\end{equation}
Moreover, the number of growing modes is
\begin{equation}
n^{u}\left(  \mu\right)  =n^{-}\left(  D_{\mu}^{0}\right)  -i_{\mu}.
\label{formula-unstable-mode}%
\end{equation}

\end{theorem}

The index $i_{\mu}$ in (\ref{defn-imu}) is well-defined, since $M^{\prime
}\left(  \mu\right)  $ and $\frac{d}{d\mu}\left(  \frac{M\left(  \mu\right)
}{R_{\mu}}\right)  $ can not be zero at the same point (Lemma
\ref{lemma-M-R-no critical}). The stability of non-rotating stars under
nonradial perturbations was known in the astrophysics literature as
Antonov-Lebowitz Theorem (\cite{antonov} \cite{lebowitz65}). Theorem
\ref{Thm1-EP} implies that the spectra of the linearized Euler-Poisson
equation at $\rho_{\mu}$ are contained in the imaginary axis except finitely
many unstable (stable) eigenvalues with finite algebraic multiplicity.

\begin{theorem}
\label{Thm3-trichotomy}The operator $\mathcal{J}_{\mu}\mathcal{L}_{\mu}$
generates a $C^{0}$ group $e^{t\mathcal{J}_{\mu}\mathcal{L}_{\mu}}$ of bounded
linear operators on $\mathbf{X}$ and there exists a decomposition%
\[
\mathbf{X}=E^{u}\oplus E^{c}\oplus E^{s},\quad
\]
with the following properties: (i) $E^{u}\left(  E^{s}\right)  $ consist only
of eigenvectors corresponding to negative (positive) eigenvalues of
$\mathcal{J}_{\mu}\mathcal{L}_{\mu}$ and
\begin{equation}
\dim E^{u}=\dim E^{s}=n^{-}\left(  D_{\mu}^{0}\right)  -i_{\mu}%
.\label{unstable-dimension-formula-EP}%
\end{equation}
(ii) The quadratic form $\left(  \mathcal{L}_{\mu}\cdot,\cdot\right)
_{\mathbf{X}}$ vanishes on $E^{u,s}$, but is non-degenerate on $E^{u}\oplus
E^{s}$, and
\[
E^{c}=\left\{
\begin{pmatrix}
\sigma\\
v
\end{pmatrix}
\in\mathbf{X}\mid\left\langle \mathcal{L}_{\mu}%
\begin{pmatrix}
\sigma\\
v
\end{pmatrix}
,%
\begin{pmatrix}
\sigma_{1}\\
v_{1}%
\end{pmatrix}
\right\rangle =0,\ \forall%
\begin{pmatrix}
\sigma_{1}\\
v_{1}%
\end{pmatrix}
\in E^{s}\oplus E^{u}\right\}  .
\]
(iii) $E^{c},E^{u},E^{s}$ are invariant under $e^{t\mathcal{J}_{\mu
}\mathcal{L}_{\mu}}$. Let $\lambda_{u}=\min\{\lambda\mid\lambda\in
\sigma(\mathcal{J}_{\mu}\mathcal{L}_{\mu}|_{E^{u}})\}>0$. Then there exist
$C_{0}>0$ such that
\begin{equation}%
\begin{split}
&  \left\vert e^{t\mathcal{J}_{\mu}\mathcal{L}_{\mu}}|_{E^{s}}\right\vert \leq
C_{0}e^{-\lambda_{u}t},\ t\geq0,\\
&  \left\vert e^{t\mathcal{J}_{\mu}\mathcal{L}_{\mu}}|_{E^{u}}\right\vert \leq
C_{0}e^{\lambda_{u}t},\ t\leq0,
\end{split}
\label{estimate-stable-unstable-EP}%
\end{equation}%
\begin{equation}
\left\vert e^{t\mathcal{J}_{\mu}\mathcal{L}_{\mu}}|_{E^{c}}\right\vert \leq
C_{0}(1+|t|),\ t\in\mathbb{R},\ \ \ \text{if \ }M^{\prime}(\mu)\neq
0,\label{estimate-center-EP-1}%
\end{equation}
and
\begin{equation}
\left\vert e^{t\mathcal{J}_{\mu}\mathcal{L}_{\mu}}|_{E^{c}}\right\vert \leq
C_{0}(1+|t|)^{2},\ t\in\mathbb{R},\ \ \ \text{if \ }M^{\prime}(\mu
)=0.\label{estimate-center-EP-2}%
\end{equation}

(iv) Suppose that $M^{\prime}(\mu)\neq0$. Then
\begin{equation}
|e^{t\mathcal{J}_{\mu}\mathcal{L}_{\mu}}|_{E^{c}\cap\mathbf{X}_{r}}|\leq
C,\ \label{estimate-center-radial-EP}%
\end{equation}
for some constant $C$. In particular, when $n^{-}\left(  D_{\mu}^{0}\right)
=1$ and $M^{\prime}(\mu)\frac{d}{d\mu}\left(  \frac{M\left(  \mu\right)
}{R_{\mu}}\right)  >0$, Lyapunov stability is true for radial perturbations in
the sense that
\begin{equation}
|e^{t\mathcal{J}_{\mu}\mathcal{L}_{\mu}}|_{\mathbf{X}_{r}}|\leq C.
\label{stability-general data-EP}%
\end{equation}

\end{theorem}

Above linear estimates will be useful for the future study of nonlinear
dynamics, particularly, the construction of invariant (stable, unstable and
center) manifolds for the nonlinear Euler-Poisson system. The $O(|t|)$ growth
in (\ref{estimate-center-EP-1}) is due to the nonradial generalized kernel
associated to the translation modes given in Theorem \ref{Thm1-EP} i). At the
mass extrema points, the $O(|t|^{2})$ growth in (\ref{estimate-center-EP-2})
is due to the radial generalized kernel associated to the mode of varying
center density given in Theorem \ref{Thm1-EP} ii). Lyapunov stability on the
radial center space $E^{c}\cap\mathbf{X}_{r}$ (under the non-degeneracy
condition $M^{\prime}(\mu)\neq0$) hints that the steady state might be
nonlinearly stable on the center manifold once constructed.

Theorems \ref{Thm1-EP}-\ref{Thm3-trichotomy} are applied to various examples
of equation of states. For Polytropic stars with $P\left(  \rho\right)
=K\rho^{\gamma}$ $\left(  \gamma\in\left(  \frac{6}{5},2\right)  \right)  $,
we recover the classical sharp instability criterion (\cite{ledoux58}
\cite{lin-ss-1997}) that $\gamma\in\left(  \frac{6}{5},\frac{4}{3}\right)  $.
Even for this case, our results give some new information not found in the
literature that there is only one unstable mode and Lyapunov stability is true
on the center space. Next, we consider more practical white dwarf stars with
$P\left(  \rho\right)  =Af\left(  B^{\frac{1}{3}}\rho^{\frac{1}{3}}\right)  $,
where $A,B$ are two constants and $f\left(  x\right)  $ is defined in
(\ref{defn-f-white dwarf}). It is proved in Corollary
\ref{cor-stability-white dwarf} that white dwarf stars $\rho_{\mu}\left(
\left\vert x\right\vert \right)  \ $are linearly Lyapunov stable for any
center density $\mu>0$. For stars with general equation of states, we prove in
Corollary \ref{cor-general-stars} that they are stable up to the first mass
maximum and unstable beyond this point until the next mass extrema (a
minimum). Examples for which the first mass maximum is obtained at a finite
center density including the asymptotically polytropic equation of states
satisfying that $P\left(  \rho\right)  \thickapprox\rho^{\gamma_{1}}$ (for
$\rho$ large) with $\gamma_{1}\in\left(  0,\frac{6}{5}\right)  $ or $\left(
\frac{6}{5},\frac{4}{3}\right)  $. We refer to Corollary \ref{cor-asymp-poly}
for more details.

There exist huge astrophysical literature on the stability of gaseous stars
(e.g. \cite{chandra-book-stellar} \cite{ledoux58} \cite{tassoul-1978}
\cite{shapiro-book-85} \cite{cox-book} \cite{kippenhahn 2012} and references
therein). We briefly mention some more recent mathematical works. Linear
instability of polytropic stars was studied in \cite{lin-ss-1997}. Nonlinear
instability for polytropic stars was proved in \cite{jang-instability} for
$\gamma\in\left(  \frac{6}{5},\frac{4}{3}\right)  $ and in \cite{deng-et-2002}
for $\gamma=\frac{4}{3}$. Nonlinear conditional stability was shown in
\cite{rein-stability-EP} for polytropic stars with $\gamma>\frac{4}{3}$, and
for white dwarf stars in \cite{luo-smoller-1}. In these works, stable stars
were constructed by solving variational problems, for example, by minimizing
the energy functional subject to the mass constraint. In a work under
preparation (\cite{lin-nonlinear-EP}), we will show that the linear
stability/instability criteria in Theorems \ref{Thm1-EP} and \ref{Thm2-TPP}
are also true on the nonlinear level.

In the rest of this introduction, we discuss the methods in our proof of
Theorems \ref{Thm1-EP}-\ref{Thm3-trichotomy}.\textit{ }Since the non-rotating
stars are spherically symmetric, radial and non-radial perturbations are
decoupled for the linearized Euler-Poisson equation. The stability for
nonradial perturbations was obtained in the astrophysical literature in 1960s
(\cite{antonov} \cite{lebowitz65}). The radial perturbations were usually
studied by the Eddington equation (\ref{ODE-SL-radial-oscilation}%
)-(\ref{SL-eqn-BC}), which is a singular Sturm-Liouville problem.

In this paper, we study stability of non-rotating stars in a Hamiltonian
framework. The linearized Euler-Poisson system can be written as a separable
Hamiltonian form (\ref{hamiltonian-EP}). In Section \ref{section-abstract}, we
first study general linear Hamiltonian PDEs of the separable form%
\begin{equation}
\partial_{t}\left(
\begin{array}
[c]{c}%
u\\
v
\end{array}
\right)  =\left(
\begin{array}
[c]{cc}%
0 & B\\
-B^{\prime} & 0
\end{array}
\right)  \left(
\begin{array}
[c]{cc}%
L & 0\\
0 & A
\end{array}
\right)  \left(
\begin{array}
[c]{c}%
u\\
v
\end{array}
\right)  =\mathbf{JL}\left(
\begin{array}
[c]{c}%
u\\
v
\end{array}
\right)  ,\label{hamiltonian-separated}%
\end{equation}
where $u\in X,\ v\in Y$ and $X,Y$ are real Hilbert spaces. The triple $\left(
L,A,B\right)  $ is assumed to satisfy assumptions (G1)-(G4) in Section
\ref{section-abstract}, which roughly speaking require that $B:Y^{\ast}\supset
D(B)\rightarrow X$ \ is a densely defined closed operator, $L:X\rightarrow
X^{\ast}$ is bounded and self-dual with finitely many negative modes, and
$A:Y\rightarrow Y^{\ast}$ is bounded, self-dual and nonnegative. Those
assumptions qualify (\ref{hamiltonian-separated}) as a special case of the
general linear Hamiltonian systems studied in \cite{lin-zeng-hamiltonian}.
However, the special form of such systems ensures certain more specific
structure in the linear dynamics, in particular a more explicit formula for
unstable dimensions, all non-zero eigenvalues being semi-simple, a more
detailed block decomposition, an at most cubic bound of the degree of the
algebraic growth in $E^{c}$, \textit{etc.}

Adapting above framework to the linearized Euler-Poisson system
(\ref{hamiltonian-radial}) for radial perturbations, we obtain that the number
of unstable modes equals $n^{-}\left(  L_{\mu,r}|_{\overline{R\left(
B_{\mu,r}\right)  }}\right)  $, where $L_{\mu,r}$ and $B_{\mu,r}$ are the
restriction of operators $L_{\mu}$ and $B_{\mu}$ to radial functions. The
quadratic form $\left\langle L_{\mu,r}\cdot,\cdot\right\rangle $ is exactly
the second variation of the energy functional $E_{\mu}\left(  \rho\right)  $
defined in (\ref{energy functional -rho}) and $\overline{R\left(  B_{\mu
,r}\right)  }\ $is the space of radial perturbations preserving the total
mass. The unstable index formula (\ref{formula-unstable-mode}) follows from
these structures. In particular, the index $i_{\mu}$ (defined in
(\ref{defn-imu})) measures if the mass constraint can reduce the negative
modes of $L_{\mu,r}$ by one or not. The stability condition $L_{\mu
,r}|_{\overline{R\left(  B_{\mu,r}\right)  }}\geq0$ amounts to Chandrasekhar's
variational principle (\cite{chandra-variational-64}
\cite{binney-tremaine2008}) that the stable states should be energy minimizers
under the constraint of constant mass. Moreover, the separable Hamiltonian
formulation yields that the Sturm-Liouville operator in
(\ref{ODE-SL-radial-oscilation}) can be written in a factorized form
$B_{\mu,r}^{\prime}L_{\mu,r}B_{\mu,r}A_{\mu,r}$, where $A_{\mu,r}=\rho_{\mu}$
is a positive operator on $Y_{\mu,r}$. Compared with the traditional way of
treating the singular Sturm-Liouville operator (\ref{ODE-SL-radial-oscilation}%
), this factorized form is more convenient to prove self-adjointness and
discreteness of eigenvalues (Lemma \ref{lemma-self-adjoint-2}) without relying
on ODE techniques. We refer to Remark \ref{rm-Eddington-eqn} for more details.

To get TPP from Theorem \ref{Thm1-EP}, it is reduced to find $n^{-}\left(
L_{\mu,r}\right)  =n^{-}\left(  D_{\mu}^{0}\right)  $, where $D_{\mu}^{0}$ is
a second order ODE operator from the linearization of the steady state
equation. We use a continuity argument to find $n^{-}\left(  D_{\mu}%
^{0}\right)  $. First, for small $\mu$, $n^{-}\left(  D_{\mu}^{0}\right)  $ is
shown to be equal\ to the corresponding negative index for the Lane-Emden
stars with polytropic index $\gamma_{0}$ (defined in (\ref{assumption-P2})).
For Lane-Emden stars with $\gamma\in\left(  \frac{6}{5},2\right)  $, we show
that the negative index is always $1$. For general equation of states, it can
be shown that $n^{-}\left(  D_{\mu}^{0}\right)  =1\,$for small $\mu$. For
increasing $\mu\,$, we determine $n^{-}\left(  D_{\mu}^{0}\right)  $ by
keeping track of its changes. A key observation is that $D_{\mu}^{0}$ has
one-dimensional kernel only at critical points of the mass-radius ratio
$\frac{M\left(  \mu\right)  }{R_{\mu}}$. Therefore, $n^{-}\left(  D_{\mu}%
^{0}\right)  $ can only change at critical points of $\frac{M\left(
\mu\right)  }{R\left(  \mu\right)  }$. The jump of $n^{-}\left(  D_{\mu}%
^{0}\right)  $ at such critical points is shown to be exactly the jump of
$i_{\mu}$. This not only gives us a way to find $n^{-}\left(  D_{\mu}%
^{0}\right)  $ for any $\mu>0$, but also implies that the number of unstable
modes $n^{u}\left(  \mu\right)  \ $does not change when crossing a critical
point of $\frac{M\left(  \mu\right)  }{R_{\mu}}$. At extrema points of total
mass $M\left(  \mu\right)  $, $n^{-}\left(  D_{\mu}^{0}\right)  \,$\ remains
unchanged but $i_{\mu}$ must change from $0$ to $1$ (or from $1$ to $0$) if
the bending of the mass-radius curve is counterclockwise (or clockwise). This
proves TPP\ that the number of unstable modes can only change at extrema mass
and also give an explicit way to determine $n^{u}\left(  \mu\right)  $ from
the mass-radius curve. The exponential trichotomy estimates in Theorem
\ref{Thm3-trichotomy} follow form the general Theorems \ref{T:main} and
\ref{T:main2}.

The general framework of separable Hamiltonian PDEs in Section
\ref{section-abstract} is flexible and can be used for many other problems.
Hamiltonian systems in the separable form of (\ref{hamiltonian-separated})
appear in many other problems, which include nonlinear Klein-Gordon equations,
nonlinear Schr\"{o}dinger equations and 3D Vlasov-Maxwell systems for
collisionless plasmas etc. This framework was also used in the recent study of
stability of neutron stars modeled by Euler-Einstein equation
(\cite{lin-hadzic-tpp-euler-einstein}) and relativistic globular clusters
modeled by Vlasov-Einstein equation (\cite{mahir-lin-rein}). In particular,
for Euler-Einstein equation, a similar TPP can be proved
(\cite{lin-hadzic-tpp-euler-einstein}) for relativistic stars as in Theorem
\ref{T:main2}. More recently, the stability of rotating stars of Euler-Poisson
system was studied (\cite{lin-wang-rotating}) by the separable Hamiltonian approach.

This paper is organized as follows. Section 2 is about the abstract theory for
the separable linear Hamiltonian PDEs. Section 3 is about the stability of
non-rotating stars and is divided into several subsections. Section 3.1 is for
the existence of non-rotating stars. In section 3.2, the Hamiltonian
structures of linearized Euler-Poisson is studied. Section 3.3 is to find the
negative index $n^{-}\left(  D_{\mu}^{0}\right)  $ for all $\mu>0$. In Section
3.4, we derive the equations for non-radial perturbations and prove the
Antonov-Lebowitz theorem. In Section 3.5, TPP is proved for radial
perturbations. In Section 3.6, more explicit stability criteria are given for
several classes of equation of states. In the appendix, we outline the Lagrangian formulation of the Euler-Poisson system \eqref{continuity-EP}-\eqref{Poisson-EP} and its linearization.

\section{\label{section-abstract}Separable linear Hamiltonian PDE}

Let $X$ and $Y$ be real Hilbert spaces. We make the following assumptions on
$\left(  L,A,B\right)  $ in the Hamiltonian PDE (\ref{hamiltonian-separated}):

\begin{enumerate}
\item[(\textbf{G1})] The operator $B:Y^{\ast}\supset D(B)\rightarrow X$ and
its dual operator $B^{\prime}:X^{\ast}\supset D(B^{\prime})\rightarrow Y\ $are
densely defined and closed (and thus $B^{\prime\prime}=B$).

\item[(\textbf{G2})] The operator $A:Y\rightarrow Y^{\ast}$ is bounded and
self-dual (i.e. $A^{\prime}=A$ and thus $\left\langle Au,v\right\rangle $ is a
bounded symmetric bilinear form on $Y$). Moreover, there exist $\delta>0$ and
a closed subspace $Y_{+}\subset Y$ such that
\[
Y=\ker A\oplus Y_{+},\quad\langle Au,u\rangle\geq\delta\left\Vert u\right\Vert
_{Y}^{2},\;\forall u\in Y_{+}.
\]

\item[(\textbf{G3})] The operator $L:X\rightarrow X^{\ast}$ is bounded and
self-dual (i.e. $L^{\prime}=L$ \textit{etc.}) and there exists a decomposition
of $X$ into the direct sum of three closed subspaces
\begin{equation}
X=X_{-}\oplus\ker L\oplus X_{+},\quad n^{-}(L)\triangleq\dim X_{-}%
<\infty\label{decom-X}%
\end{equation}
satisfying

\begin{enumerate}
\item[(\textbf{G3.a})] $\left\langle Lu,u\right\rangle <0$ for all $u\in
X_{-}\backslash\{0\}$;

\item[(\textbf{G3.b})] there exists $\delta>0$ such that
\[
\left\langle Lu,u\right\rangle \geq\delta\left\Vert u\right\Vert ^{2}\ ,\text{
for any }u\in X_{+}.
\]

\end{enumerate}

\item[(\textbf{G4})] The above $X_{\pm}$ and $Y_{+}$ satisfy
\[
\ker(i_{X_{+}\oplus X_{-}})^{\prime}
\subset D(B^{\prime}), \quad\ker(i_{Y_{+}})^{\prime}\subset D(B).
\]

\end{enumerate}

\begin{remark}
We adopt the notations as in \cite{yosida-book}. For a densely defined linear
operator $A:X\rightarrow Y$ between Hilbert spaces $X,Y$, we use $A^{\prime
}:Y^{\ast}\rightarrow X^{\ast}$ and $A^{\ast}:$ $Y\rightarrow X$ for the dual
and adjoint operators of $A$ respectively. The operators $A^{\prime}$ and
$A^{\ast}$ are related by
\[
A^{\ast}=I_{X}A^{\prime}I_{Y}^{-1},
\]
where $I_{X}:X^{\ast}\rightarrow X$ and $I_{Y}:Y^{\ast}\rightarrow Y$ are the
isomorphisms defined by the Riesz representation theorem. Given a closed
subspace $X_{1}$ of a Hilbert space $X$, $i_{X_{1}}:X_{1}\rightarrow X$
denotes the embedding and ($i_{X_{1}})':X^{\ast}\rightarrow X_{1}^{\ast}$ the
dual operator with
\[
\ker(i_{X_{1}})^{\prime}=\left\{  f\in X^{\ast}\mid\langle f,x\rangle
=0,\,\forall x\in X_{1}\right\}  .
\]

\end{remark}

\begin{remark}
\label{remark-G4}The assumption (\textbf{G4}) for $L$ (or for $A$) is
satisfied automatically if $\dim\ker L<\infty$ (or $\dim\ker A<\infty$). See
Remark 2.3 in \cite{lin-zeng-hamiltonian} for details.
\end{remark}

In this paper, the above abstract framework will be applied the linearized
Euler-Poisson system to be studied in details, where $A$ is actually positive
definite. The more general semi-positive definiteness assumption on $A$ is
partially motivated by the focusing nonlinear Schr\"{o}dinger equation (NLS)
with energy subcritical or critical power nonlinearity,
\begin{equation}
iu_{t}=\Delta u+|u|^{p}u,\quad u:\mathbf{R}^{1+d}\rightarrow\mathbb{C}%
=\mathbf{R}^{2},\;p\in(1,\frac{4}{d-2}] \tag{NLS}%
\end{equation}
with the Hamiltonian
\[
H(u)=\int_{\mathbf{R}^{d}}\frac{1}{2}|\nabla u|^{2}-\frac{1}{p+2}|u|^{p}dx.
\]
There exist standing waves and steady waves in the subcritical and critical
cases, respectively,
\[
U_{\omega}(t,x)=e^{-i\omega t}\phi_{\omega}(x),\quad-\Delta\phi_{\omega
}+\omega\phi_{\omega}-\phi_{\omega}^{p+1}=0.
\]
For ground states, $\phi_{\omega}(x)$ is always radially symmetric and
positive, where $\omega>0$ if $p<\frac{4}{d-2}$ and $\omega=0$ if $p=\frac
{4}{d-2}$.
The linearization of (NLS) in the rotation frame $u(t,x)=e^{-i\omega t}v(t,x)$
at $v_{\omega}=\phi_{\omega}$ with $v$ viewed as a vector in $\mathbf{R}^{2}$
takes the form of \eqref{hamiltonian-separated} where
\[
B=I,\quad L=-\Delta+\omega-(p+1)\phi_{\omega}^{p},\quad A=-\Delta+\omega
-\phi_{\omega}^{p},
\]
on the energy space $H^{1}$ in the subcritical case and $\dot{H}^{1}$ in the
critical case. Clearly $\phi_{\omega}>0$ spans $\ker A$ and thus $A\geq0$.
Viewing $L$ and $A$ as perturbations to $-\Delta+\omega$, a simple argument
based on the compactness shows (\textbf{G1-4}) are satisfied.

Equation (\ref{hamiltonian-separated}) is of the Hamiltonian form
\begin{equation}
\partial_{t}w=\mathbf{JL}w, \label{hamiltonian-general}%
\end{equation}
where $\mathbf{u}=\left(  u,v\right)  \in\mathbf{X}=X\times Y$. Here, the
operators
\[
\mathbf{J}=\left(
\begin{array}
[c]{cc}%
0 & B\\
-B^{\prime} & 0
\end{array}
\right)  :\mathbf{X}^{\ast}\supset D(\mathbf{J})\rightarrow\mathbf{X},\ \ \
\]
and
\[
\mathbf{L}=\left(
\begin{array}
[c]{cc}%
L & 0\\
0 & A
\end{array}
\right)  :\mathbf{X}\rightarrow\mathbf{X}^{\ast}\text{.}%
\]
Under assumptions (\textbf{G1-4}), we can check that:

i) The operator $\mathbf{J\ }$is anti-self-dual, in the sense that
\[
D\left(  \mathbf{J}\right)  =D\left(  B^{\prime}\right)  \times D\left(
B\right)
\]
is dense in $\mathbf{X}^{\ast}$ and $\mathbf{J}^{\prime}=-\mathbf{J}$.

ii) The operator $\mathbf{L\ }$is bounded and self-dual (i.e. $\mathbf{L}%
^{\prime}=\mathbf{L}$) such that $\left\langle \mathbf{Lu,v}\right\rangle $ is
a bounded symmetric bilinear form on $\mathbf{X}$. For any $\mathbf{u}=\left(
u,v\right)  \in\mathbf{X}$, note that
\[
\left\langle \mathbf{\mathbf{L}u,u}\right\rangle =\left\langle
Lu,u\right\rangle +\left\langle Av,v\right\rangle , \quad\ker\mathbf{L=}\ker
L\times\ker A.
\]
Let
\begin{equation}
\mathbf{X}_{-}=X_{-}\times\left\{  0\right\}  ,\quad\mathbf{X}_{+}=X_{+}\times
Y_{+}, \label{defn-X-+-}%
\end{equation}
where $X_{\pm}$ and $Y_{+}$ are as in (\textbf{G2}) and (\textbf{G3}). Then we
have the decomposition
\[
\mathbf{X=X}_{-}\oplus\ker\mathbf{L}\oplus\mathbf{X}_{+},\quad\dim
\mathbf{X}_{-}=n^{-}(\mathbf{L})=n^{-}(L),
\]
satisfying: $\left\langle \mathbf{Lu,u}\right\rangle <0$ for all
$\mathbf{u}\in\mathbf{X}_{-}\backslash\{0\}$ and there exists $\delta_{0}>0$
such that
\[
\left\langle \mathbf{\mathbf{L}u,u}\right\rangle \geq\delta_{0}\left\Vert
\mathbf{u}\right\Vert ^{2}\ =\delta_{0}\left(  \left\Vert u\right\Vert
_{X}^{2}+\left\Vert v\right\Vert _{Y}^{2}\right)  ,\text{ for any
}\mathbf{u\in X}_{+}.
\]

iii) Assumption (\textbf{G4}) implies
\begin{align*}
&  \ker(i_{\mathbf{X}_{+}\mathbf{\oplus X}_{-}})^{\prime}=\{\mathbf{f}%
\in\mathbf{X}^{\ast}\mid\langle\mathbf{f},\mathbf{u}\rangle=0,\,\forall
\mathbf{u\in X}_{-}\mathbf{\oplus X}_{+}\}\\
=  &  \ker(i_{X_{+}\oplus X_{-}})^{\prime}\times\ker(i_{Y_{+}})^{\prime
}\subset D(\mathbf{J}).
\end{align*}
Therefore, $\left(  \mathbf{X},\mathbf{J},\mathbf{L}\right)  $ satisfies the
assumptions (\textbf{H1-3}) in \cite{lin-zeng-hamiltonian} and we can apply
the general theory for linear Hamiltonian PDE \cite{lin-zeng-hamiltonian} to
study the solutions of (\ref{hamiltonian-separated}). In particular, the
semigroup $e^{t\mathbf{J}\mathbf{L}}$ is well-defined. Corollary 12.1 in
\cite{lin-zeng-hamiltonian} also implies
\begin{equation}
\mathbf{L}\mathbf{J}=(\mathbf{J}\mathbf{L})^{\prime},\ BA,\ (BA)^{\prime
}=AB^{\prime},\ B^{\prime}L,\ (B^{\prime}L)^{\prime}=LB\text{ densely defined,
closed}. \label{E:closed-1}%
\end{equation}
Moreover, by using the separable nature of (\ref{hamiltonian-separated}), we
obtain more precise estimates on the instability index and the growth in the
center space. Our main Theorem for (\ref{hamiltonian-separated}) is the
following, whose proof would be self-contained except a few technical lemmas
in \cite{lin-zeng-hamiltonian} are cited. We adopt the same notations as in
\cite{lin-zeng-hamiltonian}. In particular, for a closed subspace
$X_{1}\subset X$,
we denote
\begin{equation}
L_{X_{1}}=i_{X_{1}}^{\prime}Li_{X_{1}}:X_{1}\rightarrow X_{1}^{\ast
}\Longrightarrow\langle L_{X_{1}}u_{1},u_{2}\rangle=\langle Lu_{1}%
,u_{2}\rangle,\;\forall u_{1},u_{2}\in X_{1}. \label{E:L_X_1}%
\end{equation}

\begin{theorem}
\label{T:main} Assume (\textbf{G1-4}) for (\ref{hamiltonian-separated}). The
operator $\mathbf{JL}$ generates a $C^{0}$ group $e^{t\mathbf{JL}}$ of bounded
linear operators on $\mathbf{X}$ and there exists a decomposition%
\[
\mathbf{X}=E^{u}\oplus E^{c}\oplus E^{s},\quad
\]
of closed subspaces $E^{u, s, c}$ with the following properties:

i) $E^{c},E^{u},E^{s}$ are invariant under $e^{t\mathbf{JL}}$.

ii) $E^{u}\left(  E^{s}\right)  $ only consists of eigenvectors corresponding
to negative (positive) eigenvalues of $\mathbf{JL}$ and
\begin{equation}
\dim E^{u}=\dim E^{s}=n^{-}\left(  L|_{\overline{R\left(  BA\right)  }%
}\right)  , \label{unstable-dimension-formula}%
\end{equation}
where $n^{-}\left(  L|_{\overline{R\left(  BA\right)  }}\right)  $ denotes the
number of negative modes of
$L|_{\overline{R\left(  BA\right)  }}$ as defined in \eqref{decom-X}. If
$n^{-}\left(  L|_{\overline{R\left(  BA\right)  }}\right)  >0$,
then there exists $M>0$ such that
\begin{equation}%
\begin{split}
&  \left\vert e^{t\mathbf{JL}}|_{E^{s}}\right\vert \leq Me^{-\lambda_{u}t},\;
t\geq0; \quad\left\vert e^{t\mathbf{JL}}|_{E^{u}}\right\vert \leq
Me^{\lambda_{u}t},\; t\leq0,
\end{split}
\label{estimate-stable-unstable}%
\end{equation}
where $\lambda_{u}=\min\{\lambda\mid\lambda\in\sigma(\mathbf{JL}|_{E^{u}%
})\}>0$.

iii) The quadratic form $\left\langle \mathbf{L}\cdot,\cdot\right\rangle
$ vanishes on $E^{u,s}$, i.e. $\langle\mathbf{L}\mathbf{u}, \mathbf{u}%
\rangle=0$ for all $\mathbf{u} \in E^{u, s}$, but is non-degenerate on
$E^{u}\oplus E^{s}$, and
\begin{equation}
E^{c}=\left\{  \mathbf{u}\in\mathbf{X}\mid\left\langle \mathbf{\mathbf{L}%
u,v}\right\rangle =0,\ \forall\ \mathbf{v}\in E^{s}\oplus E^{u}\right\}  .
\label{defn-center space}%
\end{equation}

iv)
There exist closed subspaces $\mathbf{X}_{j}$, $j=0, \ldots, 5$ such that
\[
E^{c} =\ker L \oplus\ker A \oplus(\oplus_{j=1}^{5} \mathbf{X}_{j}),
\quad\dim\mathbf{X}_{1} = \dim\mathbf{X}_{5} \le n^{-} (L) - \dim E^{u},
\]
\[
\mathbf{X}_{1}, \mathbf{X}_{4}, \mathbf{X}_{5} \subset X \times\{0\},
\quad\mathbf{X}_{2} \subset\{0\} \times Y.
\]
In this decomposition, $\mathbf{J}\mathbf{L}|_{E^{c}}$ and the quadratic form
$\mathbf{L}_{E^{c}}$ take the block form
\[
\mathbf{L}_{E^{c}} \longleftrightarrow%
\begin{pmatrix}
0 & 0 & 0 & 0 & 0 & 0 & 0\\
0 & 0 & 0 & 0 & 0 & 0 & 0\\
0 & 0 & 0 & 0 & 0 & 0 & \mathbf{L}_{15}\\
0 & 0 & 0 & \mathbf{L}_{2} & 0 & 0 & 0\\
0 & 0 & 0 & 0 & \mathbf{L}_{3} & 0 & 0\\
0 & 0 & 0 & 0 & 0 & \mathbf{L}_{4} & 0\\
0 & 0 & \mathbf{L}_{51} & 0 & 0 & 0 & 0\\
&  &  &  &  &  &
\end{pmatrix}
,
\]
\[
\mathbf{J}\mathbf{L}|_{E^{c}} \longleftrightarrow%
\begin{pmatrix}
0 & 0 & 0 & T_{X2} & T_{X3} & 0 & 0\\
0 & 0 & T_{Y1} & 0 & T_{Y3} & T_{Y4} & T_{Y5}\\
0 & 0 & 0 & T_{12} & T_{13} & 0 & 0\\
0 & 0 & 0 & 0 & 0 & 0 & T_{25}\\
0 & 0 & 0 & 0 & T_{3} & 0 & T_{35}\\
0 & 0 & 0 & 0 & 0 & 0 & 0\\
0 & 0 & 0 & 0 & 0 & 0 & 0
\end{pmatrix}
.
\]
All the non-trivial blocks of $\mathbf{L}_{E^{c}}$ are non-degenerate and
\[
\mathbf{L}_{2} \ge\epsilon, \quad\mathbf{L}_{3} \ge\epsilon,
\]
for some $\epsilon>0$. All the blocks of $\mathbf{J} \mathbf{L}$ are bounded
except $T_{3}$ is anti-self-adjoint with respect to the equivalent inner
product $\langle\mathbf{L}_{3} \cdot, \cdot\rangle$ satisfying $\ker T_{3} =
\{0\}$. Consequently, there exists $M>0$ such that
\begin{equation}
\left\vert e^{t\mathbf{JL}}|_{E^{c}}\right\vert \leq M(1+|t|)^{3}%
,\ t\in\mathbb{R}. \label{estimate-center}%
\end{equation}

v) Denote $Z$ to be the space $D(BA)$ with the graph norm
\[
\left\Vert y\right\Vert _{Z}=\left\Vert y\right\Vert _{Y}+\left\Vert
BAy\right\Vert _{X}.
\]
If the embedding $Z\hookrightarrow Y$ is compact, then the spectra of $T_{3}$
are nonzero, isolated with finite multiplicity, and have no accumulating point
except for $+\infty$. Moreover, the eigenfunctions of $T_{3}$ form an
orthonormal basis of $\mathbf{X}_{3}$ with respect to $\langle\mathbf{L}_{3}
\cdot, \cdot\rangle$. Consequently the spectra $\sigma(\mathbf{J}%
\mathbf{L})\backslash\{0\}$ are isolated with finite multiplicity, and have no
accumulating point except for $+\infty$.
\end{theorem}

\begin{remark}
\label{R:non-deg} Here the non-degeneracy of a bounded symmetric quadratic
form $B(u,v): Z \otimes Z \to\mathbf{R}$ on a real Banach space $Z$ is defined
as that the induced bounded linear operator $v \longrightarrow f= B(\cdot, v)
\in Z^{*}$ is an isomorphism from $X$ to $X^{*}$.
\end{remark}

The above theorem implies that the solutions of (\ref{hamiltonian-separated})
are spectrally stable (i.e. nonexistence of exponentially growing solution) if
and only if $L|_{\overline{R\left(  BA\right)  }}\geq0$. Moreover,
$n^{-}\left(  L|_{\overline{R\left(  BA\right)  }}\right)  $ gives the
dimension of the subspaces of exponentially growing solutions. The exponential
trichotomy estimates (\ref{estimate-stable-unstable})-(\ref{estimate-center})
are important in the study of nonlinear dynamics near an unstable steady
state, for which the linearized equation (\ref{hamiltonian-separated}) is
derived. If the spaces $E^{u,s}$ have higher regularity, then the exponential
trichotomy can be lifted to more regular spaces. We refer to Theorem 2.2 in
\cite{lin-zeng-hamiltonian} for more precise statements.

Compared to \cite{lin-zeng-hamiltonian}, the separable Hamiltonian form of
\eqref{hamiltonian-separated} yields a simpler block form. In particular, the
anti-self-adjointness of $T_{3}$ ensures the semi-simplicity of any eigenvalue
$\lambda\in i\mathbf{R}\backslash\{0\}$ and the non-degeneracy of $\mathbf{L}$
restricted to its subspace of generalized eigenvectors. This is not true for
general linear Hamiltonian systems, see examples in
\cite{lin-zeng-hamiltonian}. The separable Hamiltonian form also implies the
order $O(|t|^{3})$ of the growth in the center direction which is better than
the general cases in \cite{lin-zeng-hamiltonian}. These properties hold
essentially due to the second order equation \eqref{2nd order eqn-v} satisfied
by $v$.

\begin{remark}
As the only nontrivial block $T_{3}$ in the block decomposition of
$\mathbf{J}\mathbf{L}$ is anti-self-adjoint with respect to an equivalent
norm, it is clear that all the possible algebraic growth of $e^{t\mathbf{J}%
\mathbf{L}}$ must be associated to the possible zero eigenvalue. The
second order equation \eqref{2nd order eqn-v} allows at most $O(|t|)$ growth as in the case of wave equations. So it is natural to guess that the solutions of the first order system \eqref{hamiltonian-separated} may also grow no faster than $O(|t|)$.
However, the possible degeneracy of $B$ and $A$ indeed creates more growth and
the above $O(|t|^{3})$ growth is optimal. Consider the following example:
\[
X = \mathbf{R}^{2}, \;\; Y= \mathbf{R}^{3}, \;\; A=%
\begin{pmatrix}
0 & 0 & 0\\
0 & 1 & 0\\
0 & 0 & 1
\end{pmatrix}
, \;\; L =
\begin{pmatrix}
2 & -1\\
-1 & 0
\end{pmatrix}
, \;\; B =%
\begin{pmatrix}
1 & 1 & 0\\
0 & 1 & 0
\end{pmatrix}
.
\]
One may compute
\[
\mathbf{J} \mathbf{L}=%
\begin{pmatrix}
0 & 0 & 0 & 1 & 0\\
0 & 0 & 0 & 1 & 0\\
-2 & 1 & 0 & 0 & 0\\
-1 & 1 & 0 & 0 & 0\\
0 & 0 & 0 & 0 & 0
\end{pmatrix}
, \quad(\mathbf{J} \mathbf{L})^{2}=%
\begin{pmatrix}
-1 & 1 & 0 & 0 & 0\\
-1 & 1 & 0 & 0 & 0\\
0 & 0 & 0 & -1 & 0\\
0 & 0 & 0 & 0 & 0\\
0 & 0 & 0 & 0 & 0
\end{pmatrix}
,
\]
\[
(\mathbf{J} \mathbf{L})^{3}=%
\begin{pmatrix}
0 & 0 & 0 & 0 & 0\\
0 & 0 & 0 & 0 & 0\\
1 & -1 & 0 & 0 & 0\\
0 & 0 & 0 & 0 & 0\\
0 & 0 & 0 & 0 & 0
\end{pmatrix}
, \quad(\mathbf{J} \mathbf{L})^{4}=0.
\]
Therefore $e^{t\mathbf{J}\mathbf{L}}$ exhibits $O(|t|^{3})$ growth.
\end{remark}

In the following theorem, we given some conditions on $(L, A, B)$ which yields
better growth estimate of $e^{t\mathbf{J}\mathbf{L}}$ on the center subspace
$E^{c}$.

\begin{theorem}
\label{T:main2} Assume (\textbf{G1-4}) for (\ref{hamiltonian-separated}). The
following hold under additional assumptions:

i) If $A$ is injective on $\overline{R(B^{\prime}LBA)}$, then $|e^{t\mathbf{J}%
\mathbf{L}}|_{E^{c}} | \le M (1 + t^{2})$ for some $M>0$.

ii) If $\overline{R(BA)}= X$, then $|e^{t\mathbf{J}\mathbf{L}}|_{E^{c}} | \le
M (1 + |t|)$ for some $M>0$.

iii) Suppose $\left\langle L\cdot,\cdot\right\rangle $ and $\langle
A\cdot,\cdot\rangle$ are non-degenerate on $\overline{R\left(  B\right)  }$
and $\overline{R\left(  B^{\prime}\right)  }$, respectively, then
$|e^{t\mathbf{JL}}|_{E^{c}}|\leq M$ for some $M>0$. Namely, there is Lyapunov
stability on the center space $E^{c}$.
\end{theorem}

\begin{remark}
Motivated by the second order equation \eqref{2nd order eqn-v}, when
$L|_{\overline{R\left(  BA\right)  }}$ has a negative mode, it is tempting to
find the most unstable eigenvalue $\lambda_{0}>0\ $of
(\ref{hamiltonian-separated}) satisfying $B^{\prime}LBAv=-\lambda_{0}^{2}v$ by
solving the variational problem
\begin{equation}
-\lambda_{0}^{2}=\min_{\left\langle Av,v\right\rangle =1,v\in D\left(
A\right)  }\left\langle B^{\prime}LBAv,Av\right\rangle .\label{variational}%
\end{equation}
However, in many applications particularly to kinetic models such as
Vlasov-Maxwell and Vlasov-Einstein systems, it is difficult to solve the
variational problem (\ref{variational}) directly due to the lack of
compactness. In Theorem \ref{T:main}, the existence of unstable eigenvalues
follows from the self-adjointness of the operator $B^{\prime}LBA$ and the
assumption $n^{-}\left(  L\right)  <\infty$.
\end{remark}

The proof of Theorem \ref{T:main} will be split into several lemmas and
propositions. We start with a general functional analysis result which might
be of independent interest.

\begin{proposition}
\label{prop-self-adjoint -1} Let $X,Y$ be Hilbert spaces, $L:X\rightarrow X$
is a bounded self-adjoint operator, and $A\colon Y\supset D(A)\rightarrow X$
is a densely defined and closed operator. In addition, assume that:

1). The adjoint operator $A^{\ast}:X\supset D(A^{\ast})\rightarrow Y$ is also
densely defined.

2). $\exists$ $\delta>0$ and a closed subspace $X_{+}\subset X$ such that
$(Lx,x)\geq\delta\left\Vert x\right\Vert ^{2}$, $\forall x\in X_{+}$ and
$X_{+}^{\perp}\subset D(A^{\ast})$.

Then: i) the operator $A^{\ast}LA$ is self-adjoint on $Y$ with domain
$D\left(  A^{\ast}LA\right)  \subset D\left(  A\right)  $.

ii) Denote $Z$ to be the space $D\left(  A\right)  $ equipped with the graph
norm
\[
\left\Vert y\right\Vert _{Z}=\left\Vert y\right\Vert _{Y}+\left\Vert
Ay\right\Vert _{X}.
\]
If the embedding $Z\hookrightarrow Y$ is compact, then the spectra of
$A^{\ast}LA$ are purely discrete, and have no accumulating point except for
$+\infty$. Moreover, the eigenfunctions of $A^{\ast}LA$ form a basis of $Y$.
\end{proposition}

\begin{proof}
Let
\[
X_{1}=\{x\in X\mid\langle Lx,x^{\prime}\rangle=0,\ \forall x^{\prime}\in
X_{+}\}.
\]
The uniform positivity of $L$ on $X_{+}$ and Lemma 12.2 in
\cite{lin-zeng-hamiltonian} imply
\[
X=X_{+}\oplus X_{1},\quad P_{1}^{\ast}LP_{+}=P_{+}^{\ast}LP_{1}=0,
\]
where $P_{+,1}$ are the associated projections. Therefore,
\[
L=P_{+}^{\ast}LP_{+}+P_{1}^{\ast}LP_{1}\triangleq L_{+}-L_{1}%
\]
with symmetric bounded $L_{+,1}$ and $L_{+}\geq0$. Since $R(P_{1}^{\ast
})=X_{+}^{\perp}\subset D(A^{\ast})$, the Closed Graph Theorem implies that
$A^{\ast}P_{1}^{\ast}$ is bounded. Therefore, $P_{1}A$ has a continuous
extension $(A^{\ast}P_{1}^{\ast})^{\ast}=(P_{1}A)^{\ast\ast}$, i.e. $P_{1}A$
is bounded. Thus $P_{+}A$ is closed and densely defined. Let $S_{+}%
:X\rightarrow X$ be a bounded symmetric linear operator such that
\[
S_{+}^{\ast}S_{+}=S_{+}^{2}=L_{+},\quad S_{+}\geq0.
\]
Moreover, for any $x\in X_{+}$,
\[
\left\Vert S_{+}x\right\Vert _{X}^{2}=(L_{+}x,x)=(Lx,x)\geq\delta\left\Vert
x\right\Vert _{X}^{2},
\]
which implies that%
\begin{equation}
\left\Vert S_{+}x\right\Vert _{X}\geq\sqrt{\delta}\left\Vert x\right\Vert
_{X},\;\forall x\in X_{+}. \label{positivity-S+}%
\end{equation}
This lower bound of $S_{+}$ implies that $T_{+}\triangleq S_{+}P_{+}A$ is also
closed with the dense domain $D(T_{+})=D(A)$ and thus $T_{+}^{\ast}T_{+}$ is
self-adjoint. We note that
\begin{align}
A^{\ast}LA=  &  A^{\ast}P_{+}^{\ast}L_{+}P_{+}A-A^{\ast}P_{1}^{\ast}L_{1}%
P_{1}A\label{decomp-A*LA}\\
=  &  (A^{\ast}P_{+}^{\ast}S_{+})(S_{+}P_{+}A)-A^{\ast}P_{1}^{\ast}L_{1}%
P_{1}A\triangleq T_{+}^{\ast}T_{+}-B_{1}.\nonumber
\end{align}
Here, $B_{1}$ is bounded and symmetric. Therefore, by Kato-Rellich Theorem
$A^{\ast}LA$ is self-adjoint with
\[
D\left(  A^{\ast}LA\right)  \subset D(T_{+})=D(A).
\]

By Theorem 4.2.9 in \cite{edmunds-evans}, to prove conclusions in ii), it
suffices to show that the embedding $Z_{1}\hookrightarrow Y$ is compact. Here,
the space $Z_{1}$ is $D(T_{+})=D\left(  A\right)  $ with the graph norm
\[
\left\Vert y\right\Vert _{Z_{1}}=\left\Vert y\right\Vert _{Y}+\left\Vert
T_{+}y\right\Vert _{X}.
\]
We show that $\left\Vert \cdot\right\Vert _{Z_{1}}$ and $\left\Vert
\cdot\right\Vert _{Z}$ are equivalent. Indeed, since $A$ and $T_{+}$ are
closed with the same domain, $A: Z_{1} \to X$ and $T_{+}: Z \to X$ are also
apparently closed and thus bounded, which immediately implies the equivalence
of $\left\Vert \cdot\right\Vert _{Z_{1}}$ and $\left\Vert \cdot\right\Vert
_{Z}$.

\end{proof}

In the above Proposition, we can allow $n^{-}\left(  L\right)  =\infty$, but
the condition $X_{+}^{\perp}\subset D(A^{\ast})$ need to be verified. The next
lemma shows that this condition is implied by our assumptions
(\textbf{G1-4}).

\begin{lemma}
\label{lemma-self-adjoint-2} Suppose the operators $L,B,$ $B^{\prime},A$
satisfy assumptions (\textbf{G1-4}). Then:

i) $\tilde L=AB^{\prime}LBA: Y \supset D(\tilde L) \to Y^{*}$ and $\tilde A=
LBAB^{\prime}L: X \supset D(\tilde A) \to X^{*}$ are self-dual, namely $\tilde
L^{\prime}=\tilde L$ and $\tilde A^{\prime}= \tilde A$.

ii) In addition to (\textbf{G1-4}), assume $\ker A=\{0\}$, then
$\mathbb{\tilde{L}}=B^{\prime}LBA$ is self-adjoint on $(Y,\left[  \cdot
,\cdot\right]  )$ with the equivalent inner product $\left[  \cdot
,\cdot\right]  =\left\langle A\cdot,\cdot\right\rangle $.

iii) Denote $Z$ to be the space $D\left(  BA\right)  $ with the graph norm
\[
\left\Vert y\right\Vert _{Z}=\left\Vert y\right\Vert _{Y}+\left\Vert
BAy\right\Vert _{X}.
\]
If the embedding $Z\hookrightarrow Y$ is compact and $\ker A =\{0\}$, then the
spectra of$\ \mathbb{\tilde{L}}$ are purely discrete with finite multiplicity,
and have no accumulating point except for $+\infty$. Moreover, the
eigenfunctions of $\mathbb{\tilde{L}}$ form a basis of $Y$.
\end{lemma}

\begin{proof}
Recall that $I_{X}:X^{\ast}\ \rightarrow X,$ $I_{Y}:Y^{\ast}\rightarrow
Y\ $are isomorphisms defined by the Riesz representation theorem. Define the
operators
\[
\mathbb{A=}BA:Y\supset D(\mathbb{A})\rightarrow X,\ \ \ L_{1}=I_{X}%
L:X\rightarrow X.
\]
The adjoint operators are
\[
L_{1}^{\ast}=L_{1},\quad\mathbb{A}^{\ast}=I_{Y}AB^{\prime}I_{X}^{-1}.
\]
According to \eqref{E:closed-1}, $\mathbb{A}^{\ast}$ is densely defined and closed.

Since $(X_{+}\oplus X_{-})^{\perp}\subset X_{+}^{\perp}$ is a closed subspace
of codimension equal to $\dim X_{-}<\infty$, we have
\[
\dim W
=\dim X_{-}<\infty,
\]
where
\[
W=X_{+}^{\perp}\cap(X_{+}\oplus X_{-}),\quad X_{+}^{\perp}=W\oplus(X_{+}\oplus
X_{-})^{\perp}.
\]
Assumption (\textbf{G4}) implies that $D(\mathbb{A}^{\ast})\cap(X_{+}\oplus
X_{-})$ is dense in $X_{+}\oplus X_{-}$. Approximate $W$ by $\tilde{W}\subset
D(\mathbb{A}^{\ast})\cap(X_{+}\oplus X_{-})$ such that $\dim W=\dim\tilde{W}$,
which is possible since $\dim W<\infty$. Let
\[
\tilde{X}_{+}=\{x\in X_{+}\oplus X_{-}\mid(x,y)=0,\ \forall y\in\tilde{W}\}.
\]
The quadratic form $
\left\langle L\cdot,\cdot\right\rangle $ is
uniformly positive definite on the approximation $\tilde{X}_{+}$ of $X_{+}$
and
\[
\tilde{X}_{+}^{\perp}=(X_{+}\oplus X_{-})^{\perp}\oplus\tilde{W}\subset
D(\mathbb{A}^{\ast}).
\]
Therefore, all conditions in Proposition \ref{prop-self-adjoint -1} are
satisfied by $\tilde{X}_{+}$, $L_{1}$, and $\mathbb{A}$ and thus
$\mathbb{A}^{\ast}L_{1}\mathbb{A}=I_{Y}AB^{\prime}LBA$ are self-adjoint. This
implies that $\tilde{L}=AB^{\prime}LBA$ satisfies $\tilde{L}^{\prime}%
=\tilde{L}$. It follows from the same argument that $\tilde{A}^{\prime}%
=\tilde{A}$.

Statement ii) and iii) are direct corollaries of i) and
Proposition \ref{prop-self-adjoint -1}.
\end{proof}

We shall start the proof of Theorem \ref{T:main} with several steps of
decomposition of $X$ and $Y$.

\begin{lemma}
\label{L:decom-1} Assume (\textbf{G1-4}). Suppose $X_{1,2}$ are closed
subspaces of $X$ and $Y_{1,2}$ are closed subspaces of $Y$ such that
$X=X_{1}\oplus X_{2}$, $Y=Y_{1}\oplus Y_{2}$. Let $P_{1,2}:X\rightarrow
X_{12,}$ and $Q_{1,2}:Y\rightarrow Y_{1,2}$ be the associated projections and,
for $j,k=1,2$,
\[
L_{j}=(i_{X_{j}})^{\prime}Li_{X_{j}},\quad A_{j}=(i_{Y_{j}})^{\prime}%
Ai_{Y_{j}},\quad B_{jk}=P_{j}BQ_{k}^{\prime jk}=Q_{j}B^{\prime}P_{k}^{\prime
},
\]%
\[
\mathbf{X}_{1}=X_{1}\times Y_{1},\quad\mathbf{L}_{1}=%
\begin{pmatrix}
L_{1} & 0\\
0 & A_{1}%
\end{pmatrix}
,\quad\mathbf{J}_{1}=%
\begin{pmatrix}
0 & B_{11}\\
-B^{11} & 0
\end{pmatrix}
,
\]%
\[
\mathbf{X}_{2}=X_{2}\times Y_{2},\quad\mathbf{L}_{2}=%
\begin{pmatrix}
L_{2} & 0\\
0 & A_{2}%
\end{pmatrix}
,\quad\mathbf{J}_{2}=%
\begin{pmatrix}
0 & B_{22}\\
-B^{22} & 0
\end{pmatrix}
.
\]
In addition, we assume
\[
\langle Lx_{1},x_{2}\rangle=0,\;\forall x_{1}\in X_{1},\,x_{2}\in X_{2}%
;\quad\langle Ay_{1},y_{2}\rangle=0,\;\forall y_{1}\in Y_{1},\,y_{2}\in
Y_{2};
\]

\[
P_{1}^{\prime}(X_{1}^{\ast})\subset D(B^{\prime}),\quad\;Q_{1}^{\prime}%
(Y_{1}^{\ast})\subset D(B).
\]

Then we have

\begin{enumerate}
\item In this decomposition, $\mathbf{J}\mathbf{L}$ takes the form
\[
\mathbf{J}\mathbf{L}\longleftrightarrow%
\begin{pmatrix}
\mathbf{J}_{1}\mathbf{L}_{1} & \mathbf{T}_{12}\\
\mathbf{T}_{21} & \mathbf{J}_{2}\mathbf{L}_{2}%
\end{pmatrix}
,
\]
where
\[
\mathbf{T}_{12}=%
\begin{pmatrix}
0 & B_{12}A_{2}\\
-B^{12}L_{2} & 0
\end{pmatrix}
,\quad\mathbf{T}_{21}=%
\begin{pmatrix}
0 & B_{21}A_{1}\\
-B^{21}L_{1} & 0
\end{pmatrix}
.
\]

\item We have that $B_{22}$ and $B^{22}$ are densely defined closed operators
and $B_{jk}$ and $B^{jk}$, $(j,k)\neq(2,2)$, and thus $\mathbf{T}_{12}$,
$\mathbf{T}_{21}$, and $\mathbf{J}_{1}\mathbf{L}_{1}$ are all bounded. Here,
we abuse the notations slightly in using $B_{jk}$ and $B^{jk}$, for
$(j,k)\neq(2,2)$, to also denote their continuous extensions.

\item $B^{jk} = B_{kj}^{\prime}$ for all $j, k=1,2$.

\item $(L_{1}, A_{1}, B_{11})$ and $(L_{2}, A_{2}, B_{22})$ satisfy
(\textbf{G1-4}) and
\[
n^{-}(L) = n^{-}(L_{1}) + n^{-}(L_{2}), \; \ker L = \ker L_{1} \oplus\ker
L_{2}, \; \ker A = \ker A_{1} \oplus\ker A_{2}.
\]

\end{enumerate}
\end{lemma}

\begin{proof}
The assumptions $P_{1}^{\prime}(X_{1}^{\ast})\subset D(B^{\prime})$ and
$Q_{1}^{\prime}(Y_{1}^{\ast})\subset D(B)$ imply that $B^{\prime}P_{1}%
^{\prime}$ and $BQ_{1}^{\prime}$ are closed operators defined on Hilbert
spaces $X_{1}^{\ast}$ and $Y_{1}^{\ast}$. The Closed Graph Theorem yields that
$B^{\prime}P_{1}^{\prime}$ and $BQ_{1}^{\prime}$ are bounded operators.
Therefore, $P_{1}B$ and $Q_{1}B^{\prime}$ are also both bounded as they have
continuous extensions $(P_{1}B)^{\prime\prime}=(B^{\prime}P_{1}^{\prime
})^{\prime}$ and $(Q_{1}B^{\prime})^{\prime\prime}=(BQ_{1}^{\prime})^{\prime}%
$. Consequently the second statement, as well as the closedness of $P_{2}B$
and $B^{\prime}P_{2}^{\prime}$ with dense domains, follows.

For $(j,k)\neq(2,2)$, it is easy to verify $B^{jk}=B_{jk}^{\prime}$ as they
are compositions of bounded operators. To show $B^{22}=B_{22}$, we notice that
the closedness and the density of the domains of $P_{2}B$ and $B^{\prime}%
P_{2}^{\prime}=(P_{2}B)^{\prime}$ imply
\begin{align*}
P_{2}B  &  =(P_{2}B)^{\prime\prime}=(B^{\prime}P_{2}^{\prime})^{\prime},\ \\
\left(  B^{22}\right)  ^{\prime}  &  =(Q_{2}B^{\prime}P_{2}^{\prime})^{\prime
}=(B^{\prime}P_{2}^{\prime})^{\prime}Q_{2}^{\prime}=P_{2}BQ_{2}^{\prime
}=B_{22}.
\end{align*}
The closedness of $B_{22}$ and $B^{22}$ again yields $B^{22}=(B^{22}%
)^{\prime\prime}=B_{22}^{\prime}$. It completes the proof of the third statement.

The $L$-orthogonality of the splitting $X_{1}\oplus X_{2}$ and the
$A$-orthogonality of $Y=Y_{1}\oplus Y_{2}$ yield block diagonal forms of $L$
and $A$ in these splittings. The block form of $\mathbf{J}\mathbf{L}$ follows
from straightforward calculations.

It has been proved in the above that $B_{11}$ and $B_{22}$ satisfy
(\textbf{G1}), while (\textbf{G2}) for $A_{1}$ and $A_{2}$ and (\textbf{G3})
for $L_{1}$ and $L_{2}$ are proved in Lemma 12.3 in
\cite{lin-zeng-hamiltonian}. Apparently (\textbf{G4}) is satisfied by
$(L_{1},A_{1},B_{11})$ as $B_{11}$ is a bounded operator. Finally,
(\textbf{G4}) for $(L_{2},A_{2},B_{22})$ also follows directly from the proof
of Lemma 12.3 in \cite{lin-zeng-hamiltonian}.
\end{proof}

\begin{remark}
Even though the framework in \cite{lin-zeng-hamiltonian} is slightly
different, those properties of $J$ and $L$ used in the proof of Lemma 12.3
therein are all satisfied by $L_{2}$, $A_{2}$, and $B_{22}$ here. Therefore,
the same proof works to show that (\textbf{G4}) is satisfied by $(L_{2}%
,A_{2},B_{22})$.
\end{remark}

The following three lemmas focus on decomposing a subspace of the center subspace.

\begin{lemma}
\label{L:decom-X} Assume (\textbf{G1--3}) and that $L$ is non-degenerate (in
the sense of Remark \ref{R:non-deg}). Let $\tilde{X}\subset X$ be a closed
subspace such that $\ker(i_{\tilde{X}})^{\prime}\subset D(B^{\prime})$, then
there exist closed subspaces $X_{j}$, $j=1,2,3,4$, such that
\begin{align*}
&  \tilde{X}=X_{1}\oplus X_{2},\;\;\tilde{X}^{\perp_{L}}\triangleq\{x\in
X\mid\langle Lx,\tilde{x}\rangle=0,\,\forall\tilde{x}\in\tilde{X}%
\}=X_{1}\oplus X_{3},\\
&  X=\oplus_{j=1}^{4}X_{j},\quad n_{1}\triangleq\dim X_{1}=\dim X_{4}\leq
n^{-}(L).
\end{align*}
Moreover, let $P_{j}$, $j=1,2,3,4$, be the associated projections and it
holds
\[
P_{1}^{\prime}(X_{1}^{\ast})\oplus P_{3}^{\prime}(X_{3}^{\ast})\oplus
P_{4}^{\prime}(X_{4}^{\ast})=\ker(i_{X_{2}})^{\prime}\subset D(B^{\prime}).
\]
In this decomposition, the quadratic form $L$ takes the block form
\[
L\longleftrightarrow%
\begin{pmatrix}
0 & 0 & 0 & L_{14}\\
0 & L_{2} & 0 & 0\\
0 & 0 & L_{3} & 0\\
L_{41} & 0 & 0 & 0
\end{pmatrix}
,\quad L_{jk}=(i_{X_{j}}^{\prime})Li_{X_{k}}:X_{k}\rightarrow X_{j}^{\ast
},\quad L_{j}=L_{jj},
\]
with $L_{14}=L_{41}^{\prime}$, $L_{2}$, and $L_{3}$ all non-degenerate.

\end{lemma}

As stated in Remark \ref{remark-G4}, assumption (\textbf{G4}) holds for
non-degenerate $L$.

\begin{proof}
Let
\[
X_{1}=\tilde{X}\cap\tilde{X}^{\perp_{L}}=\left(  \tilde{X}+\tilde{X}%
^{\perp_{L}}\right)  ^{\perp_{L}},
\]
where the non-degeneracy of $L$ was used in the second equality. Since
$\langle Lx,x\rangle=0$ for all $x\in X_{1}\subset X$,
\[
n_{1}=\dim X_{1}=codim\left(  \tilde{X}+\tilde{X}^{\perp_{L}}\right)  \leq
n^{-}(L)
\]
is a direct consequence of the non-degeneracy assumption of $L$ and Theorem
5.1 in \cite{lin-zeng-hamiltonian}. According to the density of $D\left(
B^{\prime}\right)  $, there exist $f_{j}\in D(B^{\prime})$, $j=1,\ldots,n_{1}%
$, such that $(i_{X_{1}})^{\prime}f_{j}\in X_{1}^{\ast}$, $j=1,\ldots,n_{1}$,
form a basis of $X_{1}^{\ast}$. Let $x_{j}\in X_{1}$, $j=1,\ldots,n_{1}$, be
the basis of $X_{1}$ dual to $\{(i_{X_{1}})^{\prime}f_{j}\}_{j=1}^{n_{1}}$,
namely, $\langle f_{j},x_{k}\rangle=\delta_{jk}$. Let
\[
X_{4}=span\{L^{-1}f_{j}-\frac{1}{2}\sum_{k=1}^{n_{1}}\langle f_{j},L^{-1}%
f_{k}\rangle x_{k},\,j=1,\ldots,n_{1}\}.
\]
It is easy to verify that
\[
\dim X_{4}=n_{1}\quad\langle Lx,\tilde{x}\rangle=0,\,\forall x,\tilde{x}\in
X_{4},
\]
and $L_{14}=L_{41}^{\prime}$ is non-degenerate. Let
\[
X_{2}=\{x\in\tilde{X}\mid\langle f_{j},x\rangle=0,\,j=1,\ldots,n_{1}\},
\]
and
\[
X_{3}=\{x\in\tilde{X}^{\perp_{L}}\mid\langle f_{j},x\rangle=0,\,j=1,\ldots
,n_{1}\}.
\]
The direct sum relations and the block form of $L$ stated in the lemma follow
straightforwardly. The non-degeneracy of $L$ and the definitions of $X_{2}$
and $X_{3}$ imply that $L_{X_{2}}$ and $L_{X_{3}}$ (as defined in
\eqref{E:L_X_1}) are injective. Therefore, Lemma 12.2 in
\cite{lin-zeng-hamiltonian} yields the non-degeneracy of $L_{2}=L_{X_{2}}$ and
$L_{3}=L_{X_{3}}$.
Finally, observing%
\begin{equation}
L(X_{1})=P_{4}^{\prime}(X_{4}^{\ast})\subset P_{3}^{\prime}(X_{3}^{\ast
})\oplus P_{4}^{\prime}(X_{4}^{\ast})=\ker(i_{\tilde{X}})^{\prime}\subset
D(B^{\prime}) \label{E:ker-B'}%
\end{equation}
and
\[
P_{1}^{\prime}(X_{1}^{\ast})\oplus P_{4}^{\prime}(X_{4}^{\ast})=\ker
(i_{X_{2}\oplus X_{3}})^{\prime}=span\{f_{1},\ldots,f_{n_{1}}\}+L(X_{1}%
)\subset D(B^{\prime}),
\]
the proof of the lemma is complete.
\end{proof}


\begin{lemma}
\label{L:decom-Y} In addition to (\textbf{G1-4}), assume $\ker A=\{0\}$ and
$n^{-}(L|_{\overline{R\left(  B\right)  }})=0$, the latter of which implies
$L|_{\overline{R\left(  B\right)  }}\geq0$ and $A\geq\delta>0$.
Let $Y_{1}=\ker\tilde{\mathbb{L}}$ and
\[
Y_{2}=Y_{1}^{\perp_{A}}=\{y\in Y\mid\langle Ay,\tilde{y}\rangle=0,\,\forall
\tilde{y}\in Y_{1}\},
\]
where $\tilde{\mathbb{L}}=B^{\prime}LBA$ is defined as in Lemma
\ref{lemma-self-adjoint-2}. Then it holds
\[
Y_{1}=\ker(L_{\overline{R(B)}}BA),\ Y_{2}=\overline{R(\tilde{\mathbb{L}}%
)}=\overline{R(B^{\prime}L_{\overline{R(B)}})},\quad Y=Y_{1}\oplus Y_{2}.
\]
In this decomposition, the quadratic form $A$ takes the block form
\[
A\longleftrightarrow%
\begin{pmatrix}
A_{1} & 0\\
0 & A_{2}%
\end{pmatrix}
,\quad A_{j}=(i_{Y_{j}})^{\prime}Ai_{Y_{j}}:Y_{j}\rightarrow Y_{j}^{\ast}.
\]

\end{lemma}

Here $L_{\overline{R(B)}}: \overline{R(B)} \to\left(  \overline{R(B)}\right)
^{*}$ is defined as in \eqref{E:L_X_1}. In the following, we also view $B$ as
a closed operator from $Y^{*}$ to $\overline{R(B)}$.

\begin{proof}
Observing that $Y_{2}$ is defined as the orthogonal complement of $Y_{1}$ in
$Y_{A}$ and $\tilde{\mathbb{L}}$ is self-adjoint on $Y_{A}$, it follows
immediately that $Y_{2}=\overline{R(\tilde{\mathbb{L}})}$ and $Y=Y_{1}\oplus
Y_{2}$. We shall show the remaining alternative representations of $Y_{1}$ and
$Y_{2}$ in the rest of the proof.

On the one hand, since
\begin{equation}
B=i_{\overline{R(B)}}B\,\text{ and }\,B^{\prime}=B^{\prime}(i_{\overline
{R(B)}})^{\prime}\,\Longrightarrow\tilde{\mathbb{L}}=B^{\prime}L_{\overline
{R(B)}}BA, \label{E:BFL}%
\end{equation}
clearly $\ker(L_{\overline{R(B)}}BA)\subset Y_{1}$ according to their
definitions. On the other hand, each $y\in Y_{1}$ satisfies
\[
\langle L_{\overline{R(B)}}BAy,BAy\rangle=[\tilde{\mathbb{L}}y,y]=0.
\]
Due to the assumption $L|_{\overline{R\left(  B\right)  }}\geq0$, a standard
variational argument implies $L_{\overline{R(B)}}BAy=0$ and thus $y\in
\ker(L_{\overline{R(B)}}BA)$. We obtain $Y_{1}=\ker(L_{\overline{R(B)}}BA)$.

For any $x\in D(B^{\prime}L_{\overline{R(B)}})$ and $y\in Y_{1}=\ker
(L_{\overline{R(B)}}BA)$, we have
\[
\lbrack B^{\prime}L_{\overline{R(B)}}x,y]=\langle Ay,B^{\prime}L_{\overline
{R(B)}}x\rangle=\langle L_{\overline{R(B)}}BAy,x\rangle=0,
\]
which along with the closedness of $Y_{2}$, implies $\overline{R(B^{\prime
}L_{\overline{R(B)}})}\subset Y_{2}$. Obviously, $R(\tilde{\mathbb{L}})\subset
R(B^{\prime}L_{\overline{R(B)}})$ and thus $Y_{2}\subset\overline{R(B^{\prime
}L_{\overline{R(B)}})}$. Therefore, the equal sign holds and this completes
the proof of the lemma.
\end{proof}

Applying the above lemmas (Lemma \ref{L:decom-X} to $\tilde{X}=\overline
{R(B)}$), we obtain the following decomposition.

\begin{proposition}
\label{P:block-1} In addition to (\textbf{G1-4}), assume a.) $L$ is
non-degenerate, b.) $n^{-}(L_{\overline{R\left(  B\right)  }})=0$, and c.)
$A\geq\delta>0$. Let
\begin{equation}%
\begin{cases}
\mathbf{X}_{1}=X_{1}\times\{0\},\quad\mathbf{X}_{2}=\{0\}\times Y_{1}%
,\quad\mathbf{X}_{3}=X_{2}\times Y_{2},\\
\mathbf{X}_{4}=X_{3}\times\{0\},\quad\mathbf{X}_{5}=X_{4}\times\{0\}
\end{cases}
\label{E:splitting-1}%
\end{equation}
as defined in Lemmas \ref{L:decom-X} and \ref{L:decom-Y}. Then in the
decomposition $\mathbf{X}=\oplus_{j=1}^{5}\mathbf{X}_{j}$, $\mathbf{J}%
\mathbf{L}$ and the quadratic form $\mathbf{L}$ take the form
\[
\mathbf{L}\longleftrightarrow%
\begin{pmatrix}
0 & 0 & 0 & 0 & \mathbf{L}_{15}\\
0 & \mathbf{L}_{2} & 0 & 0 & 0\\
0 & 0 & \mathbf{L}_{3} & 0 & 0\\
0 & 0 & 0 & \mathbf{L}_{4} & 0\\
\mathbf{L}_{51} & 0 & 0 & 0 & 0
\end{pmatrix}
,\quad\mathbf{J}\mathbf{L}\longleftrightarrow%
\begin{pmatrix}
0 & T_{12} & T_{13} & 0 & 0\\
0 & 0 & 0 & 0 & T_{25}\\
0 & 0 & T_{3} & 0 & T_{35}\\
0 & 0 & 0 & 0 & 0\\
0 & 0 & 0 & 0 & 0
\end{pmatrix}
.
\]
All the non-trivial blocks of $\mathbf{L}$ are non-degenerate,
\[
\mathbf{L}_{15}=L_{14},\quad\mathbf{L}_{51}=L_{41},\quad\mathbf{L}_{2}%
=A_{1}\geq\delta,\quad\mathbf{L}_{3}=%
\begin{pmatrix}
L_{2} & 0\\
0 & A_{2}%
\end{pmatrix}
\geq\epsilon,\quad\mathbf{L}_{4}=L_{3},
\]
for some $\epsilon>0$. All the blocks of $\mathbf{J}\mathbf{L}$
\begin{align*}
&  T_{12}=P_{1}BQ_{1}^{\prime}A_{1},\quad T_{13}%
\begin{pmatrix}
x\\
y
\end{pmatrix}
=P_{1}BQ_{2}^{\prime}A_{2}y,\quad T_{25}=-Q_{1}B^{\prime}P_{1}^{\prime}%
L_{14},\\
&  T_{35}=%
\begin{pmatrix}
0\\
-Q_{2}B^{\prime}P_{1}^{\prime}L_{14}%
\end{pmatrix}
,\quad T_{3}=%
\begin{pmatrix}
0 & P_{2}BQ_{2}^{\prime}A_{2}\\
-Q_{2}B^{\prime}P_{2}^{\prime}L_{2} & 0
\end{pmatrix}
,\quad\ker T_{3}=\{0\},
\end{align*}
are bounded except $T_{3}$ is anti-self-adjoint with respect to the equivalent
inner product $\langle\mathbf{L}_{3}\cdot,\cdot\rangle$. Here $P_{1,2,3,4}$
and $Q_{1,2}$ are the projections associated to the decomposition of $X$ and
$Y$ given in Lemma \ref{L:decom-X} and \ref{L:decom-Y}. Finally, denote $Z$ to
be the space $D\left(  BA\right)  $ with the graph norm
\[
\left\Vert y\right\Vert _{Z}=\left\Vert y\right\Vert _{Y}+\left\Vert
BAy\right\Vert _{X}.
\]
If the embedding $Z\hookrightarrow Y$ is compact, then the spectra of $T_{3}$
are nonzero, isolated with finite multiplicity, and have no accumulating point
except for $+\infty$. Moreover, the eigenfunctions of $T_{3}$ form an
orthonormal basis of $\mathbf{X}_{3}$ with respect to $\langle\mathbf{L}%
_{3}\cdot,\cdot\rangle$.
\end{proposition}

\begin{remark}
\label{R:notation-1}
One should notice that $P_{1}$ in $T_{12}$ and $Q_{2}$ in the lower left entry
of $T_{3}$ are put there only to specify the target spaces, but do not change
any values.
\end{remark}

\begin{proof}
Since Lemma \ref{L:decom-X} and \ref{L:decom-Y} imply
\begin{equation}
P_{3}^{\prime}(X_{3}^{\ast})\oplus P_{4}^{\prime}(X_{4}^{\ast})=\ker
B^{\prime}\;\text{ and }\;X_{2}^{\ast}=R(L_{\overline{R\left(  B\right)  }%
})\Longrightarrow B^{\prime}P_{2}^{\prime}(X_{2}^{\ast})\subset Y_{2},
\label{E:kerB'}%
\end{equation}
in such decompositions of $X$ and $Y$,
the operator
\[
B^{\prime}:\oplus_{j=1}^{4}P_{j}^{\prime}(X_{j}^{\ast})=X^{\ast}\supset
D(B^{\prime})\rightarrow Y=Y_{1}\oplus Y_{2}%
\]
takes the form
\[
B^{\prime}\longleftrightarrow%
\begin{pmatrix}
Q_{1}B^{\prime}P_{1}^{\prime} & 0 & 0 & 0\\
Q_{2}B^{\prime}P_{1}^{\prime} & Q_{2}B^{\prime}P_{2}^{\prime} & 0 & 0
\end{pmatrix}
.
\]
The block forms of $\mathbf{L}$ and $\mathbf{J}\mathbf{L}$ follow from those
of $L$, $A$, $B^{\prime}$, and $B$ through a direct calculation. From Lemma
\ref{L:decom-X}, $L_{2}$ is non-degenerate, which along with $\overline
{R\left(  B\right)  }=X_{1}\oplus X_{2}$, $X_{1}=\ker L_{\overline{R\left(
B\right)  }}$, and the additional assumption $L_{\overline{R\left(  B\right)
}}\geq0$, we obtain the uniform positivity of $L_{2}$, and thus that of
$\mathbf{L}_{3}$.


The proof of the boundedness of $T_{jk}$ and the anti-self-adjointness of
$T_{3}$ is much as that in the proof of Lemma \ref{L:decom-1}. In fact,
according to Lemma \ref{L:decom-X}, $B^{\prime}P_{j}^{\prime}:X_{j}^{\ast
}\rightarrow Y$, $j\neq2$, is a closed operator on the domain $X_{j}^{\ast}$,
the Closed Graph Theorem implies that it is also bounded. Since $B^{\prime
}P_{j}^{\prime}=(P_{j}B)^{\prime}$, $j\neq2$, $P_{j}B$ also has a continuous
extension given by $(B^{\prime}P_{j}^{\prime})^{\prime}$, therefore $P_{j}B$,
$j\neq2$ is also bounded. The boundedness of $T_{jk}$, the closedness and the
density of the domains of $P_{2}B$ and $B^{\prime}P_{2}^{\prime}$ follow
immediately. Moreover $Q_{2}B^{\prime}P_{2}^{\prime}:X_{2}^{\ast}\rightarrow
Y_{2}$ is also closed since $B^{\prime}P_{2}^{\prime}(X_{2}^{\ast})\subset
Y_{2}$ and thus $Q_{2}B^{\prime}P_{2}^{\prime}=B^{\prime}P_{2}^{\prime}$.
Consequently
\[
P_{2}B=(P_{2}B)^{\prime\prime}=(B^{\prime}P_{2}^{\prime})^{\prime},\quad
(Q_{2}B^{\prime}P_{2}^{\prime})^{\prime}=(B^{\prime}P_{2}^{\prime})^{\prime
}Q_{2}^{\prime}=P_{2}BQ_{2}^{\prime},
\]
and
\[
Q_{2}B^{\prime}P_{2}^{\prime}=(Q_{2}B^{\prime}P_{2}^{\prime})^{\prime\prime
}=\left(  (B^{\prime}P_{2}^{\prime})^{\prime}Q_{2}^{\prime}\right)  ^{\prime
}=(P_{2}BQ_{2}^{\prime})^{\prime}.
\]
Since $A_{2}$ and $L_{2}$ are isomorphisms satisfying $A_{2}^{\prime}=A_{2}$
and $L_{2}^{\prime}=L_{2}$, we obtain
\[
(L_{2}P_{2}BQ_{2}^{\prime}A_{2})^{\prime}=A_{2}Q_{2}B^{\prime}P_{2}^{\prime
}L_{2}\;\text{ and }\;(A_{2}Q_{2}B^{\prime}P_{2}^{\prime}L_{2})^{\prime}%
=L_{2}P_{2}BQ_{2}^{\prime}A_{2}.
\]
Therefore, $T_{3}$ is anti-self-adjoint with respect to the equivalent inner
product $\langle\mathbf{L}_{3}\cdot,\cdot\rangle$. Finally, \eqref{E:kerB'}
imply that $\ker(Q_{2}B^{\prime}P_{2}^{\prime})=\ker(B^{\prime}P_{2}^{\prime
})=\{0\}$ and thus $Q_{2}B^{\prime}P_{2}^{\prime}L_{2}$ is injective due to
the non-degeneracy of $L_{2}$. Moreover, $R(L_{\overline{R(B)}})=P_{2}%
^{\prime}(X_{2}^{\ast})$ and $Y_{2}=\overline{R(B^{\prime}L_{\overline{R(B)}%
})}$ also yield that $R(Q_{2}B^{\prime}P_{2}^{\prime})=R(B^{\prime}%
P_{2}^{\prime})\subset Y_{2}$ is dense. Therefore, the dual operator
$P_{2}BQ_{2}^{\prime}$ is injective and the injectivity of $T_{3}$ follows.

Finally, let us make the additional assumption of the compactly embedding of
$Z$ into $Y$. Let $Z_{2}=D(P_{2}BQ_{2}^{\prime}A_{2})\subset Y_{2}$. Since
\[
(BA-P_{2}BQ_{2}^{\prime}A_{2})|_{Z_{2}}=P_{1}BA|_{Z_{2}}\in L(Y_{2},X)
\]
is bounded due to the boundedness of $P_{1}B$, we have that $Z_{2}$ is also
compactly embedded in $Y_{2}$. As $A_{2}$ is uniformly positive definite.
Lemma \ref{L:decom-1} and Lemma \ref{lemma-self-adjoint-2} imply that
$Q_{2}B^{\prime}P_{2}^{\prime}L_{2}P_{2}BQ_{2}^{\prime}A_{2}$ is self-adjoint
on $(Y_{2},\langle A_{2},\cdot,\cdot\rangle)$ with an orthonormal basis of
eigenvectors $\{y_{n}\}_{n=1}^{\infty}$ associated to a sequence of
eigenvalues $0<\lambda_{1}\leq\lambda_{2}\leq\cdots$ of finite multiplicity
accumulating only at $+\infty$. Here the eigenvalues are positive due to
$L_{2}\geq\epsilon>0$ and $\ker T_{3}=\{0\}$. Let
\[
\mathbf{u}_{n}^{\pm}=\left(  \langle L_{2}P_{2}BQ_{2}^{\prime}A_{2}y_{n}%
,P_{2}BQ_{2}^{\prime}A_{2}y_{n}\rangle+\lambda_{n}\langle Ay_{n},y_{n}%
\rangle\right)  ^{-\frac{1}{2}}(\pm P_{2}BQ_{2}^{\prime}A_{2}y_{n},\lambda
_{n}y_{n}).
\]
It is easy to see that $\{\mathbf{u}_{n}^{\pm}\}$ form an orthonormal basis of
$\mathbf{X}_{2}$ by using $\ker T_{3}=\{0\}$ and $T_{3}\mathbf{u}_{n}^{\pm
}=\mp\lambda_{n}\mathbf{u}_{n}^{\mp}$. This completes the proof of the lemma.
\end{proof}

With these preparations, we are ready to prove Theorem \ref{T:main}.

\begin{proof}
[Proof of Theorem \ref{T:main}]
We will prove the theorem largely based on Lemma \ref{lemma-self-adjoint-2}
and the observation that solutions to (\ref{hamiltonian-separated}) satisfy
a second order equation
\begin{equation}
\partial_{tt}v +B^{\prime}LBA v =0. \label{2nd order eqn-v}%
\end{equation}

$\bullet$ \textit{Step 1. Preliminary removal of $\ker L$ and $\ker A$.} Let
\[
\tilde{X}_{1}=\ker L,\quad\tilde{X}_{2}=X_{+}\oplus X_{-},\quad\tilde{Y}%
_{1}=\ker A,\quad\tilde{Y}_{2}=Y_{+}%
\]%
\[
\tilde{L}_{j}=(i_{\tilde{X}_{j}})^{\prime}Li_{\tilde{X}_{j}},\;\;\tilde{A}%
_{j}=(i_{\tilde{Y}_{j}})^{\prime}Ai_{\tilde{Y}_{j}},\;\;\tilde{B}_{jk}%
=\tilde{P}_{j}B\tilde{Q}_{k}^{\prime},\;\;\tilde{B}^{jk}=\tilde{Q}%
_{j}B^{\prime}\tilde{P}_{k}^{\prime},
\]
where $j,k=1,2$ and $\tilde{P}_{1,2}$ are projections associated to
$X=\tilde{X}_{1}\oplus\tilde{X}_{2}$ and $\tilde{Q}_{1,2}$ to $Y=\tilde{Y}%
_{1}\oplus\tilde{Y}_{2}$. Assumptions (\textbf{G1-4}) imply that hypotheses of
Lemma \ref{L:decom-1} are satisfied. Therefore, in the splitting
\[
\mathbf{X}=(\tilde{X}_{1}\oplus\tilde{Y}_{1})\oplus(\tilde{X}_{2}\oplus
\tilde{Y}_{2})
\]
the operator $\mathbf{J}\mathbf{L}$ take the form
\begin{equation}
\mathbf{J}\mathbf{L}\leftrightarrow%
\begin{pmatrix}
0 & \tilde{\mathbf{T}}_{12}\\
0 & \tilde{\mathbf{J}}_{2}\tilde{\mathbf{L}}_{2}%
\end{pmatrix}
, \label{E:JL-decom-1}%
\end{equation}
where
\[
\tilde{\mathbf{J}}_{2}\leftrightarrow%
\begin{pmatrix}
0 & \tilde{B}_{22}\\
-\tilde{B}^{22} & 0
\end{pmatrix}
,\quad\tilde{\mathbf{L}}_{2}\leftrightarrow%
\begin{pmatrix}
\tilde{L}_{2} & 0\\
0 & \tilde{A}_{2}%
\end{pmatrix}
,\quad\tilde{\mathbf{T}}_{12}\leftrightarrow%
\begin{pmatrix}
0 & \tilde{B}_{12}\tilde{A}_{2}\\
-\tilde{B}^{12}\tilde{L}_{2} & 0
\end{pmatrix}
\]
and $(\tilde{L}_{2},\tilde{J}_{2},\tilde{B}_{22})$ satisfy (\textbf{G1-4}).
Moreover, the same lemma also implies that $\tilde{T}_{12}$ is bounded and
both $\tilde{L}_{2}$ and $\tilde{A}_{2}$ are non-degenerate.

$\bullet$ \textit{Step 2. Hyperbolic subspaces.}
As $\tilde{A}_{2}\geq\epsilon$ for some $\epsilon>0$, according to Lemma
\ref{lemma-self-adjoint-2}, $\mathbb{\tilde{L}}=\tilde{B}_{22}^{\prime}%
\tilde{L}_{2}\tilde{B}_{22}\tilde{A}_{2}$ is self-adjoint on $\tilde{Y}_{2}$
with respect to the inner product $\left[  \cdot,\cdot\right]  =\left\langle
\tilde{A}_{2}\cdot,\cdot\right\rangle $. Since for any $v_{1},v_{2}\in
D(\mathbb{\tilde{L}})\subset\tilde{Y}_{2}$,
\[
\left[  \mathbb{\tilde{L}}v_{1},v_{2}\right]  =\left\langle \tilde{A}%
_{2}\tilde{B}_{22}^{\prime}\tilde{L}_{2}\tilde{B}_{22}\tilde{A}_{2}v_{1}%
,v_{2}\right\rangle =\left\langle \tilde{L}_{2}\tilde{B}_{22}\tilde{A}%
_{2}v_{1},\tilde{B}_{22}\tilde{A}_{2}v_{2}\right\rangle ,
\]
and
\[
BA=%
\begin{pmatrix}
\tilde{B}_{11} & \tilde{B}_{12}\\
\tilde{B}_{21} & \tilde{B}_{22}%
\end{pmatrix}%
\begin{pmatrix}
0 & 0\\
0 & \tilde{A}_{2}%
\end{pmatrix}
=%
\begin{pmatrix}
0 & \tilde{B}_{12}\tilde{A}_{2}\\
0 & \tilde{B}_{22}\tilde{A}_{2}%
\end{pmatrix}
\Longrightarrow R(\tilde{B}_{22}\tilde{A}_{2})=\tilde{P}_{2}\left(
R(BA)\right)  ,
\]
along with the definition of $\tilde{X}_{1}$, we obtain the dimension of the
eigenspace of negative eigenvalues of the operator $\mathbb{\tilde{L}}$ given
by
\[
n_{1}\triangleq n^{-}\left(  \mathbb{\tilde{L}}\right)  =n^{-}\left(
\tilde{L}_{2}|_{_{\overline{R\left(  \tilde{B}_{22}\tilde{A}_{2}\right)  }}%
}\right)  =n^{-}(L|_{_{\overline{R(BA)}}})\leq n^{-}(L).
\]
Let $\tilde{v}_{j}$ be the eigenvectors of $\tilde{\mathbb{L}}$ associate with
eigenvalues $-\lambda_{j}^{2}<0$, $j=1,\ldots,n_{1}$, which might be repeated,
such that
\[
\lbrack\tilde{v}_{j},\tilde{v}_{j^{\prime}}]=\delta_{jj^{\prime}},\quad
\lbrack\tilde{\mathbb{L}}\tilde{v}_{j},\tilde{v}_{j^{\prime}}]=-\lambda
_{j}^{2}\delta_{jj^{\prime}}.
\]
Let
\[
\tilde{u}_{j}=\lambda_{j}^{-1}\tilde{B}_{22}\tilde{A}_{2}\tilde{v}%
_{j},\;\;\tilde{\mathbf{u}}_{j}^{\pm}=(\tilde{u}_{j},\pm\tilde{v}%
_{j})\Longrightarrow\tilde{\mathbf{J}}_{2}\tilde{\mathbf{L}}_{2}%
\tilde{\mathbf{u}}_{j}^{\pm}=\pm\lambda_{j}\tilde{\mathbf{u}}_{j}^{\pm
},\;\;\langle\tilde{L}_{2}\tilde{u}_{j},\tilde{u}_{k}\rangle=-\delta_{jk}.
\]
To return to $\mathbf{J}\mathbf{L}$, let
\[
\mathbf{u}_{j}^{\pm}=(u_{j},\pm v_{j})\triangleq\tilde{\mathbf{u}}_{j}^{\pm
}\pm\lambda_{j}^{-1}\tilde{\mathbf{T}}_{12}\tilde{\mathbf{u}}_{j}^{\pm
}=\left(  (\tilde{u}_{j}+\lambda_{j}^{-1}\tilde{B}_{12}\tilde{A}_{2}\tilde
{v}_{j},\pm(v_{j}-\lambda_{j}^{-1}\tilde{B}^{12}\tilde{L}_{2}\tilde{u}%
_{j})\right)  ,
\]
which are the eigenvectors of $\mathbf{J}\mathbf{L}$ with eigenvalue
$\pm\lambda_{j}$ satisfying
\[
\mathbf{J}\mathbf{L}\mathbf{u}_{j}^{\pm}=\pm\lambda_{j}\mathbf{u}_{j}^{\pm
},\quad\langle Lu_{j},u_{k}\rangle=-\delta_{jk},\quad\langle Av_{j}%
,v_{k}\rangle=\delta_{jk}.
\]
Define the hyperbolic subspaces as
\[
E^{u}=span\{\mathbf{u}_{j}^{+}\mid j=1,\ldots,n_{1}\},\quad E^{s}%
=span\{\mathbf{u}_{j}^{-}\mid j=1,\ldots,n_{1}\},
\]
and statement ii) follows.

$\bullet$ \textit{Step 3. Reduction to the center subspace.} Let%
\[
X_{h}=span\left\{  u_{j}\ |\ j=1,\cdots,n_{1}\right\}  \subset R\left(
B\right)  ,\ \ X_{c}=\left\{  u\in X\ |\ \left\langle Lu,\tilde{u}%
\right\rangle ,\ \tilde{u}\in X_{h}\right\}  ,
\]%
\[
Y_{h}=span\left\{  v_{j}\ |\ j=1,\cdots,n_{1}\right\}  \subset R\left(
B^{\prime}\right)  ,\ \ Y_{c}=\left\{  v\in Y\ |\ \left\langle Av,\tilde
{v}\right\rangle ,\ \tilde{v}\in Y_{h}\right\}  ,
\]
and
\[
E^{c}=X_{c}\times Y_{c}\Longrightarrow\mathbf{X}=(X_{h}\times Y_{h})\oplus
E^{c}=E^{s}\oplus E^{u}\oplus E^{c}.
\]
Due to the invariance of $E^{s,u}$ under $e^{t\mathbf{J}\mathbf{L}}$, that of
$E^{c}$ and the rest of statements i) and (iii) follow from standard arguments
(see, e.g. \cite{lin-zeng-hamiltonian}, for more details). Apparently $\ker
L\subset X_{c}$ and $\ker A\subset Y_{c}$.

The above calculations show that $L_{X_{h}}$ and $A_{Y_{h}}$ are
non-degenerate and thus Lemma 12.2 in \cite{lin-zeng-hamiltonian} yields
\[
X=X_{h}\oplus X_{c},\quad Y=Y_{h}\oplus Y_{c},
\]
with associated projections $P_{h,c}$ and $Q_{h,c}$. By their definitions we
have
\[
P_{h}^{\prime}(X_{h}^{\ast})=\ker(i_{X_{c}})^{\prime}=L(X_{h})\subset
D(B^{\prime}),\;\;Q_{h}^{\prime}(Y_{h}^{\ast})=\ker(i_{Y_{c}})^{\prime
}=A(Y_{h})\subset D(B).
\]
Therefore, these decompositions satisfy the assumptions of Lemma
\ref{L:decom-1}
and thus \eqref{hamiltonian-separated} restricted on the invariant $E^{c}$
also has the separable Hamiltonian form with
\[
(L_{X_{c}},A_{Y_{c}},B_{c}=P_{c}BQ_{c}^{\prime})
\]
satisfying (\textbf{G1-4}). The invariance of $E^{c}$ and $X_{h}\times Y_{h}$
and the block form in Lemma \ref{L:decom-1} imply
\[
R(B_{c}A_{Y_{c}})=BA(Y_{c})\subset X_{c}.
\]
Due to the $L$-orthogonality between $X_{c}$ and $X_{h}$, we also have the
$L$-orthogonality between $X_{h}$ and $\overline{R(B_{c}A_{c})}$ both of which
are contained in $\overline{R(BA)}$. As $L$ is negative definite on $X_{h}$,
we obtain
\[
n^{-}(L|_{_{\overline{R(BA)}}})\geq n^{-}(L_{\overline{R(B_{c}A_{c})}})+\dim
X_{h}=n^{-}(L_{\overline{R(B_{c}A_{c})}})+n^{-}(L|_{_{\overline{R(BA)}}}),
\]
which implies
\begin{equation}
n^{-}(L_{X_{c}}|_{\overline{R(B_{c}A_{Y_{c}})}})=0. \label{E:Morse-I-C}%
\end{equation}

\begin{remark}
Due to the invariance of $E^{u,s,c}$ under $e^{t\mathbf{J}\mathbf{L}}$ and the
non-degeneracy of $\mathbf{J}\mathbf{L}$ and $A$ on the finite dimensional
$E^{u,s}$ and $Y_{h}$ respectively, it follows that

a.) $A_{Y^{c}}$ is injective on $\overline{R(B_{c}^{\prime}L_{X_{c}}%
B_{c}A_{Y_{c}})}=\overline{B^{\prime}LBA(X_{c})}$ if $A$ is injective on
$\overline{R(B^{\prime}LBA)}$;

b.) $\overline{R(B_{c}A_{Y_{c}})} = X_{c}$ if $\overline{R(BA)}= X$.
\end{remark}

$\bullet$ \textit{Step 4. Reduction (again) of $\ker L_{X_{c}}$ and $A_{Y_{c}%
}$ in $E^{c}$.} We shall basically redo Step 1 in $E^{c} = X_{c} \times Y_{c}%
$. It would be a much cleaner exposition if we could find a way to combine
these two steps together. However we were not able to manage that as the
positivity of $A$ is required in Lemma \ref{lemma-self-adjoint-2} to identify
the hyperbolic directions and meanwhile there is not a clear simple way to
separate the kernels in a decomposition invariant under $e^{t\mathbf{J}%
\mathbf{L}}$.

Let
\begin{equation}%
\begin{split}
\mathbf{X}_{0L}  &  =\ker L_{X_{c}}\times\{0\}=\ker L\times\{0\},\\
\mathbf{X}_{0A}  &  =\{0\}\times\ker A_{Y_{c}}=\{0\}\times\ker A.
\end{split}
\label{E:X_0}%
\end{equation}
According to Lemma \ref{L:decom-1}, $X_{c}$ and $Y_{c}$ satisfies
(\textbf{G1-4}), so there exist closed subspaces of $\tilde{X}\subset X_{c}$
and $\tilde{Y}\subset Y_{c}$ such that
\[
X_{c}=\tilde{X}\oplus\ker L,\quad\ker(i_{\tilde{X}})^{\prime}\subset
D(B_{c}^{\prime}),\quad Y_{c}=\ker A\oplus\tilde{Y},\quad\ker(i_{\tilde{Y}%
})^{\prime}\subset D(B_{c}).
\]
Let
\[
\tilde{\mathbf{X}}=\tilde{X}\times\tilde{Y}.
\]
Applying Lemma \ref{L:decom-1} again to the decomposition $E^{c}%
=(\mathbf{X}_{0L}\oplus\mathbf{X}_{0A})\oplus\tilde{\mathbf{X}}$, we obtain
the block forms of \eqref{hamiltonian-separated} restricted on the invariant
$E^{c}$ and its energy $\mathbf{L}_{E^{c}}$
\[
\mathbf{L}_{E^{c}}\longleftrightarrow%
\begin{pmatrix}
0 & 0\\
0 & \tilde{\mathbf{L}}%
\end{pmatrix}
,\quad\mathbf{J}\mathbf{L}|_{E^{c}}\longleftrightarrow%
\begin{pmatrix}
0 & T_{0\sim}\\
0 & \tilde{\mathbf{J}}\tilde{\mathbf{L}}%
\end{pmatrix}
,
\]
where $T_{0\sim}$ is bounded and $\tilde{\mathbf{J}}\tilde{\mathbf{L}}$ has
the separable Hamiltonian form with
\[
(L_{\tilde{X}},A_{\tilde{Y}},\tilde{B}=\tilde{P}B_{c}\tilde{Q}^{\prime}),\quad
L_{\tilde{X}}\text{ and }A_{\tilde{Y}}\text{ non-degenerate}.
\]
Here $\tilde{P}:X_{c}\rightarrow\tilde{X}$ and $\tilde{Q}:Y_{c}\rightarrow
\tilde{Y}$ are the associated projections.
Finally, Lemma \ref{L:decom-1} implies $\tilde{B}A_{\tilde{Y}}=\tilde{P}%
B_{c}A_{Y_{c}}|_{\tilde{Y}}$, which along with the definition of
$\mathbf{X}_{0L,0A}$, the fact that $A_{\tilde{Y}}:\tilde{Y}\rightarrow
\tilde{Y}^{\ast}$ is isomorphic, and \eqref{E:Morse-I-C} yield
\[
n^{-}(L_{\tilde{X}}|_{\overline{R(\tilde{B})}})=n^{-}(L_{\tilde{X}%
}|_{\overline{R(\tilde{B}A_{\tilde{Y}})}})=n^{-}(L_{\tilde{X}}|_{\overline
{\tilde{P}R(B_{c}A_{Y_{c}})}})=n^{-}(L_{\tilde{X}}|_{\overline{R(B_{c}%
A_{Y_{c}})}})=0.
\]
Therefore, $(L_{\tilde{X}},A_{\tilde{Y}},\tilde{B})$ satisfy all the
assumptions in Proposition \ref{P:block-1}.

\begin{remark}
Due to the upper triangular block form of $\mathbf{J} \mathbf{L}%
|_{\mathbf{X}_{c}}$ and the remark at the end of the last step, we have

a.) $\mathbf{X}_{0A} =\{0\}$ if $A$ is injective on $\overline{R(B^{\prime
}LBA)}$;

b.) $\overline{R(\tilde B A_{\tilde Y})} = \overline{R(\tilde P B_{c}
A_{Y_{c}})}= \tilde X$ if $\overline{R(BA)}= X$.
\end{remark}

$\bullet$ \textit{Step 5. Proof of statement iv).} The block form
decomposition of $\mathbf{L}$ and $\mathbf{J}\mathbf{L}$ on $E^{c}$ follows
from the above splitting and Proposition \ref{P:block-1}. As in Lemma
\ref{L:decom-X}, here $\mathbf{X}_{1}=X_{1}\times\{0\}$ and $X_{1}%
=L_{\overline{R(\tilde{B})}}$. Those zero blocks in the bounded operator
\[
T_{0\sim}:\tilde{\mathbf{X}}\rightarrow\ker L\times\ker A
\]
are due to the facts that $\mathbf{J}\mathbf{L}$ maps $X\times\{0\}$ to
$\{0\}\times Y$ and vice versa. The well-posedness of $e^{t\mathbf{J}%
\mathbf{L}}$ and its $O(1+|t|^{3})$ growth estimate follow from direct
computation based on the block form of $\mathbf{J}\mathbf{L}$ where the only
unbounded operator $T_{3}$ generates a unitary group $e^{tT_{3}}$.

$\bullet$ \textit{Statement v)} follows directly from Proposition
\ref{P:block-1}.
\end{proof}

In order to obtain the better estimates of $e^{t\mathbf{J}\mathbf{L}}$, we
only need to refine or modify the decomposition under various assumptions.

\begin{proof}
[Proof of Theorem \ref{T:main2}]According to the remark at the end of the
above Step 4, $\mathbf{X}_{0A} =\{0\}$ under the assumption of i) and thus the
second row and column in the block form of $\mathbf{J}\mathbf{L}_{E^{c}}$
disappear which immediately implies the $O(1+t^{2})$ growth of $e^{t\mathbf{J}
\mathbf{L}}|_{E^{c}}$. The same remark and Lemma \ref{L:decom-X} imply that,
under the assumption of ii), $\mathbf{X}_{1}= \mathbf{X}_{5} =\{0\}$, the
$O(1+|t|)$ growth of $e^{t\mathbf{J} \mathbf{L}}|_{E^{c}}$ follows from the
reduced block form of $\mathbf{J}\mathbf{L}_{E^{c}}$ readily.

$\bullet$ \textit{Proof of statement iii).} Under the non-degeneracy
assumptions of $L_{\overline{R(B)}}$ and $A_{\overline{R(B^{\prime})}}$, the
decomposition of $X$ can be carried out in a different, but much simpler, way.
In fact, Lemma 12.2 in \cite{lin-zeng-hamiltonian} implies
\[
X=X_{0}\oplus\tilde{X},\;\;\tilde{X}=\overline{R(B)},\;\;X_{0}=\ker(B^{\prime
}L)=\{u\in X\mid\langle Lu,\tilde{u}\rangle=0,\,\tilde{u}\in\tilde{X}\},
\]%
\[
Y=Y_{0}\oplus\tilde{Y},\;\;\tilde{Y}=\overline{R(B^{\prime})},\;\;Y_{0}%
=\ker(BA)=\{u\in Y\mid\langle Av,\tilde{v}\rangle=0,\,\tilde{v}\in\tilde
{Y}\},
\]
associated with the projection $\tilde{P}$ on $X$ and $\tilde{Q}$ on $Y$,
respectively. In the decomposition
\[
\mathbf{X}=\mathbf{X}_{0}\oplus\tilde{\mathbf{X}},\quad\mathbf{X}_{0}%
=X_{0}\times Y_{0},\quad\tilde{\mathbf{X}}=\tilde{X}\times\tilde{Y},
\]
which is invariant under $\mathbf{J}\mathbf{L}$, we have
\[
\mathbf{J}\mathbf{L}\Longleftrightarrow%
\begin{pmatrix}
0 & 0\\
0 & \tilde{\mathbf{J}}\tilde{\mathbf{L}}%
\end{pmatrix}
,\quad\tilde{\mathbf{L}}=%
\begin{pmatrix}
L_{\tilde{X}} & 0\\
0 & A_{\tilde{Y}}%
\end{pmatrix}
,\ \ \ \ \tilde{\mathbf{J}}=%
\begin{pmatrix}
0 & \tilde{P}B\tilde{Q}^{\prime}\\
-\tilde{Q}B^{\prime}\tilde{P}^{\prime} & 0
\end{pmatrix}
,
\]
where $\tilde{\mathbf{J}}\tilde{\mathbf{L}}$ is also in the separable
Hamiltonian form $(L_{\tilde{X}},A_{\tilde{Y}},\tilde{B}=\tilde{P}B\tilde
{Q}^{\prime})$. In particular, $L_{\tilde{X}}$ and $A_{\tilde{Y}}$ are
non-degenerate on $\tilde{X}$ and $\tilde{Y}$ and $\tilde{\mathbf{J}}%
\tilde{\mathbf{L}}$ is injective on $\tilde{\mathbf{X}}$, the last of which
implies $\overline{R(\tilde{B}A_{\tilde{Y}})}=\tilde{X}$. From the above
theorem, the system $\tilde{\mathbf{J}}\tilde{\mathbf{L}}$ has the trichotomy
decomposition
\[
\tilde{\mathbf{X}}=\tilde{E}^{u}\oplus\tilde{E}^{s}\oplus\tilde{E}^{c}%
,\quad\dim\tilde{E}^{u,s}=n^{-}(L_{\overline{R(\tilde{B}A_{\tilde{Y}})}%
})=n^{-}(L_{\tilde{X}}).
\]
Lemma 12.2 in \cite{lin-zeng-hamiltonian} implies $\tilde{\mathbf{L}}$ is
uniformly positive definite on $\tilde{E}^{c}$ and thus we obtain the Lyapunov
stability on $\tilde{E}^{c}$. Clearly
\[
E^{u,s}=\tilde{E}^{u,s},\quad E^{c}=\mathbf{X}_{0}\oplus\tilde{E}^{c},
\]
give the trichotomy decomposition of $\mathbf{J}\mathbf{L}$ and the thus its
Lyapunov stability on $E^{c}$ follows.

\end{proof}


To end the section, we prove the following result on perturbations to $L$.

\begin{proposition}
\label{P:perturbation} Suppose $X$ is a Hilbert space and $L:X\rightarrow
X^{\ast}$ satisfies (\textbf{G3}) and $n_{0}=\dim\ker L<\infty$. It holds that
there exists $C,\delta>0$ such that any bounded $\tilde{L}:X\rightarrow X$
with $\tilde{L}^{\ast}=\tilde{L}$ and $\Vert\tilde{L}-L\Vert<\delta$ also
satisfies (\textbf{G3}). Moreover, there exists $\tilde{L}_{0}:\ker
L\rightarrow(\ker L)^{\ast}$ such that
\begin{align*}
&  \dim\ker\tilde{L}=\dim\ker\tilde{L}_{0},\quad n^{-}(\tilde{L}%
)-n^{-}(L)=n^{-}(\tilde{L}_{0}),\\
&  \Vert\tilde{L}_{0}-(\tilde{L}-L)_{\ker L}\Vert<C\Vert\tilde{L}-L\Vert^{2},
\end{align*}
where the notation $(\tilde{L}-L)_{\ker L}$ is in the fashion of \eqref{E:L_X_1}.
\end{proposition}

\begin{corollary}
\label{cor neg unchange}If, in addition, $L$ is non-degenerate, then
$\tilde{L}$ is also non-degenerate and $n^{-}(\tilde{L})=n^{-}(L)$.
\end{corollary}

\begin{proof}
[Proof of Proposition \ref{P:perturbation}]Let $X_{\pm}\subset X$ be closed
subspaces ensured by (\textbf{G3}) for $L$. Denote
\[
X_{0}=\ker L,\;\;X_{1}=X_{+}\oplus X_{-},\;\;\tilde{X}_{0}=X_{1}%
^{\perp_{\tilde{L}}}=\{x\in X\mid\langle\tilde{L}x,x_{1}\rangle=0,\ \forall
x_{1}\in X_{1}\}.
\]
Clearly $L_{X_{1}}=i_{X_{1}}^{\ast}Li_{X_{1}}:X_{1}\rightarrow X_{1}^{\ast}$
is an isomorphism. The closeness between $\tilde{L}$ and $L$ implies that
$\tilde{L}_{X_{1}}:X_{1}\rightarrow X_{1}^{\ast}$ is also an isomorphism and
$n^{-}(\tilde{L}_{X_{1}})=n^{-}(L_{X_{1}})$. Therefore, the we have
\[
\dim\ker\tilde{L}=\dim\ker\tilde{L}_{\tilde{X}_{0}},\quad n^{-}(\tilde
{L})-n^{-}(L)=n^{-}(\tilde{L}_{\tilde{X}_{0}}).
\]
To analyze $\tilde{L}_{\tilde{X}_{0}}$, a standard argument yields a unique
bounded linear operator $S:X_{0}\rightarrow X_{1}$ such that
\[
\Vert S\Vert\leq C\Vert\tilde{L}-L\Vert,\quad\tilde{X}_{0}=graph(S)=\{x_{0}%
+Sx_{0}\mid x_{0}\in X_{0}\}.
\]
Using the isomorphism $I+S:X_{0}\rightarrow\tilde{X}_{0}$ as conjugacy map,
let
\[
\tilde{L}_{0}=(I+S^{\ast})\tilde{L}(I+S):X_{0}\rightarrow X_{0}^{\ast}.
\]
We have, for $x_{0},x_{0}^{\prime}\in X_{0}$,
\begin{align*}
&  \langle\tilde{L}_{0}x_{0},x_{0}^{\prime}\rangle=\langle\tilde{L}%
(x_{0}+Sx_{0}),(x_{0}^{\prime}+Sx_{0}^{\prime})\rangle\\
=  &  \langle\tilde{L}_{X_{0}}x_{0},x_{0}^{\prime}\rangle+\langle\tilde
{L}Sx_{0},x_{0}^{\prime}\rangle+\langle\tilde{L}x_{0},Sx_{0}^{\prime}%
\rangle+\langle\tilde{L}Sx_{0},Sx_{0}^{\prime}\rangle\\
=  &  \langle(\tilde{L}-L)_{X_{0}}x_{0},x_{0}^{\prime}\rangle+\langle
(\tilde{L}-L)_{X_{0}}x_{0}^{\prime},Sx_{0}\rangle+\left\langle (\tilde
{L}-L)_{X_{0}}x_{0},Sx_{0}^{\prime}\right\rangle +\left\langle S^{\ast}%
\tilde{L}Sx_{0},x_{0}^{\prime}\right\rangle
\end{align*}
where we used $L_{X_{0}}=0$. Therefore, the estimate on $\tilde{L}_{0}$
follows from that on $S$.
\end{proof}

\section{Stability of non-rotating stars}

In this section, we study stability of non-rotating stars. We divide it into
several steps.

\subsection{Existence of non-rotating stars}

Non-rotating stars are steady solutions$\ \left(  \rho,u\right)  =\left(
\rho_{0}\left(  \left\vert x\right\vert \right)  ,0\right)  $ of
(\ref{continuity-EP})-(\ref{Poisson-EP}), where $\rho_{0}\left(  r\right)  $
satisfies
\begin{equation}
-\nabla P\left(  \rho_{0}\right)  -\rho_{0}\nabla V_{0}=0\label{steady-EP}%
\end{equation}
with $\Delta V_{0}=4\pi\rho_{0}$. For the consideration of the existence of
non-rotating stars, we assume $P\left(  \rho\right)  $ satisfies assumption
(\ref{assumption-P1}) and
\begin{equation}
\lim_{s\rightarrow0+}s^{1-\gamma_{0}}P^{\prime}\left(  s\right)
=K>0,\ \text{\ for some }\gamma_{0}>\frac{6}{5}.\label{assumption-P2'}%
\end{equation}
Note that the enthalpy function $\Phi$ defined by (\ref{defn-enthalpy}) is
convex since $P^{\prime}\left(  \rho\right)  >0\ $for $\rho>0\ $by assumption
(\ref{assumption-P1}). Let $F\left(  s\right)  =\left(  \Phi^{\prime}\right)
^{-1}\left(  s\right)  $ for $s\in\left(  0,s_{\max}\right)  $, where
\[
s_{\max}=\int_{0}^{\infty}\frac{P^{\prime}\left(  \rho\right)  }{\rho}d\rho
\in(0,\infty].
\]
We extend $F\left(  s\right)  $ to $s\in\left(  -\infty,0\right)  $ by zero
extension and denote the extended function by $F_{+}\left(  s\right)
:\mathbf{R}\rightarrow\lbrack0,\infty)$. We consider physically realistic
non-rotating stars $\rho_{0}$ with compact support
\[
\left\{  \rho_{0}>0\right\}  =\left\{  \left\vert x\right\vert <R\right\}
\triangleq B_{R},
\]
where $R>0$ is the radius of the support. Then by (\ref{steady-EP}), we have%
\begin{equation}
V_{0}+\Phi^{\prime}\left(  \rho_{0}\right)  =V_{0}\left(  R\right)
\label{eqn-V0}%
\end{equation}
and $\rho_{0}=F\left(  V_{0}\left(  R\right)  -V_{0}\right)  $ inside $B_{R}$.
Since $V_{0}^{\prime}\left(  r\right)  >0$ by the Poisson equation, when $r>R$
we have
\[
\rho_{0}\left(  r\right)  =0=F_{+}\left(  V_{0}\left(  R\right)  -V_{0}\left(
r\right)  \right)  .
\]
Therefore, the steady potential $V_{0}\left(  \left\vert x\right\vert \right)
$ satisfying the nonlinear elliptic equation in radial coordinates%
\begin{equation}
\Delta V_{0}=V_{0}^{\prime\prime}+\frac{2}{r}V_{0}^{\prime}=4\pi F_{+}\left(
V_{0}\left(  R\right)  -V_{0}\right)  .\label{eqn-V-EP}%
\end{equation}
Define $y\left(  r\right)  =V_{0}\left(  R\right)  -V_{0}\left(  r\right)
=\Phi^{\prime}\left(  \rho_{0}\right)  $. Then $y$ satisfies the ODE
\begin{equation}
y^{\prime\prime}+\frac{2}{r}y^{\prime}=-4\pi F_{+}\left(  y\right)
.\label{eqn-ode-star}%
\end{equation}
Let $\mu=\rho_{0}\left(  0\right)  $ to be the center density. We solve
(\ref{eqn-ode-star}) with the initial condition
\begin{equation}
y\left(  0\right)  =\Phi^{\prime}\left(  \rho_{0}\left(  0\right)  \right)
=\Phi^{\prime}\left(  \mu\right)  >0,\ \ \ y^{\prime}\left(  0\right)
=0,\label{initial-condition-ode}%
\end{equation}
or equivalently the first order equation
\begin{equation}
y^{\prime}\left(  r\right)  =-\frac{4\pi}{r^{2}}\int_{0}^{r}s^{2}F_{+}\left(
y\left(  s\right)  \right)  ds,\ \ \ y\left(  0\right)  =\Phi^{\prime}\left(
\mu\right)  \text{.}\label{first-order-ode}%
\end{equation}
It is easy to see that the unique solution $y_{\mu}\left(  r\right)  $ of the
above ODE exists for $r\in\left(  0,+\infty\right)  $ and $y_{\mu}^{\prime
}\left(  r\right)  <0$. If there exists a finite number $R_{\mu}>0$ such that
$y_{\mu}\left(  R_{\mu}\right)  =0$, define
\begin{equation}
\rho_{\mu}\left(  \left\vert x\right\vert \right)  =\left\{
\begin{array}
[c]{cc}%
F\left(  y_{\mu}\left(  \left\vert x\right\vert \right)  \right)   & \text{if
}\left\vert x\right\vert <R_{\mu}\\
0 & \text{if }\left\vert x\right\vert \geq R_{\mu}%
\end{array}
\right.  \label{defn-rho-mu}%
\end{equation}
and $V_{\mu}=4\pi\Delta^{-1}\rho_{\mu}$. Then $\left(  \rho_{\mu},0\right)  $
is a non-rotating steady solution of (\ref{continuity-EP})-(\ref{Poisson-EP})
with compact support and $R_{\mu}$ is the support radius.

\begin{remark}
\label{R:smooth-in-mu} Since $F_{+}$ is actually a $C^{1}$ function for
$\gamma\in(\frac{6}{5},2)$, the solution $(y,y^{\prime})$ to
\eqref{eqn-ode-star} and \eqref{initial-condition-ode} is $C^{1}$ in both $r$
and in $\mu$ with $y^{\prime}<0$. Therefore, the Implicit Function Theorem
implies that $R_{\mu}$ is $C^{1}$ in $\mu$ and thus so is $\rho_{\mu}$.
\end{remark}

Below, we give some conditions to ensure that the ODE (\ref{eqn-ode-star}) has
solutions with compact support. Assume $P\left(  \rho\right)  $ satisfies
(\ref{assumption-P1})-(\ref{assumption-P2}). For Polytropic stars with
$P\left(  \rho\right)  =K\rho^{\gamma}$ $\left(  \gamma>\frac{6}{5}\right)  $,
it is well known (\cite{chandra-book-stellar}) that for any center density
$\mu>0$, there exists compact supported solutions. Let $\gamma=1+\frac{1}{n}$,
(\ref{eqn-ode-star}) becomes the classical Lane-Emden equation
\begin{equation}
y^{\prime\prime}+\frac{2}{r}y^{\prime}=-4\pi\left(  \frac{\gamma-1}{K\gamma
}\right)  ^{n}y_{+}^{n}=-C_{\gamma}y_{+}^{n},\ \ \ \label{ode-y-LE}%
\end{equation}
where $0<n<5,\ y_{+}=\max\left\{  y,0\right\}  ,$ and
\[
C_{\gamma}=4\pi\left(  \frac{\gamma-1}{K\gamma}\right)  ^{\frac{1}{\gamma-1}%
}.\
\]
Let $y_{\mu}\left(  r\right)  =\Phi^{\prime}\left(  \rho_{\mu}\left(
r\right)  \right)  $ be the solution of (\ref{ode-y-LE}) with
\[
y_{\mu}\left(  0\right)  =\Phi^{\prime}\left(  \mu\right)  =\frac{K\gamma
}{\gamma-1}\mu^{\gamma-1}=:\alpha.
\]
Denote the transformation
\begin{equation}
y_{\mu}\left(  r\right)  =\alpha\theta\left(  \alpha^{\frac{n-1}{2}}r\right)
,\ \ s=\alpha^{\frac{n-1}{2}}r,\ \label{transform-L_E}%
\end{equation}
then $\theta\left(  s\right)  \,$satisfies the same equation
\begin{equation}
\theta^{\prime\prime}+\frac{2}{s}\theta^{\prime}=-C_{\gamma}\theta_{+}%
^{n},\ \ \ \theta\left(  0\right)  =1,\theta^{\prime}\left(  0\right)  =0.
\label{lane-emden equation}%
\end{equation}
The function $\theta\left(  s\right)  \,$is called the Lane-Emden function.

The next lemma shows that under assumption (\ref{assumption-P1}%
)-(\ref{assumption-P2'}), non-rotating stars with compact support exist for
small center density.

\begin{lemma}
\label{lemma-steady-small-density}Assume (\ref{assumption-P1}) and
(\ref{assumption-P2'}). There exists $\mu_{0}>0$ such that for any $\mu
\in\left(  0,\mu_{0}\right)  $, $y_{\mu}\left(  R_{\mu}\right)  =0$ for some
$R_{\mu}>0$. Here, $y_{\mu}\left(  r\right)  $ is the solution of
(\ref{eqn-ode-star}) with the initial condition (\ref{initial-condition-ode}).
Then $\rho_{\mu}\left(  \left\vert x\right\vert \right)  $ defined by
(\ref{defn-rho-mu}) is a non-rotating star with support radius $R_{\mu}$.
\end{lemma}

\begin{proof}
It is equivalent to prove the statement for $\alpha=y_{\mu}\left(  0\right)
=\Phi^{\prime}\left(  \mu\right)  $ sufficiently small. Motivated by
(\ref{assumption-P2'}) and (\ref{transform-L_E}), we define
\begin{equation}
y_{\mu}\left(  r\right)  =\alpha\theta_{\alpha}\left(  \alpha^{\frac{n_{0}%
-1}{2}}r\right)  ,\ \ s=\alpha^{\frac{n_{0}-1}{2}}r, \label{defn-y-mu-scaling}%
\end{equation}
where $n_{0}=\frac{1}{\gamma_{0}-1}$. Then $\theta_{\alpha}\left(  s\right)  $
satisfies the equation
\begin{equation}
\theta_{\alpha}^{\prime\prime}+\frac{2}{s}\theta_{\alpha}^{\prime}=-4\pi
\frac{1}{\alpha^{n_{0}}}F_{+}\left(  \alpha\theta_{\alpha}\right)  =-
g_{\alpha}\left(  \theta_{\alpha}\right)  ,\ \ \ \ \label{eqn-theta-alpha}%
\end{equation}
with the initial condition $\theta_{\alpha}\left(  0\right)  =1,\theta
_{\alpha}^{\prime}\left(  0\right)  =0$. Denote
\begin{equation}
g_{\alpha}\left(  \theta\right)  =4\pi\frac{1}{\alpha^{n_{0}}}F_{+}\left(
\alpha\theta\right)  ,\ \ \theta\in\left[  0,1\right]  , \label{defn-g-alpha}%
\end{equation}
and
\begin{equation}
g_{0}\left(  \theta\right)  =C_{\gamma_{0}}\theta_{+}^{n_{0}},\ \ \ C_{\gamma
_{0}}=4\pi\left(  \frac{\gamma_{0}-1}{K\gamma_{0}}\right)  ^{\frac{1}%
{\gamma_{0}-1}}. \label{defn-g-0}%
\end{equation}
Then by assumption (\ref{assumption-P2'}) and the definition of $F_{+}$, it is
easy to show that when $\alpha\rightarrow0+,\ \ g_{\alpha}\rightarrow g_{0}$
in $C^{1}\left(  [ 0,1]\right)  $ and in $C^{0}((-\infty, 1])$. Let
$\theta_{0}\left(  s\right)  $ be the Lane-Emden function satisfying
\begin{equation}
\theta_{0}^{\prime\prime}+\frac{2}{s}\theta_{0}^{\prime}=-C_{\gamma_{0}%
}\left(  \theta_{0}\right)  _{+}^{n_{0}}=g_{0}\left(  \theta_{0}\right)
,\ \ \theta_{0}\left(  0\right)  =1,\theta_{0}^{\prime}\left(  0\right)  =0.
\label{eqn-LE-g-0}%
\end{equation}
Then for any $R>0$, we have $\theta_{\alpha}\rightarrow\theta_{0}$ in
$C^{1}\left(  0,R\right)  $. Define $G\left(  \alpha,s\right)  =\theta
_{\alpha}\left(  s\right)  $ for $\alpha>0,s>0\ $and $G\left(  0,s\right)  =$
$\theta_{0}\left(  s\right)  $ . Let $R_{0}$ be the support radius of
$\theta_{0}$, then $G\left(  0,R_{0}\right)  =\theta_{0}\left(  R_{0}\right)
=0$ and $\frac{\partial}{\partial s}G\left(  0,R_{0}\right)  =\theta
_{0}^{\prime}\left(  R_{0}\right)  <0$. By the Implicit Function Theorem,
there exists $\alpha_{0}>0$ such that when $\alpha\in\left(  0,\alpha
_{0}\right)  $, $G\left(  \alpha,s\right)  $ has a unique zero $S_{\alpha}$
near $R_{0}$. Then $S_{\alpha}$ is the support radius of $\theta_{\alpha}$.
Therefore, for any $0<\mu<\mu_{0}=F\left(  \alpha_{0}\right)  $, there exists
a unique non-rotating solution $y_{\mu}\left(  r\right)  $ defined by
(\ref{defn-y-mu-scaling}) with the support radius $R_{\mu}=\alpha
^{-\frac{n_{0}-1}{2}}S_{\alpha}$.
\end{proof}

Let
\[
\mu_{\max} = \sup\{ \mu\mid\exists\text{ solution } \rho_{\mu^{\prime}} \text{
is compactly supported}, \ \forall\mu^{\prime}\in(0, \mu] \} \in(0,+\infty].
\]
For any center density $\rho_{\mu}\left(  0\right)  =\mu\in\left(  0,\mu
_{\max}\right)  $,
let $R_{\mu}=R\left(  \mu\right)  <\infty$ be the support radius of the
density $\rho_{\mu}\left(  \left\vert x\right\vert \right)  $ of the unique
non-rotating stars and
\[
M\left(  \mu\right)  =\int_{R^{3}}\rho_{\mu}\ dx=\int_{\left\vert x\right\vert
<R_{\mu}}\rho_{\mu}\ dx
\]
to be the total mass of the star.

\begin{remark}
\label{R:equilibria} For Polytropic stars with $P\left(  \rho\right)
=K\rho^{\gamma}$ $\left(  \gamma>\frac{6}{5}\right)  $, we have $\mu_{\max
}=+\infty$. The scaling relation (\ref{transform-L_E}) implies the classical
formulae (\cite{chandra-book-stellar})%
\begin{equation}
M\left(  \mu\right)  =C_{1}\mu^{\frac{1}{2}\left(  3\gamma-4\right)
},\ \ R_{\mu}=C_{2}\mu^{\frac{1}{2}\left(  \gamma-2\right)  }.
\label{relation-M-R-polytropic}%
\end{equation}
for positive constants $C_{1},C_{2}$ depending only on $\gamma$.

For general equation of states satisfying (\ref{assumption-P1}) and
(\ref{assumption-P2'}) with $\gamma_{0}\geq\frac{4}{3}$, it was shown in
(\cite{uggala-Newtonian-stars}) that $\mu_{\max}=+\infty$. See also
\cite{Rein-existence} \cite{rein-rendall-compact} \cite{makino-84} for the
case $\gamma_{0}>\frac{4}{3}$. On the other hand, for $\gamma_{0}\in\left(
\frac{6}{5},\frac{4}{3}\right)  $, counterexamples of $P\left(  \rho\right)  $
with $\mu_{\max}<\infty$ were constructed in \cite{rein-rendall-compact}. For
physically realistic equation of states such as white dwarf stars, $\gamma
_{0}=\frac{5}{3}$ (see \cite{chandra-book-stellar} \cite{shapiro-book-85}).
\end{remark}

\subsection{\label{section-L-EP}Linearized Euler-Poisson equation}

We assume $P\left(  \rho\right)  $ satisfies (\ref{assumption-P1}%
)-(\ref{assumption-P2}). Near a non-rotating star $\left(  \rho_{\mu
},0\right)  $ with center density $\mu$, the linearized Euler-Poisson system
is
\begin{equation}
\sigma_{t}=-\nabla\cdot\left(  \rho_{\mu}v\right)
,\label{eqn-linearized-Continuity}%
\end{equation}%
\begin{equation}
v_{t}=-\nabla\left(  \Phi^{\prime\prime}\left(  \rho_{\mu}\right)
\sigma+V\right)  ,\ \ \ \label{eqn-linearized-Euler}%
\end{equation}
with $\Delta V=4\pi\rho$. Here, $\sigma,v\ $are the density and velocity
perturbations respectively.
In the linear approximation, we take
the density perturbation $\sigma$ and the velocity perturbation $v$ with the same support as
$\rho_{\mu}$, that is
\[
\text{supp}(\sigma), \ \text{supp}(v)\subset \overline{S_{\mu}}=\left\{  \left\vert
x\right\vert \le R_{\mu}\right\}.
\]
This is reasonable in the view of
the underlying Lagrangian formulation of the problem. See Appendix for more details.
Formally, the above linearized system
has an invariant energy functional
\begin{equation}
H_{\mu}\left(  \sigma,v\right)  =\frac{1}{2}\int_{S_{\mu}}\left(  \rho_{\mu
}\left\vert v\right\vert ^{2}+\Phi^{\prime\prime}\left(  \rho_{\mu}\right)
\sigma^{2}\right)  dx-\frac{1}{8\pi}\int_{\mathbf{R}^{3}}\left\vert \nabla
V\right\vert ^{2}dx.\label{defn-H-linearized-EP}%
\end{equation}
To ensure $H_{\mu}\left(  \sigma,v\right)  <\infty$, we consider the natural
energy space $X_{\mu}=L_{\Phi^{\prime\prime}\left(  \rho_{\mu}\right)  }^{2}$
for $\sigma$ and $Y_{\mu}=\left(  L_{\rho_{\mu}}^{2}\right)  ^{3}$ for $v$.
Here, $L_{\Phi^{\prime\prime}\left(  \rho_{\mu}\right)  }^{2},L_{\rho_{\mu}%
}^{2}$ are the $\Phi^{\prime\prime}\left(  \rho_{\mu}\right)  $ $,\rho_{\mu}$
weighted $L^{2}$ spaces in $S_{\mu}$  and thus (\ref{eqn-linearized-Continuity})-(\ref{eqn-linearized-Euler})
form a linear evolution system on $X_{\mu} \times Y_\mu$.
For $\sigma\in L_{\Phi^{\prime\prime
}\left(  \rho_{\mu}\right)  }^{2}$, we have
\begin{align}
\int_{\mathbf{R}^{3}}\left\vert \nabla V\right\vert ^{2}dx &  =-4\pi
\int_{S_{\mu}}\rho Vdx\leq4\pi\left\Vert \sigma\right\Vert _{L_{\Phi
^{\prime\prime}\left(  \rho_{\mu}\right)  }^{2}}\left(  \int_{S_{\mu}}%
\frac{V^{2}}{\Phi^{\prime\prime}\left(  \rho_{\mu}\right)  }dx\right)
^{\frac{1}{2}}\label{estimate-H1-dot}\\
&  \lesssim\left\Vert \sigma\right\Vert _{L_{\Phi^{\prime\prime}\left(
\rho_{\mu}\right)  }^{2}}\left\Vert V\right\Vert _{L^{6}\left(  \mathbf{R}%
^{3}\right)  }\lesssim\left\Vert \sigma\right\Vert _{L_{\Phi^{\prime\prime
}\left(  \rho_{\mu}\right)  }^{2}}\left\Vert \nabla V\right\Vert
_{L^{2}\left(  \mathbf{R}^{3}\right)  }\nonumber
\end{align}
and thus $\left\Vert \nabla V\right\Vert _{L^{2}\left(  \mathbf{R}^{3}\right)
}\lesssim\left\Vert \sigma\right\Vert _{L_{\Phi^{\prime\prime}\left(
\rho_{\mu}\right)  }^{2}}$. In above estimates, we use the fact that $\frac
{1}{\Phi^{\prime\prime}\left(  \rho_{\mu}\right)  }$ is bounded in
$\overline{S_{\mu}}$ since $\frac{1}{\Phi^{\prime\prime}\left(  \rho_{\mu
}\right)  }\thickapprox\rho_{\mu}^{2-\gamma_{0}}$ $\left(  \gamma
_{0}<2\right)  \ $for $\rho_{\mu}\ll1$. The notation $P\lesssim Q$ means
$P\leq C_{\mu}Q$ for some constant $C_{\mu}$ depending only on $\mu$.

\begin{remark}
Since $\gamma_0 \in (\frac 65, 2)$ and
\[
\rho_\mu = O\big((R_\mu -r)^{\frac 1{\gamma_0-1}}\big), \quad \Phi''\big(\rho_\mu(r)\big) = O\big((R_\mu -r)^{\frac {\gamma_0-2}{\gamma_0-1}}\big),
\]
in such weighted spaces, as $r \to R_\mu -$, $v \in Y_\mu$ allows $v$ to approach infinity, while $\sigma \in X_\mu$ may approach infinity for $\gamma_0 \in (\frac 32, 2)$ or has to satisfy $\lim \inf_{r\to R_\mu-} \sigma (r) =0$ for $\gamma_0 \in (\frac 65, \frac 32]$. Recalling that supp$(\rho)$ is the domain occupied by the fluid, the vanishing of $\sigma(R_\mu -)$ in the latter case does not mean that the domain does not evolve, but only not reflected in the linear order of the density perturbation due to its degeneracy near the boundary for $\gamma_0 \in (\frac 65, \frac 32]$. In fact, the variation of the domain is clearly indicated in that $v$ does not have to vanish near $r=R_\mu$.
\end{remark}

Define the operators
\[
L_{\mu}=\Phi^{\prime\prime}\left(  \rho_{\mu}\right)  -4\pi\left(
-\Delta\right)  ^{-1}:X_{\mu}\rightarrow X_{\mu}^{\ast},\ \ A_{\mu}=\rho_{\mu
}:Y_{\mu}\rightarrow Y_{\mu}^{\ast}%
\]
and
\[
B_{\mu}=-\nabla\cdot=-\operatorname{div}:Y_{\mu}^{\ast}\rightarrow X_{\mu
},\ \ \ B_{\mu}^{\prime}=\nabla:X_{\mu}^{\ast}\rightarrow Y_{\mu}.
\]
Here, for $\sigma\in X_{\mu}$, we denote
\[
\left(  -\Delta\right)  ^{-1}\sigma=\int_{S_{\mu}}\frac{1}{4\pi\left\vert
x-y\right\vert }\sigma\left(  y\right)  dy\ |_{S_{\mu}}\text{. }%
\]
Then the linearized system (\ref{eqn-linearized-Continuity}%
)-(\ref{eqn-linearized-Euler}) can be written in a separable Hamiltonian form
\[
\partial_{t}\left(
\begin{array}
[c]{c}%
\sigma\\
v
\end{array}
\right)  =\left(
\begin{array}
[c]{cc}%
0 & B_{\mu}\\
-B_{\mu}^{\prime} & 0
\end{array}
\right)  \left(
\begin{array}
[c]{cc}%
L_{\mu} & 0\\
0 & A_{\mu}%
\end{array}
\right)  \left(
\begin{array}
[c]{c}%
\sigma\\
v
\end{array}
\right)  =\mathcal{J}_{\mu}\mathcal{L}_{\mu}\left(
\begin{array}
[c]{c}%
\sigma\\
v
\end{array}
\right)  ,
\]
which will be checked to satisfy assumptions (\textbf{G1-4}) in the general
framework of Section 2. First, (\textbf{G2}) is obvious for the operator
$A_{\mu}$ defined in (\ref{defn-L-A}) with $\ker A_{\mu}=\{0\}$. We note that
\[
S_{1}=\sqrt{\rho_{\mu}}:\left(  L^{2}\left(  S_{\mu}\right)  \right)
^{3}\rightarrow Y_{\mu}^{\ast}=\left(  L_{\frac{1}{\rho_{\mu}}}^{2}\right)
^{3},\ \ S_{2}=\sqrt{\Phi^{\prime\prime}\left(  \rho_{\mu}\right)  }:X_{\mu
}\rightarrow L^{2}\left(  S_{\mu}\right)
\]
are isomorphisms. Therefore, to show $B_{\mu}:$ $Y_{\mu}^{\ast}\rightarrow
X_{\mu}$ is densely defined and closed, it is equivalent to check
\[
\tilde{B}_{\mu}=S_{2}BS_{1}=-\sqrt{\Phi^{\prime\prime}\left(  \rho_{\mu
}\right)  }\operatorname{div}\left(  \sqrt{\rho_{\mu}}\cdot\right)  :\left(
L^{2}\left(  S_{\mu}\right)  \right)  ^{3}\rightarrow L^{2}\left(  S_{\mu
}\right)
\]
is densely defined and closed. The domain of $\tilde{B}_{\mu}$ is
\[
D\left(  \tilde{B}_{\mu}\right)  =\left\{  u\in\left(  L^{2}\left(  S_{\mu
}\right)  \right)  ^{3}|\ \sqrt{\Phi^{\prime\prime}\left(  \rho_{\mu}\right)
}\nabla\cdot\left(  \sqrt{\rho_{\mu}}u\right)  \in L^{2}\ \text{in the
distribution sense}\right\}  .
\]
It is clear that any $C^{1}$ function with compact support inside $S_{\mu}$ is
in $D\left(  \tilde{B}_{\mu}\right)  $, thus $D\left(  \tilde{B}_{\mu}\right)
$ is dense in $\left(  L^{2}\left(  S_{\mu}\right)  \right)  ^{3}$. Define
\[
\tilde{C}_{\mu}=\sqrt{\rho_{\mu}}\nabla\left(  \sqrt{\Phi^{\prime\prime
}\left(  \rho_{\mu}\right)  }\cdot\right)  :L^{2}\left(  S_{\mu}\right)
\rightarrow\left(  L^{2}\left(  S_{\mu}\right)  \right)  ^{3},
\]
with
\[
D\left(  \tilde{C}_{\mu}\right)  =\left\{  \sigma\in L^{2}\left(  S_{\mu
}\right)  |\ \sqrt{\rho_{\mu}}\nabla\left(  \sqrt{\Phi^{\prime\prime}\left(
\rho_{\mu}\right)  }\sigma\right)  \in\left(  L^{2}\left(  S_{\mu}\right)
\right)  ^{3}\right\}  .
\]
Then $\tilde{C}_{\mu}$ is also densely defined.

\begin{lemma}
\label{L:duality} The above defined operators satisfy $\tilde{C}_{\mu}%
=\tilde{B}_{\mu}^{\ast}$ and $\tilde{B}_{\mu}=\left(  \tilde{C}_{\mu}\right)
^{\ast}=(\tilde{B}_{\mu})^{\ast\ast}$. Thus $\tilde{B}_{\mu}$ and $\tilde
{B}_{\mu}^{\ast}$ are both closed.
\end{lemma}

\begin{proof}
We start the proof of the lemma with a basic property of functions in
$D(\tilde{C}_{\mu})$. Namely, for any $f\in D(\tilde{C}_{\mu})$, there exists
$M>0$, such that for any $r\in(\frac{1}{2}R_{\mu},R_{\mu})$, it holds that
\begin{equation}
\Vert\sqrt{\rho_{\mu}\Phi^{\prime\prime}(\rho_{\mu})}f\Vert_{L^{2}(\partial
S(r))}\leq M(R_{\mu}-r)^{\frac{1}{2}}, \label{E:DtCmu}%
\end{equation}
where $\partial S(r)$ is the sphere with radius $r$. In fact, by the
definition of $D(\tilde{C}_{\mu})$, the trace of $f$ on any sphere $\partial
S(r)$ with radius $r<R_{\mu}$ belongs to $L^{2}\big(\partial S(r)\big)$ and
\[
g\triangleq\sqrt{\rho_{\mu}}\partial_{r}\left(  \sqrt{\Phi^{\prime\prime
}\left(  \rho_{\mu}\right)  }f\right)  \in L^{2}(S_{\mu}).
\]
Since for any $\theta\in S^{2}$,
\[
\left(  \sqrt{\Phi^{\prime\prime}\left(  \rho_{\mu}\right)  }f\right)
(r\theta)=\left(  \sqrt{\Phi^{\prime\prime}\left(  \rho_{\mu}\right)
}f\right)  (\frac{1}{2}R_{\mu}\theta)+\int_{\frac{1}{2}R_{\mu}}^{r}(\rho_{\mu
}^{-\frac{1}{2}}g)(r^{\prime}\theta)dr^{\prime},
\]
it follows that
\begin{align*}
\Vert\sqrt{\Phi^{\prime\prime}\left(  \rho_{\mu}\right)  }f\Vert
_{L^{2}(\partial S(r))}\leq &  M\left(  1+\Vert g\Vert_{L^{2}(S_{\mu})}\left(
\int_{\frac{1}{2}R_{\mu}}^{r}\left(  R_{\mu}-r^{\prime}\right)  ^{\frac
{1}{1-\gamma_{0}}}dr^{\prime}\right)  ^{\frac{1}{2}}\right) \\
\leq &  M\big(1+(R_{\mu}-r)^{\frac{2-\gamma_{0}}{2(1-\gamma_{0})}}\big)
\end{align*}
and thus \eqref{E:DtCmu} follows.

By the definition of adjoint operators, $f\in D(\tilde{B}_{\mu}^{\ast})\subset
L^{2}(S_{\mu})$ and $w=\tilde{B}_{\mu}^{\ast}f$ if and only if, for any $v\in
D(\tilde{B}_{\mu})$,
\begin{equation}
\int_{S_{\mu}}w\cdot vdx=\langle f,\tilde{B}_{\mu}v\rangle=-\int_{S_{\mu}%
}\sqrt{\Phi^{\prime\prime}\left(  \rho_{\mu}\right)  }f\nabla\cdot\left(
\sqrt{\rho_{\mu}}v\right)  dx. \label{E:dual1}%
\end{equation}
By taking compacted supported $v$ and integrating by parts, we obtain that
$f\in D(\tilde{C}_{\mu}^{\ast})$ and $w=\tilde{C}_{\mu}f$ is necessary. To
show this is also sufficient, for any $v\in D(\tilde{B}_{\mu})$, we integrate
on smaller balls and take the limit,
\begin{align*}
&  -\int_{S_{\mu}}\sqrt{\Phi^{\prime\prime}\left(  \rho_{\mu}\right)  }%
f\nabla\cdot\left(  \sqrt{\rho_{\mu}}v\right)  dx=-\lim_{n\rightarrow\infty
}\int_{S(R_{\mu}-\epsilon_{n})}\sqrt{\Phi^{\prime\prime}\left(  \rho_{\mu
}\right)  }f\nabla\cdot\left(  \sqrt{\rho_{\mu}}v\right)  dx\\
=  &  \langle f,\tilde{C}_{\mu}v\rangle-\lim_{n\rightarrow\infty}%
\int_{\partial S(R_{\mu}-\epsilon_{n})}\sqrt{\Phi^{\prime\prime}\left(
\rho_{\mu}\right)  \rho_{\mu}}f\ v\cdot\frac{x}{R_{\mu}-\epsilon_{n}}dS,
\end{align*}
where $\epsilon_{n}\rightarrow0+$. According to \eqref{E:DtCmu},
\[
\left\vert \int_{\partial S(R_{\mu}-\epsilon_{n})}\sqrt{\Phi^{\prime\prime
}\left(  \rho_{\mu}\right)  \rho_{\mu}}f\ v\cdot\frac{x}{R_{\mu}-\epsilon_{n}%
}dS\right\vert \leq M\epsilon_{n}^{\frac{1}{2}}\Vert v\Vert_{L^{2}(\partial
S(R_{\mu}-\epsilon_{n}))}.
\]
Since $v\in L^{2}(S_{\mu})$, there exist a sequence $\epsilon_{n}%
\rightarrow0+$ such that
\[
\epsilon_{n}^{\frac{1}{2}}\Vert v\Vert_{L^{2}(\partial S(R_{\mu}-\epsilon
_{n}))}\rightarrow0
\]
and thus \eqref{E:dual1} holds which implies $\tilde{B}_{\mu}^{\ast}=\tilde
{C}_{\mu}$.

Much as in the above, $\tilde{C}_{\mu}^{\ast}=\tilde{B}_{\mu}$ and this
completes the proof of the lemma.
\end{proof}

We now check that $L_{\mu}$ defined by (\ref{defn-L-A}) satisfies
(\textbf{G3}). Let $I_{X_{\mu}}=\frac{1}{\Phi^{\prime\prime}\left(  \rho_{\mu
}\right)  }:X_{\mu}^{\ast}\rightarrow X_{\mu}$ be the isomorphism from Riesz
representation theorem, and define the operator
\begin{equation}
\mathbb{L}_{\mu}\mathbb{=}I_{X_{\mu}}L_{\mu}=Id-\frac{1}{4\pi\Phi
^{\prime\prime}\left(  \rho_{\mu}\right)  }\left(  -\Delta\right)
^{-1}:X_{\mu}\rightarrow X_{\mu}. \label{defn-L-selfadjoint}%
\end{equation}

\begin{lemma}
\label{lemma-L-self-adjoint}$\mathbb{L}$ is bounded and self-adjoint on
$X_{\mu}\ $and $\mathbb{L}_{\mu}\mathbb{-}Id$ is compact.
\end{lemma}

\begin{proof}
Let
\[
\mathbb{K=L}_{\mu}\mathbb{-}Id=-\frac{1}{4\pi\Phi^{\prime\prime}\left(
\rho_{\mu}\right)  }\left(  -\Delta\right)  ^{-1}:X_{\mu}\rightarrow X_{\mu}.
\]
We first show that $\mathbb{K}$ is compact. Indeed, for any $\sigma\in X_{\mu
}$, we have
\[
\left\Vert \mathbb{K}\sigma\right\Vert _{X_{\mu}}=\left(  \int_{S_{\mu}}%
\frac{V^{2}}{\Phi^{\prime\prime}\left(  \rho_{\mu}\right)  }dx\right)
^{\frac{1}{2}}\lesssim\left(  \int_{S_{\mu}}V^{2}dx\right)  ^{\frac{1}{2}},
\]
where $\Delta V=4\pi\rho$. By the previously established estimate $\left\Vert
V\right\Vert _{\dot{H}^{1}}\lesssim\left\Vert \sigma\right\Vert _{X_{\mu}}$
and the compactness of $\dot{H}^{1}\left(  \mathbf{R}^{3}\right)  $ to
$L^{2}\left(  S_{\mu}\right)  $, the compactness of $\mathbb{K}$ follows.
Since $\mathbb{K}$ is symmetric on $X_{\mu}$, the self-adjointness of
$\mathbb{L}_{\mu}$ follows from the Kato-Rellich theorem.
\end{proof}

Assumption (\textbf{G3}) readily follows from above lemma. To compute
$n^{-}\left(  L_{\mu}|_{X_{\mu}}\right)  $, we define the elliptic operator
\[
D_{\mu}=-\Delta-4\pi F_{+}^{\prime}\left(  V_{\mu}\left(  R_{\mu}\right)
-V_{\mu}\right)  :\dot{H}^{1}\left(  \mathbf{R}^{3}\right)  \rightarrow\dot
{H}^{-1}\left(  \mathbf{R}^{3}\right)  .
\]
Then for $\phi\in\dot{H}^{1}\left(  \mathbf{R}^{3}\right)  $,
\[
\left\langle D_{\mu}\phi,\phi\right\rangle =\int_{\mathbf{R}^{3}}\left\vert
\nabla\phi\right\vert ^{2}dx-4\pi\int_{S_{\mu}}F^{\prime}\left(  V_{\mu
}\left(  R_{\mu}\right)  -V_{\mu}\right)  \left\vert \phi\right\vert ^{2}dx
\]
defines a bounded bilinear symmetric form on $\dot{H}^{1}\left(
\mathbf{R}^{3}\right)  $. The next lemma shows that the study of the quadratic
form
\[
\left\langle L_{\mu}\sigma,\sigma\right\rangle =\int_{S_{\mu}}\Phi
^{\prime\prime}\left(  \rho_{\mu}\right)  \sigma^{2}-\frac{1}{4\pi}%
\int_{\mathbf{R}^{3}}\left\vert \nabla V\right\vert ^{2}dx,\quad\sigma\in
X_{\mu},
\]
can be reduced to study $D_{\mu}$ on $\dot{H}^{1}\left(  \mathbf{R}%
^{3}\right)  $.

\begin{lemma}
\label{lemma-equivalent-operators}It holds that $n^{-}\left(  L_{\mu}%
|_{X_{\mu}}\right)  =n^{-}\left(  \mathbb{L}_{\mu}\right)  =n^{-}\left(
D_{\mu}\right)  $ and $\dim\ker L_{\mu}=\dim\ker\mathbb{L}_{\mu}=\dim\ker
D_{\mu}$.
\end{lemma}

\begin{proof}
The proof of the lemma is largely based on the observation $D_{\mu}=
F^{\prime}L_{\mu}(-\Delta)$ in $S_{\mu}$.

First, for any $\rho\in X_{\mu}$, we can show that
\[
\left\langle L_{\mu}\rho,\rho\right\rangle \geq\frac{1}{4\pi}\left(  D_{\mu
}V,V\right)  ,\ \ \ \Delta V=4\pi\rho.
\]
Indeed, inside $S_{\mu}$ we have $F^{\prime}\left(  V_{\mu}\left(  R_{\mu
}\right)  -V_{\mu}\right)  =\frac{1}{\Phi^{\prime\prime}\left(  \rho_{\mu
}\right)  }$. Then
\begin{align*}
\left\langle L_{\mu}\rho,\rho\right\rangle  &  =\int_{S_{\mu}}\frac
{1}{F^{\prime}}\rho^{2}dx-\frac{1}{4\pi}\int_{\mathbf{R}^{3}}\left\vert \nabla
V\right\vert ^{2}dx\\
&  =\int_{\mathbf{R}^{3}}\frac{1}{4\pi}\left\vert \nabla V\right\vert
^{2}dx+\int_{S_{\mu}}\left(  2V\rho+\frac{1}{F^{\prime}}\rho^{2}\right)  dx\\
&  \geq\int\left(  \frac{1}{4\pi}\left\vert \nabla V\right\vert ^{2}%
-F^{\prime}V^{2}\right)  dx=\frac{1}{4\pi}\left\langle D_{\mu}V,V\right\rangle
.
\end{align*}
Denote $n^{\leq0}\left(  L_{\mu}\right)  $ and $n^{\leq0}\left(  D_{\mu
}\right)  $ to be the maximal dimension of non-positive subspaces of $L_{\mu}$
and $D_{\mu}$ respectively. Then above inequality implies that $n^{\leq
0}\left(  L_{\mu}\right)  \leq n^{\leq0}\left(  D_{\mu}\right)  $. Second, for
any $\phi\in\dot{H}^{1}\left(  \mathbf{R}^{3}\right)  $, let $\rho_{\phi
}=F_{+}^{\prime}\phi\in X_{\mu}$ and $\Delta V_{\phi}=4\pi\rho_{\phi}$. Then
\begin{align*}
\left\langle D_{\mu}\phi,\phi\right\rangle  &  =\int_{\mathbf{R}^{3}%
}\left\vert \nabla\phi\right\vert ^{2}dx-4\pi\int_{S_{\mu}}F^{\prime
}\left\vert \phi\right\vert ^{2}dx\\
&  =4\pi\left(  \int_{S_{\mu}}\frac{\left\vert \rho_{\phi}\right\vert ^{2}%
}{F^{\prime}}dx+\frac{1}{4\pi}\int_{\mathbf{R}^{3}}\left\vert \nabla
\phi\right\vert ^{2}dx-2\int_{S_{\mu}}\rho_{\phi}\bar{\phi}\ dx\right) \\
&  =4\pi\left(  \int_{S_{\mu}}\frac{\left\vert \rho_{\phi}\right\vert ^{2}%
}{F^{\prime}}dx+\frac{1}{4\pi}\int_{\mathbf{R}^{3}}\left\vert \nabla
\phi\right\vert ^{2}dx-\frac{1}{2\pi}\int_{\mathbf{R}^{3}}\nabla V_{\phi}%
\cdot\nabla\bar{\phi}\ dx\right) \\
&  \geq4\pi\left(  \int_{S_{\mu}}\frac{\left\vert \rho_{\phi}\right\vert ^{2}%
}{F^{\prime}}dx-\frac{1}{4\pi}\int_{\mathbf{R}^{3}}\left\vert \nabla V_{\phi
}\right\vert ^{2}dx\right)  =4\pi\left\langle L_{\mu}\rho_{\phi},\rho_{\phi
}\right\rangle \text{. }%
\end{align*}
Thus $n^{\leq0}\left(  L_{\mu}\right)  \geq n^{\leq0}\left(  D_{\mu}\right)  $
and a combination with the previous inequality yields
\begin{equation}
n^{\leq0}\left(  L_{\mu}\right)  =n^{\leq0}\left(  D_{\mu}\right)  .
\label{equality-non-negative-dim}%
\end{equation}
We note that: $L_{\mu}\rho=0$ for $\rho\in X_{\mu}$ is equivalent to $D_{\mu
}V=0$ where $\Delta V=4\pi\rho$, and $D_{\mu}\phi=0$ for $\phi\in\dot{H}^{1}$
is equivalent to $L_{\mu}\rho_{\phi}=0$ $\left(  \rho_{\phi}=F_{+}^{\prime
}\phi\right)  $. Thus we have $\dim\ker L_{\mu}=\dim\ker D_{\mu}$ and
consequently $n^{-}\left(  L_{\mu}\right)  =n^{-}\left(  D_{\mu}\right)  $
follows from (\ref{equality-non-negative-dim}).
\end{proof}

In the rest of this subsection, we study some basic properties of the operator
$D_{\mu}$. Since the potential term in $D_{\mu}$ is radially symmetric, we can
use spherical harmonic functions to decompose $D_{\mu}$ into operators on
radially symmetric spaces. Let $Y_{lm}\left(  \theta\right)  $ be the standard
spherical harmonics on $\mathbb{S}^{2}$ where $l=0,1,\cdots;m=-l,\cdots,l$.
Then $\Delta_{\mathbb{S}^{2}}Y_{lm}=-l\left(  l+1\right)  Y_{lm}$. For any
function $u\left(  x\right)  \in\dot{H}^{1}$, we decompose
\[
u\left(  x\right)  =\sum_{l=0}^{\infty}\sum_{m=-l}^{l}u_{lm}\left(  r\right)
Y_{lm}\left(  \theta\right)  ,\ \ \ u_{lm}\left(  r\right)  =\int%
_{\mathbb{S}^{2}}u\left(  r\theta\right)  Y_{lm}\left(  \theta\right)
dS_{\theta}\text{. }%
\]
Then we have
\[
D_{\mu}u=\sum_{l=0}^{\infty}\sum_{m=-l}^{l}D_{\mu}^{l}u_{lm}\left(  r\right)
\ Y_{lm}\left(  \theta\right)  ,
\]
where
\begin{equation}
D_{\mu}^{l}=-\Delta_{r}+\frac{l\left(  l+1\right)  }{r^{2}}-4\pi F_{+}%
^{\prime}\left(  V_{\mu}\left(  R_{\mu}\right)  -V_{\mu}\left(  r\right)
\right)  , \label{defn-D-mu-l}%
\end{equation}
and $\Delta_{r}=\frac{d^{2}}{dr^{2}}+\frac{2}{r}\frac{d}{dr}$. In particular,
the operator
\begin{equation}
D_{\mu}^{0}=-\Delta_{r}-4\pi F_{+}^{\prime}\left(  V_{\mu}\left(  R_{\mu
}\right)  -V_{\mu}\left(  r\right)  \right)  \label{defn-D-mu-0-general}%
\end{equation}
is $D_{\mu}$ restricted to radial functions.

The study of $D_{\mu}$ is reduced to the study of operators $D_{\mu}%
^{l}\ \left(  l\geq0\right)  $ for radial functions.

\begin{lemma}
\label{lemma-property-A-EP}

i) $\ker D_{\mu}^{1}=\left\{  V_{\mu}^{\prime}\left(  r\right)  \right\}  $
and $D_{\mu}^{1}\geq0.$

ii) For $l\geq2,\ D_{\mu}^{l}>0$.

iii) $n^{-}\left(  D_{\mu}\right)  =n^{-}\left(  D_{\mu}^{0}\right)  \geq1.$
\end{lemma}

\begin{proof}
The arguments are rather standard. Taking $\partial_{x_{i}}$ of the steady
equation
\begin{equation}
\Delta V_{\mu}=V_{\mu}^{\prime\prime}+\frac{2}{r}V_{\mu}^{\prime}=4\pi
F_{+}\left(  V_{\mu}\left(  R_{\mu}\right)  -V_{\mu}\left(  r\right)  \right)
, \label{eqn-potential-steady}%
\end{equation}
we get $D_{\mu}\partial_{x_{i}}V_{\mu}=0,\ i=1,2,3$. Thus $D_{\mu}^{1}V_{\mu
}^{\prime}\left(  r\right)  =0$. Since $V_{\mu}^{\prime}\left(  r\right)  >0$
for $r>0$, i) follows from the Sturm-Liouville theory for the ODE operator
$D_{\mu}^{1}$. Then for $l\geq2$,
\[
D_{\mu}^{l}=D_{\mu}^{1}+\frac{l\left(  l+1\right)  -2}{r^{2}}>0.
\]
By i) and ii), we have $n^{-}\left(  D_{\mu}\right)  =n^{-}\left(  D_{\mu}%
^{0}\right)  $. Since $D_{\mu}\partial_{x_{i}}V_{\mu}=0$ and $\partial_{x_{i}%
}V_{\mu}$ changes sign, $0$ can not be the first eigenvalue of $D_{\mu}$. Thus
$n^{-}\left(  D_{\mu}\right)  \geq1.$ This proves iii).
\end{proof}

\subsection{The negative index of $D_{\mu}$}

We find the negative index $n^{-}\left(  D_{\mu}\right)  =n^{-}\left(  D_{\mu
}^{0}\right)  $ in this subsection. Although $D_{\mu}$ is defined as an
operator $\dot{H}^{1}\rightarrow\dot{H}^{-1}$, the eigenfunctions with
negative eigenvalues of $D_{\mu}$ decay exponentially fast at infinity and are
in $H^{2}$. Thus, when computing $n^{-}\left(  D_{\mu}\right)  $ below, we can
treat $D_{\mu}$ as an operator $H^{2}\rightarrow L^{2}$ and $D_{\mu}^{0}%
:H_{r}^{2}\rightarrow L_{r}^{2}$.

The following formula for the surface potential $V_{\mu}\left(  R_{\mu
}\right)  $ will be used later.

\begin{lemma}
\label{lemma-formula-V-R}It holds that
\begin{equation}
V_{\mu}\left(  R_{\mu}\right)  =-\frac{M\left(  \mu\right)  }{R_{\mu}}.
\label{formula-V-R}%
\end{equation}

\end{lemma}

\begin{proof}
Since%
\[
V_{\mu}^{\prime\prime}+\frac{2}{r}V_{\mu}^{\prime}=\frac{1}{r^{2}}\frac{d}%
{dr}\left(  r^{2}V_{\mu}^{\prime}\left(  r\right)  \right)  =4\pi\rho_{\mu},
\]
we have
\begin{equation}
V_{\mu}^{\prime}\left(  r\right)  =\frac{4\pi}{r^{2}}\int_{0}^{r}\rho_{\mu
}\left(  r\right)  r^{2}dr=\frac{M\left(  \mu\right)  }{r^{2}},\ \text{\ for
}r\geq R_{\mu}. \label{V-R-prime-1}%
\end{equation}
Thus
\[
V_{\mu}\left(  r\right)  =-\frac{M\left(  \mu\right)  }{r},\ \ \ \ \text{for
}\ r\geq R_{\mu},
\]
and formula (\ref{formula-V-R}) follows.
\end{proof}




To find $n^{-}\left(  D_{\mu}^{0}\right)  $, our key observation is that
$D_{\mu}^{0}$ has a kernel only at critical points of the surface potential
$V_{\mu}\left(  R_{\mu}\right)  $, or equivalently at points where $\frac
{d}{d\mu}\left(  \frac{M\left(  \mu\right)  }{R_{\mu}}\right)  =0\ $by above lemma.

\begin{lemma}
\label{lemma-kernel-D-mu}When $\frac{d}{d\mu}\left(  \frac{M\left(
\mu\right)  }{R_{\mu}}\right)  \neq0$, $\ker D_{\mu}^{0}=\left\{  0\right\}
$; When $\frac{d}{d\mu}\left(  \frac{M\left(  \mu\right)  }{R_{\mu}}\right)
=0$, $\ker D_{\mu}^{0}=\left\{  \frac{\partial}{\partial\mu}V_{\mu}\right\}  $.
\end{lemma}

\begin{proof}
Let $y_{\mu}\left(  r\right)  =V_{\mu}\left(  R_{\mu}\right)  -V_{\mu}\left(
r\right)  $, then
\[
\Delta_{r}y_{\mu}=y_{\mu}^{\prime\prime}+\frac{2}{r}y_{\mu}^{\prime}=-4\pi
F_{+}\left(  y_{\mu}\left(  r\right)  \right)  .
\]
Observing that $F_{+}$ is actually a $C^{1}$ function for $\gamma\in(\frac65,
2)$, denote $u_{\mu}\left(  r\right)  =\frac{\partial}{\partial\mu}y_{\mu
}\left(  r\right)  $ and by taking $\frac{\partial}{\partial\mu}$ of above
equation for $y_{\mu}$, we get
\begin{equation}
u_{\mu}^{\prime\prime}+\frac{2}{r}u_{\mu}^{\prime}=-4\pi F_{+}^{\prime}\left(
y_{\mu}\left(  r\right)  \right)  u_{\mu}. \label{ode in ball}%
\end{equation}
Suppose $D_{\mu}^{0}v\left(  r\right)  =0$ with $v\left(  \left\vert
x\right\vert \right)  \in\dot{H}^{1}\left(  \mathbf{R}^{3}\right)  $. Then
\begin{equation}
v^{\prime\prime}+\frac{2}{r}v^{\prime}=\frac{1}{r^{2}}\frac{d}{dr}\left(
r^{2}v^{\prime}\left(  r\right)  \right)  =-4\pi F_{+}^{\prime}\left(  y_{\mu
}\left(  r\right)  \right)  v\left(  r\right)  \label{ode v}%
\end{equation}
and
\[
v^{\prime}\left(  r\right)  =-\frac{4\pi}{r^{2}}\int_{0}^{r}s^{2}
F_{+}^{\prime}\left(  y_{\mu}\left(  s\right)  \right)  v\left(  s\right)
ds,
\]
which implies that $v\in C^{1}\left(  0,+\infty\right)  $. Since both $u_{\mu
}\left(  r\right)  $ and $v\left(  r\right)  $ satisfy the same 2nd order ODE
(\ref{ode in ball}) and (\ref{ode v})
with zero derivative at $r=0$, we have $v\left(  r\right)  =Cu_{\mu}\left(
r\right)  $ for some constant $C\neq0$.
It implies $u_{\mu}\in\dot{H}^{1}\left(  \mathbf{R}^{3}\right)  $ harmonic
outside $S_{\mu}$. Along with $\lim_{r\to\infty} V(r)=0$ we obtain
\[
0= \lim_{r\to+\infty} u_{\mu}(r) = \frac{d}{d\mu}\left(  \frac{M\left(
\mu\right)  }{R_{\mu}}\right)  .
\]
Therefore, $D_{\mu}^{0}$ has a kernel only when $\frac{d}{d\mu}\left(
\frac{M\left(  \mu\right)  }{R_{\mu}}\right)  =0$, and in this case it follows
from above analysis that $\ker D_{\mu}^{0}=\left\{  \frac{\partial}%
{\partial\mu}V_{\mu}\right\}  $.
\end{proof}

To find $n^{-}\left(  D_{\mu}^{0}\right)  $, we use a continuity approach to
follow its changes when $\mu$ is increased from $0$ to $\mu_{\max}$. First, we
find $n^{-}\left(  D_{\mu}^{0}\right)  $ for small $\mu$. By above lemma, for
increasing $\mu$, the negative index $n^{-}\left(  D_{\mu}^{0}\right)  $ can
only change at critical points of $\frac{M\left(  \mu\right)  }{R\left(
\mu\right)  }$. Then we find the jump formula of $n^{-}\left(  D_{\mu}%
^{0}\right)  $ at those critical points. Combining these steps, we get
$n^{-}\left(  D_{\mu}^{0}\right)  $ for any $\mu>0$.

By the proof of Lemma \ref{lemma-steady-small-density}, for small $\mu$ the
steady state $\rho_{\mu}\ $is close (up to a scaling) to the Lane-Emden stars.
So we first find $n^{-}\left(  D_{\mu}^{0}\right)  $ for Lane-Emden stars. We
treat the case $\gamma\in(\frac{6}{5},\frac{4}{3})$ and $\gamma\in\lbrack
\frac{4}{3},2)$ separately.

\begin{lemma}
\label{lemma-Morse-3-below}Let $P\left(  \rho\right)  =K\rho^{\gamma}$,
$\gamma\in(\frac65, 2)$, then $n^{-}\left(  D_{\mu}^{0}\right)  =1$ for any
$\mu>0$.
\end{lemma}

\begin{proof}
Let $y_{\mu}\left(  r\right)  $ be the solution of (\ref{ode-y-LE}) with
$y_{\mu}\left(  0\right)  =\alpha=$ $\Phi^{\prime}\left(  \mu\right)  $.
Recall that $y_{\mu}\left(  r\right)  =\alpha\theta\left(  \alpha^{\frac
{n-1}{2}}r\right)  $, where $\theta\left(  s\right)  $ is the Lane-Emden
function satisfying (\ref{lane-emden equation}). Then
\[
D_{\mu}^{0}=-\Delta_{r}-C_{\gamma}n\left(  y_{\mu}\right)  _{+}^{n-1}%
,\ \ \ n=\frac{1}{\gamma-1}.
\]
Let $\psi\left(  r\right)  $ be an eigenfunction satisfying $D_{\mu}^{0}%
\psi=\lambda\psi$ with $\lambda<0$. Define $\psi\left(  r\right)  =\phi\left(
\alpha^{\frac{n-1}{2}}r\right)  $ and $s=$ $\alpha^{\frac{n-1}{2}}r$. Then
$\phi\left(  s\right)  $ satisfies the equation
\[
\left(  -\Delta_{s}-C_{\gamma}n\theta_{+}^{n-1}\right)  \phi=\alpha^{-\left(
n-1\right)  }\lambda\phi.
\]
Thus $n^{-}\left(  D_{\mu}^{0}\right)  =n^{-}\left(  B_{n}\right)  $, where
\begin{equation}
B_{n}=-\Delta_{s}-C_{\gamma}n\theta_{+}^{n-1}. \label{defn-B-n}%
\end{equation}
It suffices to show that $n^{-}\left(  B_{n}\right)  =1$.

We first consider the case $\gamma\in(\frac43, 2)$ where $n \in(1, 3]$. Define
$\theta_{a}\left(  s\right)  =a\theta\left(  a^{\frac{n-1}{2}}s\right)  $
$\left(  a>0\right)  \ $and
\[
w\left(  s\right)  =\frac{d}{da}\left(  \theta_{a}\left(  s\right)  \right)
|_{a=1}=\theta\left(  s\right)  +\frac{n-1}{2}s\theta^{\prime}\left(
s\right)  .
\]
Note that $\theta_{a}\left(  s\right)  $ satisfies the Lane-Emden equation
\begin{equation}
\theta_{a}^{\prime\prime}+\frac{2}{s}\theta_{a}^{\prime}=-C_{\gamma}%
\theta_{a,+}^{n},\ \ \theta_{a}\left(  0\right)  =a,\ \theta_{a}^{\prime
}\left(  0\right)  =0.\ \label{eqn-LE-a}%
\end{equation}
Let $R_{n}$ be the support radius of $\theta\left(  s\right)  $, then
$\theta\left(  R_{n}\right)  =0$ and $\theta\left(  s\right)  >0,\theta
^{\prime}\left(  s\right)  <0$ for $s\in\left(  0,R_{n}\right)  $. By taking
$\frac{d}{d\alpha}$ of (\ref{eqn-LE-a}), we have
\begin{equation}
w^{\prime\prime}+\frac{2}{s}w^{\prime}=-C_{\gamma}n\theta_{+}^{n-1}%
w,\ \ \ s\in\left(  0,R_{n}\right)  , \label{ode-w}%
\end{equation}
with $w\left(  0\right)  =1,w^{\prime}\left(  0\right)  =0.$ We show that
$w\left(  s\right)  $ has a unique zero in $\left(  0,R_{n}\right)  $. Indeed,
since $w\left(  0\right)  =1$ and $w\left(  R_{n}\right)  =\frac{n-1}{2}%
R_{n}\theta^{\prime}\left(  R_{n}\right)  <0$, by continuity of $w\left(
s\right)  \ $there exists $s_{0}\in\left(  0,R_{n}\right)  $ such that
$w\left(  s_{0}\right)  =0$. Moreover, for $s\in( 0,R_{n}) \ $we have
\begin{align*}
w^{\prime}\left(  s\right)   &  =\frac{n+1}{2}\theta^{\prime}\left(  s\right)
+\frac{n-1}{2}s\theta^{\prime\prime}\left(  s\right) \\
&  =\frac{n+1}{2}\theta^{\prime}\left(  s\right)  +\frac{n-1}{2}\left(
-2\theta^{\prime}\left(  s\right)  -C_{\gamma}s\theta\left(  s\right)
^{n}\right) \\
&  =\frac{3-n}{2}\theta^{\prime}\left(  s\right)  -\frac{n-1}{2}C_{\gamma
}s\theta\left(  s\right)  ^{n}<0.
\end{align*}
Thus $w\left(  s\right)  $ is monotone decreasing with exactly one zero
$s_{0}\ $in $\left(  0,R_{n}\right)  $. We extend $w\left(  s\right)  $ to be
a $C^{1}\left(  0,\infty\right)  $ function by solving the ODE (\ref{ode-w})
in $\left(  R_{n},\infty\right)  $. Noting that the right hand side of
(\ref{ode-w}) is zero in $\left(  R_{n},\infty\right)  $, we get
\[
w\left(  s\right)  =\frac{C_{1}}{s}+C_{2},\ \ s\in\left(  R_{n},+\infty
\right)  ,
\]
where
\[
C_{1}=-R_{n}^{2}w^{\prime}\left(  R_{n}\right)  >0,\ \ \ C_{2}=w\left(
R_{n}\right)  -\frac{C_{1}}{R_{n}}<0.
\]
Thus $w\left(  s\right)  <0$ in $\left(  R_{n},\infty\right)  $ and $w\left(
s\right)  \searrow C_{2}$ as $s\rightarrow+\infty$. Therefore, $w\left(
s\right)  $ only has one zero in $\left(  0,+\infty\right)  $. We show
$n^{-}\left(  B_{n}\right)  =1$ by comparison arguments. Suppose $n^{-}\left(
B_{n}\right)  \geq2$. Let $\lambda_{1}<0$ be the second negative eigenvalue of
$B_{n}$ and $\xi\left(  s\right)  \in H_{r}^{1}$ be the corresponding
eigenfunction, that is,
\begin{equation}
\left(  \xi^{\prime\prime}+\frac{2}{s}\xi^{\prime}\right)  =-C_{\gamma}%
n\theta_{+}^{n-1}\xi-\lambda_{1}\xi.\ \ \label{ode-xi}%
\end{equation}
Then $\xi\left(  s\right)
=c s^{-1} e^{-\sqrt{-\lambda_{1}}s}$
for $s> R_{n}$. By Sturm-Liouville theory, $\xi\left(  s\right)  $ has exactly
one zero $s_{1}\in\left(  0,+\infty\right)  $. We claim that this would lead
to $w\left(  s\right)  $ having two zeros, one in $\left(  0,s_{1}\right)  $
and the other in $\left(  s_{1},\infty\right)  $. We can assume $\xi\left(
s\right)  >0$ in $\left(  0,s_{1}\right)  $, them $\xi^{\prime}\left(
s_{1}\right)  <0$. Suppose $w\left(  s\right)  $ has no zero in $\left(
0,s_{1}\right)  $, then $w\left(  s\right)  >0$ in $\left(  0,s_{1}\right)  $
and $w^{\prime}\left(  s\right)  <0$ in $\left[  0,s_{1}\right]  $. The
integration of
\[
\int_{0}^{s_{1}}\left[  (\ref{ode-w})\xi\left(  s\right)  -\left(
\ref{ode-xi}\right)  w\left(  s\right)  \right]  s^{2}ds
\]
$\ $and an integration by parts yield
\[
-s_{1}^{2}\xi^{\prime}\left(  s_{1}\right)  w\left(  s_{1}\right)
=\lambda_{1}\int_{0}^{s_{1}}\xi\left(  s\right)  w\left(  s\right)  ds .
\]
This is an contradiction since the left hand side is positive and the right
hand side is negative. Thus $w\left(  s\right)  $ must have one zero in
$\left(  0,s_{1}\right)  $. By the same argument, $w\left(  s\right)  $ has
another zero in $\left(  s_{1},\infty\right)  $. This is in contradiction to
the fact that $w\left(  s\right)  $ has exactly one zero in $\left(
0,\infty\right)  $. Thus $n^{-}\left(  B_{n}\right)  <2$, which together with
Lemma \ref{lemma-property-A-EP} iii) shows that $n^{-}\left(  B_{n}\right)
=1$.

We complete the proof of the lemma by a continuation argument. According to
Corollary \ref{cor neg unchange}, $n^{-}(D_{\mu}^{0})$ is locally constant in
$\mu$ and $\gamma$ on the set $\{\mu\mid\ker D_{\mu}^{0}=\{0\}\}$.
For polytropic stars with $P\left(  \rho\right)  =K\rho^{\gamma}$ $\left(
\frac{6}{5}<\gamma<2\right)  $, by (\ref{relation-M-R-polytropic}) we
have$\ \frac{M\left(  \mu\right)  }{R_{\mu}}=\frac{C_{1}}{C_{2}}\mu^{\gamma
-1}\ $and thus $\frac{d}{d\mu}\left(  \frac{M\left(  \mu\right)  }{R_{\mu}%
}\right)  >0$ for any $\mu>0$ and $\gamma\in(\frac{6}{5},2)$. Therefore, by
Lemma \ref{lemma-kernel-D-mu}, $\ker D_{\mu}^{0}=\left\{  0\right\}  $ for any
$\gamma\in\left(  \frac{6}{5},2\right)  $ and thus $n^{-}(D_{\mu}^{0})=1$ for
all $\mu>0$.

\end{proof}

For general equation of states, by Corollary \ref{cor neg unchange}, Lemma
\ref{lemma-kernel-D-mu}, and Lemma \ref{lemma-Morse-3-below}, we have


\begin{lemma}
\label{lemma-small-D-neg} Assume (\ref{assumption-P1})-(\ref{assumption-P2})
for $P\left(  \rho\right)  $. There exists $\mu_{0}>0$ such that for any
$\mu\in\left(  0,\mu_{0}\right)  $, $n^{-}\left(  D_{\mu}^{0}\right)  =1$.
Moreover, as a function of $\mu\in(0, \mu_{max})$, $n^{-}(D_{\mu}^{0})$ is
locally constant.
\end{lemma}

\begin{proof}
We use the notations in Lemma \ref{lemma-steady-small-density}, where the
non-rotating stars with small center density $\mu\ $are constructed. Define
the operator
\[
B_{\alpha}=-\Delta_{s}-g_{\alpha}^{\prime}\left(  \theta_{\alpha}\right)
:\dot{H}_{r}^{1}\rightarrow\dot{H}_{r}^{-1},
\]
where $\theta_{\alpha},g_{\alpha}\ $are defined in (\ref{defn-y-mu-scaling})
and (\ref{defn-g-alpha}). As in the proof of Lemma \ref{lemma-Morse-3-below},
we have $n^{-}\left(  D_{\mu}^{0}\right)  =n^{-}\left(  B_{\alpha}\right)  $
where $\alpha=\Phi^{\prime}\left(  \mu\right)  $. We also define
\[
B_{0}=-\Delta_{s}-g_{0}^{\prime}\left(  \theta_{0}\right)  =-\Delta
_{s}-C_{\gamma_{0}}n_{0}\left(  \theta_{0}\right)  _{+}^{n_{0}-1},
\]
where $\theta_{0}$ is the Lane-Emden function satisfying (\ref{eqn-LE-g-0})
and $g_{0}$ is defined in (\ref{defn-g-0}). By the proof of Lemma
\ref{lemma-steady-small-density}, when $\alpha\rightarrow0+,\ \ g_{\alpha
}\rightarrow g_{0}$ in $C^{1}\left(  0,1\right)  $ and $\theta_{\alpha
}\rightarrow\theta_{0}$ in $C^{1}\left(  0,R\right)  $ for any $R>0$. By
Lemmas \ref{lemma-Morse-3-below},
we have $n^{-}\left(  B_{0}\right)  =1$.
Corollary \ref{cor neg unchange} implies that there exists $\alpha_{0}>0$ such
that when $\alpha<\alpha_{0}$ we have $n^{-}\left(  B_{\alpha}\right)  =1$.
This proves the lemma by letting $\mu_{0}=\left(  \Phi^{\prime}\right)
^{-1}\left(  \alpha_{0}\right)  .$ Moreover, $n^{-}(D_{\mu}^{0})$ changes only
at critical points of $\frac{M\left(  \mu\right)  }{R_{\mu}}$ due to Corollary
\ref{cor neg unchange}.
\end{proof}

We first prove the following lemma of the non-degeneracy of the mass-radius
curve of the non-rotating stars, which will be crucial in the analysis of the
change of the Morse index $n^{-}(D_{\mu}^{0})$.

\begin{lemma}
\label{lemma-M-R-no critical}There exists no point $\mu\in\left(  0,\mu_{\max
}\right)  $ such that $M^{\prime}(\mu)=\frac{d}{d\mu}\left(  \frac{M\left(
\mu\right)  }{R_{\mu}}\right)  =0.$
\end{lemma}

\begin{proof}
Suppose otherwise, $M^{\prime}(\mu)=\frac{d}{d\mu}\left(  \frac{M\left(
\mu\right)  }{R_{\mu}}\right)  =0$ at some $\mu\in\left(  0,\mu_{\max}\right)
$. Then by Lemma \ref{lemma-kernel-D-mu}, $D_{\mu}^{0}\frac{\partial V_{\mu}%
}{\partial\mu}=0$, i.e.,
\begin{equation}
\left(  \frac{\partial V_{\mu}}{\partial\mu}\right)  ^{\prime\prime}+\frac
{2}{r}\left(  \frac{\partial V_{\mu}}{\partial\mu}\right)  ^{\prime}=-4\pi
F_{+}^{\prime}\left(  y_{\mu}\left(  r\right)  \right)  \frac{\partial V_{\mu
}}{\partial\mu},\ r>0, \label{eqn-V-mu}%
\end{equation}
and $\frac{\partial V_{\mu}}{\partial\mu}=-\frac{\partial y_{\mu}}{\partial
\mu}$ in $S_{\mu}$. By (\ref{ode in ball}) and $\rho_{\mu}=F_{+}\left(
y_{\mu}\right)  $, we have
\[
\left(  \frac{\partial y_{\mu}}{\partial\mu}\right)  ^{\prime}\left(  R_{\mu
}\right)  =-\frac{1}{R_{\mu}^{2}}\int_{0}^{R_{\mu}}s^{2}4\pi F_{+}^{\prime
}\left(  y_{\mu}\left(  s\right)  \right)  \frac{\partial y_{\mu}}{\partial
\mu}ds=-\frac{1}{R_{\mu}^{2}}M^{\prime}(\mu)=0.
\]
Then $\left(  \frac{\partial V_{\mu}}{\partial\mu}\right)  ^{\prime}\left(
R_{\mu}\right)  =0$ and by (\ref{eqn-V-mu}) it follows that $\left(
\frac{\partial V_{\mu}}{\partial\mu}\right)  ^{\prime}\left(  r\right)  =0$
for any $r>R_{\mu}$. Therefore, $\frac{\partial V_{\mu}}{\partial\mu}\left(
r\right)  =0$ for any $r\geq R_{\mu}$. By (\ref{eqn-V-mu}), this implies that
$\frac{\partial V_{\mu}}{\partial\mu}\left(  r\right)  =0$ for any $r>0$. But
this is impossible since
\[
\frac{\partial V_{\mu}}{\partial\mu}\left(  0\right)  =-\frac{\partial y_{\mu
}}{\partial\mu}\left(  0\right)  =-\Phi^{\prime\prime}\left(  \mu\right)
\neq0.
\]

\end{proof}

Finally we give the following proposition on the change of $n^{-}\left(
D_{\mu}^{0}\right)  $ at critical points of $\frac{M\left(  \mu\right)
}{R_{\mu}}$.

\begin{proposition}
\label{prop-jump-index} Let $\mu^{\ast}$ be a critical point of $\frac
{M\left(  \mu\right)  }{R_{\mu}}$,
then for $\mu$ near $\mu_{\ast}$ it holds
\begin{equation}
n^{-}(D_{\mu}^{0})=n^{-}(D_{\mu_{\ast}}^{0})+i_{\mu} \label{formula-jump-D-mu}%
\end{equation}
where the index $i_{\mu}$ is defined in (\ref{defn-imu}). Therefore, the jump
of $n^{-}\left(  D_{\mu}^{0}\right)  $ at $\mu^{\ast}$ equals that of $i_{\mu
}$.
\end{proposition}

\begin{proof}
To prove (\ref{formula-jump-D-mu}), we need to study the perturbation of zero
eigenvalue of $D_{\mu^{\ast}}^{0}$ for $\mu$ near $\mu^{\ast}$. The idea is
similar to the proof of Proposition \ref{P:perturbation}, but with a more
concrete decomposition. For $\mu$ near $\mu_{\ast}$, let
\[
Z(\mu)=\{u\in\dot{H}^{1}(\mathbf{R}^{3})\mid\langle F_{+}^{\prime}\big(V_{\mu
}(R_{\mu})-V_{\mu}(r)\big),u\rangle=0\}.
\]
Using $F_{+}(0)=0$, one may compute
\begin{align*}
&  \langle F_{+}^{\prime}\big(V_{\mu}(R_{\mu})-V_{\mu}(r)\big),\partial_{\mu
}V_{\mu}\rangle=\int_{S_{\mu}}F_{+}^{\prime}\big(V_{\mu}(R_{\mu})-V_{\mu
}(r)\big)\partial_{\mu}V_{\mu}(r)dx\\
=  &  -\partial_{\mu}\int_{S_{\mu}}F_{+}\big(V_{\mu}(R_{\mu})-V_{\mu
}(r)\big)dx+\partial_{\mu}\big(V_{\mu}(R_{\mu})\big)\int_{S_{\mu}}%
F_{+}^{\prime}\big(V_{\mu}(R_{\mu})-V_{\mu}(r)\big)dx\\
=  &  -M^{\prime}(\mu)-\partial_{\mu}\big(\frac{M(\mu)}{R_{\mu}}%
\big)\int_{S_{\mu}}F_{+}^{\prime}\big(V_{\mu}(R_{\mu})-V_{\mu}(r)\big)dx.
\end{align*}
Lemma \ref{lemma-M-R-no critical} yields that $M^{\prime}(\mu)\neq0$ for $\mu$
near $\mu_{\ast}$ and thus
\begin{equation}
\dot{H}^{1}(\mathbf{R}^{3})=Z(\mu)\oplus\mathbf{R}\{\partial_{\mu}V_{\mu}\}.
\label{E:dH1decom}%
\end{equation}
Moreover, differentiating \eqref{eqn-potential-steady} and using Lemma
\ref{lemma-formula-V-R} we obtain
\[
D_{\mu}^{0}\partial_{\mu}V_{\mu}=4\pi\partial_{\mu}\big(\frac{M(\mu)}{R_{\mu}%
}\big)F_{+}^{\prime}\big(V_{\mu}(R_{\mu})-V_{\mu}(r)\big).
\]
Therefore, \eqref{E:dH1decom} is a $D_{\mu}^{0}$-orthogonal decomposition.
From Lemma \ref{lemma-kernel-D-mu}, $D_{\mu}^{0}$ is non-degenerate on
$Z(\mu)$ for $\mu$ close to $\mu_{\ast}$ and thus
\[
n^{-}(D_{\mu}^{0})-n^{-}(D_{\mu_{\ast}}^{0})=n^{-}\big(D_{\mu}^{0}%
|_{\mathbf{R}\{\partial_{\mu}V_{\mu}\}}\big).
\]
Using the above calculations, we have
\begin{align*}
&  \langle D_{\mu}^{0}\partial_{\mu}V_{\mu},\partial_{\mu}V_{\mu}\rangle
=4\pi\partial_{\mu}\big(\frac{M(\mu)}{R_{\mu}}\big)\langle F_{+}^{\prime
}\big(V_{\mu}(R_{\mu})-V_{\mu}(r)\big),\partial_{\mu}V_{\mu}\rangle\\
=  &  -4\pi M^{\prime}(\mu)\partial_{\mu}\big(\frac{M(\mu)}{R_{\mu}}%
\big)-4\pi\left(  \partial_{\mu}\big(\frac{M(\mu)}{R_{\mu}}\big)\right)
^{2}\int_{S_{\mu}}F_{+}^{\prime}\big(V_{\mu}(R_{\mu})-V_{\mu}(r)\big)dx.
\end{align*}
Therefore, \eqref{formula-jump-D-mu} follows for $\mu$ near $\mu_{\ast}$.

\end{proof}

\subsection{Stability for non-radial perturbations}

\label{SS:non-radial}

We study the linearized system (\ref{eqn-linearized-Continuity}%
)-(\ref{eqn-linearized-Euler}) for non-radial and radial perturbations
separately. Here we follow the tradition in the astrophysics literature that
``non-radial" perturbations refer to those modes corresponding to non-constant
spherical harmonics. See Definition \ref{D:non-radial} for the precise definition.

First, we give a Helmholtz type decomposition of vector fields in $Y_{\mu}$.

\begin{lemma}
\label{lemma-decomp}There is a direct sum decomposition $Y_{\mu}=Y_{\mu
,1}\oplus Y_{\mu,2}$, where $Y_{\mu,1}$ is the closure of
\[
\left\{  u\in\left(  C^{1}\left(  S_{\mu}\right)  \right)  ^{3}\cap Y_{\mu
}\ |\ \nabla\cdot\left(  \rho_{\mu}u\right)  =0\text{ }\right\}
\]
in $Y_{\mu}$ and $Y_{\mu,2}$ is the closure of
\[
\left\{  u\in Y_{\mu}\ |\ u=\nabla p,\text{ for some }p\in C^{1}\left(
S_{\mu}\right)  \right\}
\]
in $Y_{\mu}$. \
\end{lemma}

\begin{proof}
Define the space $Z$ to be the closure of
\[
\left\{  p\in C^{1}\left(  S_{\mu}\right)  \ |\ \int_{S_{\mu}}\rho_{\mu
}\left\vert \nabla p\right\vert ^{2}\ dx<\infty\right\}
\]
under the norm $\left\Vert p\right\Vert _{Z}=\left(  \int_{S_{\mu}}\rho_{\mu
}\left\vert \nabla p\right\vert ^{2}\ dx\right)  ^{\frac{1}{2}}$, quotient the
constant functions. The inner product on $Z$ is defined as
\[
\left(  p_{1},p_{2}\right)  _{Z}=\int_{S_{\mu}}\rho_{\mu}\nabla p_{1}%
\cdot\nabla p_{2}dx.
\]
For any fixed $u\in Y_{\mu}$, we seek $p_{u}\in Z$ as a weak solution of the
equation
\[
\nabla\cdot\left(  \rho_{\mu}\nabla p\right)  =\nabla\cdot\left(  \rho_{\mu
}u\right)  .
\]
This is equivalent to that
\begin{equation}
\int_{S_{\mu}}\rho_{\mu}\nabla p_{u}\cdot\nabla pdx=\int_{S_{\mu}}\rho_{\mu
}u\cdot\nabla pdx,\ \ \ \ \forall p\in Z. \label{eqn-helmholtz-weak}%
\end{equation}
The right hand side above defines a bounded linear functional on $Z$. Thus by
the Riesz representation Theorem, there exists a unique $p_{u}\in Z$
satisfying (\ref{eqn-helmholtz-weak}). Let $u_{2}=\nabla p_{u}\in Y_{\mu,2}$.
Then $u_{1}=u-u_{2}\in Y_{\mu,1}$. Moreover, it is clear that $Y_{\mu,1}\perp
Y_{\mu,2}$ in the inner product of $Y_{\mu}$. This finishes the proof of the lemma.
\end{proof}

The decomposition
\[
X_{\mu}\times Y_{\mu}=\left(  \left\{  0\right\}  \times Y_{\mu,1}\right)
\oplus\left(  X_{\mu}\times Y_{\mu,2}\right)  ,
\]
is clearly invariant for the linearized system
(\ref{eqn-linearized-Continuity})-(\ref{eqn-linearized-Euler}). We shall call
perturbations in $\left\{  0\right\}  \times Y_{\mu,1}$ and $X_{\mu}\times
Y_{\mu,2}\ $to be pseudo-divergence free and irrotational respectively. In
particular$\ \left\{  0\right\}  \times Y_{\mu,1}$ is a subspace of steady
states for (\ref{eqn-linearized-Continuity})-(\ref{eqn-linearized-Euler}),
where $0$ is the only eigenvalue. Thus, we restrict to initial data $\left(
\sigma\left(  0\right)  ,u\left(  0\right)  \right)  \in X_{\mu}\times
Y_{\mu,2}$. Any solution $\left(  \sigma\left(  t\right)  ,u\left(  t\right)
\right)  \in X_{\mu}\times Y_{\mu,2}$ can be written as
\begin{equation}
\sigma\left(  x,t\right)  =\sigma_{1}\left(  r,t\right)  +\sigma_{2}\left(
x,t\right)  ,\ \ \label{decom-rho}%
\end{equation}
and
\begin{equation}
u\left(  x,t\right)  =\nabla\xi=v_{1}\left(  r,t\right)  \frac{x}{r}+\nabla
\xi_{2}\left(  x,t\right)  ,\label{decomp-u}%
\end{equation}
where $\left(  \sigma_{1},v_{1}\right)  $ is the radial component defined by
\[
\ \sigma_{1}\left(  r,t\right)  =\int_{\mathbb{S}^{2}}\sigma\left(
r\theta\right)  dS_{\theta},\ \xi_{1}\left(  r,t\right)  =\int_{\mathbb{S}%
^{2}}\xi\left(  r\theta\right)  dS_{\theta},\ v_{1}\left(  x,t\right)
=\frac{\partial}{\partial r}\xi_{1}\left(  r,t\right)  ,\
\]
and $\left(  \sigma_{2},\xi_{2}\right)  =\left(  \sigma-\sigma_{1},\xi-\xi
_{1}\right)  \ $ are the nonradial components.

The radial component $\left(  \sigma_{1},v_{1}\right)  $ will be studied in
next subsection. The nonradial component $\left(  \sigma_{2}\left(
x,t\right)  ,\ \xi_{2}\left(  x,t\right)  \right)  $ satisfies the system
\[
\partial_{t}\sigma_{2}=-\nabla\cdot\left(  \rho_{\mu}\nabla\xi_{2}\right)
\]%
\[
\xi_{2,t}=-\left(  \Phi^{\prime\prime}\left(  \rho_{\mu}\right)  \sigma
_{2}+V_{2}\right)  =-L_{\mu}\sigma_{2},\ \ \Delta V_{2}=4\pi\sigma_{2}.
\]
It is of the Hamiltonian form
\begin{equation}
\partial_{t}\left(
\begin{array}
[c]{c}%
\sigma_{2}\\
\xi_{2}%
\end{array}
\right)  =\left(
\begin{array}
[c]{cc}%
0 & I\\
-I & 0
\end{array}
\right)  \left(
\begin{array}
[c]{cc}%
L_{\mu} & 0\\
0 & \tilde{A}_{\mu}%
\end{array}
\right)  \left(
\begin{array}
[c]{c}%
\sigma_{2}\\
\xi_{2}%
\end{array}
\right)  ,\label{hamiltonian-non-radial}%
\end{equation}
where $\tilde{A}_{\mu}=-\nabla\cdot\left(  \rho_{\mu}\nabla\right)  $,
$\left(  \sigma_{2},\ \xi_{2}\right)  \in X_{\mu,n}\times Y_{\mu,n}$ with
\[
X_{\mu,n}=\left\{  \rho\in X_{\mu}\ |\ \int_{\mathbb{S}^{2}}\rho\left(
r\theta\right)  dS_{\theta}=0\right\}  ,
\]
and%
\[
Y_{\mu,n}=\left\{  \xi\in Y_{\mu,2}\mid\int_{S_{\mu}}\rho_{\mu}\left\vert
\nabla\xi\right\vert ^{2}dx<\infty,\ \int_{\mathbb{S}^{2}}\xi\left(
r\theta\right)  dS_{\theta}=0\right\}  ,\ \ \left\Vert \xi\right\Vert
_{Y_{\mu,n}}=\left\Vert \nabla\xi\right\Vert _{L_{\rho_{\mu}}^{2}}.
\]
We take this chance opportunity to define the following terminology.

\begin{definition}
\label{D:non-radial} Define the subspaces of radial and non-radial
perturbations for the linearized Euler-Poisson system
(\ref{eqn-linearized-Continuity})-(\ref{eqn-linearized-Euler}) as
\begin{align*}
\mathbf{X}_{r}= &  \{\big(\rho(|x|),v(|x|)\frac{x}{|x|}\big)\in X_{\mu}\times
Y_{\mu}\},\\
\mathbf{X}_{nr}= &  \big(\{0\}\times Y_{\mu1}\big)\oplus\{(\rho,u=\nabla
\xi)\in X_{\mu}\times Y_{\mu}\mid\rho\in X_{\mu,n},\ \xi\in Y_{\mu,n}\}.
\end{align*}

\end{definition}

Clearly we have that the decomposition $X_{\mu}\times Y_{\mu}=\mathbf{X}%
_{r}\oplus\mathbf{X}_{nr}$ is invariant under $e^{t\mathcal{J}_{\mu
}\mathcal{L}_{\mu}}$.

By using spherical harmonics, for any $\rho\in X_{\mu,n}$, we write
\[
\rho\left(  x\right)  =\sum_{l=1}^{\infty}\sum_{m=-l}^{l}\rho_{lm}\left(
r\right)  Y_{lm}\left(  \theta\right)  ,
\]
then
\[
L_{\mu}\rho=\sum_{l=1}^{\infty}\sum_{m=-l}^{l}L_{\mu,l}\rho_{lm}%
\ Y_{lm}\left(  \theta\right)  ,
\]
where
\begin{equation}
L_{\mu,l}=\left(  \Phi^{\prime\prime}\left(  \rho_{\mu}\right)  -4\pi\left(
-\Delta_{r}+\frac{l\left(  l+1\right)  }{r^{2}}\right)  ^{-1}\right)
:X_{\mu,r}\rightarrow X_{\mu,r}^{\ast}.\label{defn-L-mu-l}%
\end{equation}
By Lemma \ref{lemma-property-A-EP} and the proof of Lemma
\ref{lemma-equivalent-operators}, we have
\[
n^{-}\left(  L_{\mu,l}|_{X_{\mu,r}}\right)  =n^{-}\left(  D_{\mu}^{l}\right)
=0,\ \forall\ l\geq1.
\]
Therefore,
\[
n^{-}\left(  L_{\mu}|_{X_{\mu,n}}\right)  =\sum_{l=1}^{\infty}\sum_{m=-l}%
^{l}n^{-}\left(  L_{\mu,l}|_{X_{\mu,r}}\right)  =0.
\]
Since $\tilde{A}_{\mu}>0$ on $Y_{\mu,n}$, by Theorem \ref{T:main}, there is no
unstable eigenvalue for the system (\ref{hamiltonian-non-radial}). Moreover,
we shall show that all the eigenvalues of (\ref{hamiltonian-non-radial}) are
isolated with finite multiplicity. Define the space
\[
Z_{\mu,n}=\left\{  \xi\in Y_{\mu,n}\ |\ \tilde{A}_{\mu}\xi\in X_{\mu
,n}\right\}
\]
with the norm
\[
\left\Vert \xi\right\Vert _{Z_{\mu,n}}=\left\Vert \nabla\xi\right\Vert
_{L_{\rho_{\mu}}^{2}}+\left\Vert \tilde{A}_{\mu}\xi\right\Vert _{L_{\Phi
^{\prime\prime}\left(  \rho_{\mu}\right)  }^{2}}.
\]
Then by Theorem \ref{T:main}, it suffices to show that the embedding
$Z_{\mu,n}\hookrightarrow Y_{\mu,n}$ is compact. This follows from Proposition
12 in \cite{jang-makino-LEP19}.

By using spherical harmonics, we can further decompose
(\ref{hamiltonian-non-radial}). For $\left(  \sigma_{2},\ \xi_{2}\right)  \in
X_{\mu,n}\times Y_{\mu,n}$, let%

\[
\sigma_{2}\left(  x\right)  =\sum_{l=1}^{\infty}\sum_{m=-l}^{l}\sigma
_{lm}\left(  r,t\right)  Y_{lm}\left(  \theta\right)  ,\ \ \xi_{2}\left(
x,t\right)  =\sum_{l=1}^{\infty}\sum_{m=-l}^{l}\xi_{lm}\left(  r,t\right)
Y_{lm}\left(  \theta\right)  .\
\]
For each $l\geq1,\ -l\leq m\leq l,\ $the component $\left(  \sigma_{lm}\left(
r,t\right)  ,\xi_{lm}\left(  r,t\right)  \right)  $ satisfies the separable
Hamiltonian system%
\begin{equation}
\partial_{t}\left(
\begin{array}
[c]{c}%
\sigma_{lm}\\
\xi_{lm}%
\end{array}
\right)  =\left(
\begin{array}
[c]{cc}%
0 & I\\
-I & 0
\end{array}
\right)  \left(
\begin{array}
[c]{cc}%
L_{\mu,l} & 0\\
0 & A_{\mu,l}%
\end{array}
\right)  \left(
\begin{array}
[c]{c}%
\sigma_{lm}\\
\xi_{lm}%
\end{array}
\right)  ,\label{eqn-Hamiltonian-lm}%
\end{equation}
on the space $X_{\mu,r}\times\tilde{Y}_{\mu,r}$, where
\begin{equation}
\tilde{Y}_{\mu,r}=\left\{  p\left(  r\right)  \ |\ \int_{0}^{R_{\mu}}\rho
_{\mu}\left(  r^{2}\left(  \partial_{r}p\right)  ^{2}+p^{2}\right)
dr<\infty\right\}  ,\label{defn-Y-tilde-nonradial}%
\end{equation}
the operator $L_{\mu,l}$ is defined in (\ref{defn-L-mu-l}) and
\[
A_{\mu,l}=-\frac{1}{r^{2}}\partial_{r}\left(  \rho_{\mu}r^{2}\partial
_{r}\right)  +\frac{\rho_{\mu}l\left(  l+1\right)  }{r^{2}}:\tilde{Y}_{\mu
,r}\rightarrow\tilde{Y}_{\mu,r}^{\ast}.
\]
By the properties of the operators $L_{\mu,l}$ (equivalently the operators
$D_{\mu}^{l}$) given in Lemma \ref{lemma-property-A-EP}, it is easy to see
that, when $l>1$, all the eigenvalues of (\ref{eqn-Hamiltonian-lm}) are
nonzero and purely imaginary. When $l=1$, (\ref{eqn-Hamiltonian-lm}) has a
kernel space spanned by $\left(  \rho_{\mu}^{\prime}\left(  r\right)
,0\right)  ^{T}$ corresponding to translation modes $\left(  \partial_{x_{i}%
}\rho_{\mu},0\right)  ^{T}$ for the linearized Euler-Poisson system
(\ref{eqn-linearized-Continuity})-(\ref{eqn-linearized-Euler}). According to
Theorem \ref{T:main}, all eigenvalues of $\mathcal{J}_{\mu}\mathcal{L}_{\mu}$
restricted to the invariant subspace $X_{\mu}\times Y_{\mu,2}$, and thus of
\eqref{eqn-Hamiltonian-lm}, are semi-simple except for possibly the zero
eigenvalue. Since $0$ is an isolated eigenvalue, Theorem \ref{T:main} applied
to $\mathcal{J}_{\mu}\mathcal{L}_{\mu}|_{X_{\mu,n}\times Y_{\mu,n}}$ implies
that the eigenspace of $0$ only consists of generalized eigenvectors with
finite multiplicity.

Indeed (\ref{eqn-Hamiltonian-lm}) does have a nontrivial generalized
eigenvectors and thus non-trivial Jordan blocks associated to $0$. To see
this, for any $\zeta\in\tilde{Y}_{\mu,r}$, we have
\[
|\int_{S_{\mu}}\rho_{\mu}^{\prime}\zeta dx|\leq\Vert\zeta\Vert_{\tilde{Y}%
_{\mu,r}}\left(  \int_{S_{\mu}}\left(  \rho_{\mu}^{\prime}\right)  ^{2}%
\rho_{\mu}^{-1}dx\right)  ^{\frac{1}{2}}\lesssim\Vert\zeta\Vert_{\tilde
{Y}_{\mu,r}}%
\]
where we used $\gamma_{0}\in(\frac{6}{5},2)$ and
\[
\rho_{\mu}=O(|R_{\mu}-r|^{\frac{1}{\gamma_{0}-1}}),\;\;\rho_{\mu}^{\prime
}=O(|R_{\mu}-r|^{\frac{1}{\gamma_{0}-1}-1}),\text{ for }|R_{\mu}-r|\ll1.
\]
Therefore, $\rho_{\mu}^{\prime}\in\tilde{Y}_{\mu,r}^{\ast}$ and thus the
Lax-Milgram theorem implies that exists a unique $\zeta(r)\in\tilde{Y}_{\mu
,r}$ such that
\begin{equation}
\rho_{\mu}^{\prime}=A_{\mu,1}\zeta=-\frac{1}{r^{2}}\partial_{r}\left(
\rho_{\mu}r^{2}\partial_{r}\zeta\right)  +\frac{2\rho_{\mu}}{r^{2}}\zeta.
\label{E:G-kernel}%
\end{equation}
Therefore, $\big(0,\zeta Y_{1m}(\theta)\big)^{T}$, $m=0,\pm1$, belong to the
generalized kernel of \eqref{eqn-Hamiltonian-lm}, which correspond to
$\big(0,\partial_{x_{j}}(\zeta\frac{x}{r})\big)^{T}$, $j=1,2,3$, in the
generalized kernel of $\mathcal{J}_{\mu}\mathcal{L}_{\mu}$ with
\begin{equation}
\mathcal{J}_{\mu}\mathcal{L}_{\mu}\big(0,\partial_{x_{j}}\nabla\tilde{\zeta
}(|x|)\big)^{T}=\big(\partial_{x_{j}}\rho_{\mu},0\big)^{T},\quad\tilde{\zeta
}^{\prime}=\zeta. \label{E:G-kernel-1}%
\end{equation}
Moreover these functions in the generalized kernel of $\mathcal{J}_{\mu
}\mathcal{L}_{\mu}$ do \textit{not} belong to the range $R(\mathcal{J}_{\mu
}\mathcal{L}_{\mu})$. In fact, suppose
\[
(\mathcal{J}_{\mu}\mathcal{L}_{\mu})(\rho,u)^{T}=\big(0,\partial_{x_{j}}%
\nabla\tilde{\zeta}\big)^{T},\quad(\rho,u)^{T}\in X_{\mu}\times Y_{\mu},
\]
then one may compute
\begin{align*}
&  \langle A_{\mu}\partial_{x_{j}}\nabla\tilde{\zeta},\partial_{x_{j}}%
\nabla\tilde{\zeta}\rangle=\langle L_{\mu}%
\begin{pmatrix}
0\\
\partial_{x_{j}}\nabla\tilde{\zeta}%
\end{pmatrix}
,%
\begin{pmatrix}
0\\
\partial_{x_{j}}\nabla\tilde{\zeta}%
\end{pmatrix}
\rangle\\
=  &  -\langle L_{\mu}J_{\mu}L_{\mu}%
\begin{pmatrix}
0\\
\partial_{x_{j}}\nabla\tilde{\zeta}%
\end{pmatrix}
,%
\begin{pmatrix}
\rho\\
u
\end{pmatrix}
\rangle=-\langle L_{\mu}%
\begin{pmatrix}
\partial_{x_{j}}\rho_{\mu}\\
0
\end{pmatrix}
,%
\begin{pmatrix}
\rho\\
u
\end{pmatrix}
\rangle=0
\end{align*}
which is a contradiction. Therefore, we may conclude that the zero eigenvalue
of $\mathcal{J}_{\mu}\mathcal{L}_{\mu}|_{X_{\mu,n}\times Y_{\mu,n}}$ has a
6-dim eigenspace with geometric multiplicity 3 and algebraic multiplicity 6.

Above discussions are summarized below.

\begin{proposition}
\label{prop-discrete-non-radial} Any non-rotating star $\rho_{\mu}$ is
spectrally stable under non-radial perturbations in $\mathbf{X}_{nr}$.
All nonzero eigenvalues of (\ref{hamiltonian-non-radial})\ are isolated and of
finite multiplicity. The zero eigenvalue of the linearized Euler-Poisson
operator $\mathcal{J}_{\mu}\mathcal{L}_{\mu}|_{\mathbf{X}_{nr}}$
is isolated with an infinite dimensional eigenspace
\[
(\left\{  0\right\}  \times Y_{\mu1})\oplus span\{\big(\partial_{x_{j}}%
\rho_{\mu},0\big)^{T},\ \big(0,\partial_{x_{j}}(\zeta\frac{x}{r})\big)^{T}\mid
j=1,2,3\}
\]
where $\mathcal{J}_{\mu}\mathcal{L}_{\mu}$ has three $2\times2$ Jordan blocks
associated to \eqref{E:G-kernel-1} generated by the translation symmetry.
\end{proposition}

\begin{remark}
For irrotational perturbations, the eigenvalues of
(\ref{hamiltonian-non-radial})\ were shown to be purely discrete in
\cite{jang-makino-LEP19} by a different approach. In \cite{beyer-JMP-1}
\cite{beyer-JMP-2} \cite{beyer-schmidt}, the spectrum for nonradial
perturbations were shown to be countable, and it was conjectured in
\cite{beyer-schmidt} that zero is the only accumulation point. This is indeed
true for barotropic equation of states $P\left(  \rho\right)  $ by above
Proposition or results in \cite{jang-makino-LEP19}.
\end{remark}

\begin{remark}
\label{rmk-variational-steady}In the astrophysics literature (\cite{antonov}
\cite{lebowitz65} \cite{aly-perez-92}), the stability of non-rotating stars
under nonradial perturbations (Antonov-Lebowitz Theorem) was shown by using
the physical principle that the stable states should be energy minimizers
under the constraint of constant mass. We discuss such energy principle below.

The steady density $\rho_{\mu}$ has the following variational structure.
Define the functional
\begin{equation}
E_{\mu}\left(  \rho\right)  =\int\Phi\left(  \rho\right)  dx-\frac{1}{8\pi
}\int\left\vert \nabla V\right\vert ^{2}dx-V_{\mu}\left(  R_{\mu}\right)
\int\rho dx,\label{energy functional -rho}%
\end{equation}
with $\ \Delta V=4\pi\rho$. Then $\rho_{\mu}$ is a critical point of $E_{\mu
}\left(  \sigma\right)  $, that is, $E_{\mu}^{\prime}\left(  \rho_{\mu
}\right)  =0$ which is exactly the equation (\ref{first-order-variation}). The
2nd order variation of $E_{\mu}$ at $\rho_{\mu}$ is
\begin{equation}
\left\langle E_{\mu}^{\prime\prime}\left(  \rho_{\mu}\right)  \rho
,\rho\right\rangle =\int\left(  \Phi^{\prime\prime}\left(  \rho_{\mu}\right)
\rho^{2}-\frac{1}{4\pi}\left\vert \nabla V\right\vert ^{2}\right)
dx=\left\langle L_{\mu}\rho,\rho\right\rangle
.\label{quadratic-L-2nd-variation}%
\end{equation}
We note that the energy functional
\[
E\left(  \rho,u\right)  =\frac{1}{2}\int\rho\left\vert u\right\vert
^{2}dx+\int\Phi\left(  \rho\right)  dx-\frac{1}{8\pi}\int\left\vert \nabla
V\right\vert ^{2}dx
\]
is conserved for the nonlinear Euler-Poisson equation (\ref{continuity-EP}%
)-(\ref{Poisson-EP}). Let $M\left(  \rho\right)  =\int\rho dx$ to be the total
mass and define
\[
I_{\mu}\left(  \rho,u\right)  =E\left(  \rho,u\right)  -V_{\mu}\left(  R_{\mu
}\right)  M\left(  \rho\right)  =\frac{1}{2}\int\rho\left\vert v\right\vert
^{2}dx+E_{\mu}\left(  \rho\right)  .\ \ \
\]
The $\left(  \rho_{\mu},0\right)  $ is a critical point of $I_{\mu}\left(
\rho,u\right)  $. The 2nd order variation of $I_{\mu}\left(  \rho,u\right)  $
at $\left(  \rho_{\mu},0\right)  $ is given by the functional
\[
H_{\mu}\left(  \sigma,v\right)  =\frac{1}{2}\int_{S_{\mu}}\rho_{\mu}\left\vert
v\right\vert ^{2}dx+\frac{1}{2}\left\langle L_{\mu}\sigma,\sigma\right\rangle
\]
as defined in (\ref{defn-H-linearized-EP}), which is a conserved quantity of
the linearized Euler-Poisson system (\ref{eqn-linearized-Continuity}%
)-(\ref{eqn-linearized-Euler}).

By the above variational structures, the physical principle that stable stars
should be energy minimizers under the constraint of constant mass is
equivalent to the statement that $\rho_{\mu}$ is stable only when
$\left\langle E_{\mu}^{\prime\prime}\left(  \rho_{\mu}\right)  \sigma
,\sigma\right\rangle \geq0$ for all perturbations $\sigma\ $supported in
$S_{\mu}$ satisfying the mass constraint $\int\sigma\ dx=0$. This was also
called Chandrasekhar's variational principle (\cite{chandra-variational-64})
in the astrophysical literature (\cite{binney-tremaine2008}).
\end{remark}

\subsection{\label{section-TPP} Turning point principle for radial
perturbations}

Denote $X_{\mu,r}$ and $Y_{\mu,r}$ to be the radially symmetric subspace of
$L_{\Phi^{\prime\prime}(\rho_{\mu})}(S_{\mu})$ and $L_{\rho_{\mu}}^{2}\left(
S_{\mu}\right)  $, respectively. By (\ref{hamiltonian-EP}), the radial
component $\left(  \sigma_{1},v_{1}\right)  $ of $\left(  \sigma,v\right)  $
as defined in (\ref{decom-rho})-(\ref{decomp-u}) satisfies
\begin{align}
&  \partial_{t}\left(
\begin{array}
[c]{c}%
\sigma_{1}\\
v_{1}%
\end{array}
\right)  \label{hamiltonian-radial}\\
&  =\left(
\begin{array}
[c]{cc}%
0 & -\frac{1}{r^{2}}\partial_{r}\left(  r^{2}\cdot\right)  \\
-\partial_{r} & 0
\end{array}
\right)  \left(
\begin{array}
[c]{cc}%
\Phi^{\prime\prime}\left(  \rho_{\mu}\right)  -{4\pi}\left(  -\Delta
_{r}\right)  ^{-1} & 0\\
0 & \rho_{\mu}%
\end{array}
\right)  \left(
\begin{array}
[c]{c}%
\sigma_{1}\\
v_{1}%
\end{array}
\right)  \nonumber\\
&  =\left(
\begin{array}
[c]{cc}%
0 & B_{\mu,r}\\
-B_{\mu,r}^{\prime} & 0
\end{array}
\right)  \left(
\begin{array}
[c]{cc}%
L_{\mu,r} & 0\\
0 & A_{\mu,r}%
\end{array}
\right)  \left(
\begin{array}
[c]{c}%
\sigma_{1}\\
v_{1}%
\end{array}
\right)  =J^{\mu}L^{\mu}\left(
\begin{array}
[c]{c}%
\sigma_{1}\\
v_{1}%
\end{array}
\right)  .\nonumber
\end{align}
Here, $\sigma_{1}\in X_{\mu,r},\ v_{1}\in Y_{\mu,r}$ and the operators
\begin{equation}
L_{\mu,r}=\Phi^{\prime\prime}\left(  \rho_{\mu}\right)  -{4\pi}\left(
-\Delta_{r}\right)  ^{-1}:X_{\mu,r}\rightarrow X_{\mu,r}^{\ast}%
,\ \ \label{defn-L-mu-r}%
\end{equation}%
\begin{equation}
A_{\mu,r}=\rho_{\mu}:Y_{\mu,r}\rightarrow Y_{\mu,r}^{\ast},\label{defn-A-mu-r}%
\end{equation}%
\begin{equation}
B_{\mu,r}=-\frac{1}{r^{2}}\partial_{r}\left(  r^{2}\cdot\right)  :Y_{\mu
,r}^{\ast}\rightarrow X_{\mu,r},\ \ \ B_{\mu,r}^{\prime}=\partial_{r}%
:X_{\mu,r}^{\ast}\rightarrow Y_{\mu,r},\label{defn-B-mu-r}%
\end{equation}
and
\begin{equation}
J^{\mu}=\left(
\begin{array}
[c]{cc}%
0 & B_{\mu,r}\\
-B_{\mu,r}^{\prime} & 0
\end{array}
\right)  :X_{\mu,r}^{\ast}\times Y_{\mu,r}^{\ast}\rightarrow X_{\mu,r}\times
Y_{\mu,r},\ \label{defn-J-mu}%
\end{equation}%
\begin{equation}
L^{\mu}=\left(
\begin{array}
[c]{cc}%
L_{\mu,r} & 0\\
0 & A_{\mu,r}%
\end{array}
\right)  :X_{\mu,r}\times Y_{\mu,r}\rightarrow X_{\mu,r}^{\ast}\times
Y_{\mu,r}^{\ast}.\label{defn-L-mu}%
\end{equation}
As the triple $\left(  L_{\mu},A_{\mu},B_{\mu}\right)  $ in
(\ref{hamiltonian-EP}) satisfies assumptions (\textbf{G1-4}) in Section
\ref{section-abstract}, the above reduction procedure and Lemma
\ref{L:decom-1} imply that the triple $\left(  L_{\mu,r},A_{\mu,r},B_{\mu
,r}\right)  $ satisfies (\textbf{G1-4}) as well. Thus,
(\ref{hamiltonian-radial}) is a separable Hamiltonian system, for which
Theorem \ref{T:main} is applicable.

\begin{proof}
[Proof of Theorem \ref{Thm1-EP} ii)]By Theorem \ref{T:main}, the linear
stability/instability of (\ref{hamiltonian-radial}) is reduced to find
$n^{u}\left(  \mu\right)  =n^{-}\left(  L_{\mu,r}|_{\overline{R\left(
B_{\mu,r}\right)  }}\right)  $. By the proof of Lemma
\ref{lemma-equivalent-operators} restricted to radial spaces, we have
$n^{-}\left(  L_{\mu,r}\right)  =n^{-}\left(  D_{\mu}^{0}\right)  $ where
$D_{\mu}^{0}$ is defined by (\ref{defn-D-mu-0-general}). Moreover, it holds
that
\begin{equation}
\overline{R\left(  B_{\mu,r}\right)  }=\left(  \ker B_{\mu,r}^{\prime}\right)
^{\perp}=\left(  \ker\partial_{r}\right)  ^{\perp}=\left\{  \rho\in X_{\mu
,r}\ |\ \int_{S_{\mu}}\rho\ dx=0\right\}  . \label{range-mass-constraint}%
\end{equation}
Therefore, to find $n^{-}\left(  L_{\mu,r}|_{\overline{R\left(  B_{\mu
,r}\right)  }}\right)  $ it is equivalent to determine the negative dimensions
of the quadratic form $\left\langle L_{\mu,r}\cdot,\cdot\right\rangle $ under
the mass constraint $\int_{S_{\mu}}\rho\ dx=0$. We divide into three cases.

Case 1: $\frac{d}{d\mu}\left(  \frac{M\left(  \mu\right)  }{R\left(
\mu\right)  }\right)  \ne0$. By (\ref{eqn-V0}) and Lemma
\ref{lemma-formula-V-R}, the steady density $\rho_{\mu}$ satisfies the
equation
\begin{equation}
\Phi^{\prime}\left(  \rho_{\mu}\right)  -4\pi\left(  -\Delta\right)  ^{-1}%
\rho_{\mu}=V_{\mu}\left(  R_{\mu}\right)  =-\frac{M\left(  \mu\right)
}{R_{\mu}},\label{first-order-variation}%
\end{equation}
inside the support $S_{\mu}$. Taking $\partial_{\mu}$ of above equation, we
have
\begin{equation}
L_{\mu}\frac{\partial\rho_{\mu}}{\partial\mu}=\Phi^{\prime\prime}\left(
\rho_{\mu}\right)  \frac{\partial\rho_{\mu}}{\partial\mu}-4\pi\left(
-\Delta\right)  ^{-1}\frac{\partial\rho_{\mu}}{\partial\mu}=-\frac{d}{d\mu
}\left(  \frac{M\left(  \mu\right)  }{R_{\mu}}\right)  ,\ \text{in }S_{\mu
}\label{2nd order variation}%
\end{equation}
Which
implies that
\[
\overline{R\left(  B_{\mu,r}\right)  } =\{ \rho\mid\langle L_{\mu,r}%
\frac{\partial\rho_{\mu}}{\partial\mu},\rho\rangle=0\}
\]
and
\begin{equation}
\left\langle L_{\mu,r}\frac{\partial\rho_{\mu}}{\partial\mu},\frac
{\partial\rho_{\mu}}{\partial\mu}\right\rangle =-\frac{d}{d\mu}\left(
\frac{M\left(  \mu\right)  }{R_{\mu}}\right)  \int_{S_{\mu}}\frac{\partial
\rho_{\mu}}{\partial\mu}dx=-\frac{d}{d\mu}\left(  \frac{M\left(  \mu\right)
}{R_{\mu}}\right)  M^{\prime}\left(  \mu\right)  .
\label{quadratic-form-d-rho-d-mu}%
\end{equation}

Case 1a: $M^{\prime}(\mu)
\neq0$. The above properties immediately yields
\begin{align*}
n^{-}\left(  L_{\mu,r}|_{\overline{R\left(  B_{\mu,r}\right)  }}\right)   &  =%
\begin{cases}
n^{-}\left(  L_{\mu,r}\right)  -1 & \text{if }M^{\prime}(\mu)\frac{d}{d\mu
}\left(  \frac{M\left(  \mu\right)  }{R_{\mu}}\right)  >0\\
n^{-}\left(  L_{\mu,r}\right)  & \text{if }M^{\prime}(\mu)\frac{d}{d\mu
}\left(  \frac{M\left(  \mu\right)  }{R_{\mu}}\right)  <0
\end{cases}
\\
&  =n^{-}\left(  D_{\mu}^{0}\right)  -i_{\mu}.
\end{align*}

Case 1b: $M^{\prime}(\mu)=0$.
In this case, we have
\[
\left\langle L_{\mu,r}\frac{\partial\rho_{\mu}}{\partial\mu},\frac
{\partial\rho_{\mu}}{\partial\mu}\right\rangle =0, \quad\frac{\partial
\rho_{\mu}}{\partial\mu} \in\overline{R( B_{\mu,r})}, \quad\ker L_{\mu
,r}=\left\{  0\right\}  ,
\]
where Lemma \ref{lemma-kernel-D-mu} was used. There exists $\psi
\notin\overline{R( B_{\mu,r})}$. Let
\[
Z_{0} = span\{\psi, \frac{\partial\rho_{\mu}}{\partial\mu}\}, \quad Z_{1} =
\{\rho\in\overline{R( B_{\mu,r})} \mid\langle L_{\mu,r}\psi,\rho\rangle\} =0
\]
and we have
\[
X_{\mu, r} = Z_{0} \oplus Z_{1}, \quad\overline{R( B_{\mu,r})} = Z_{1}
\oplus\mathbf{R} \frac{\partial\rho_{\mu}}{\partial\mu}.
\]
We obtain from Lemma 12.3 in \cite{lin-zeng-hamiltonian} and
\eqref{quadratic-form-d-rho-d-mu} that
\[
n^{-}\left(  L_{\mu,r}|_{\overline{R\left(  B_{\mu,r}\right)  }}\right)  =
n^{-}(L_{\mu, r} |_{Z_{1}}), \quad n^{-} ( L_{\mu, r}) = n^{-}(L_{\mu, r}
|_{Z_{1}}) + n^{-}(L_{\mu, r} |_{Z_{0}}).
\]
It is straight forward to compute $n^{-}(L_{\mu, r} |_{Z_{0}}) =1$ and thus
\[
n^{-}\left(  L_{\mu,r}|_{\overline{R\left(  B_{\mu,r}\right)  }}\right)
=n^{-}\left(  L_{\mu,r}\right)  -1=n^{-}\left(  D_{\mu}^{0}\right)  -i_{\mu}.
\]

Case 2: $\frac{d}{d\mu}\left(  \frac{M\left(  \mu\right)  }{R\left(
\mu\right)  }\right)  =0$. By Lemma \ref{lemma-M-R-no critical}, we have
\[
\int_{S_{\mu}}\frac{\partial\rho_{\mu}}{\partial\mu}dx=M^{\prime}(\mu
)\neq0\Longrightarrow\frac{\partial\rho_{\mu}}{\partial\mu}\notin%
\overline{R\left(  B_{\mu,r}\right)  }.
\]
Therefore,
\[
X_{\mu,r}=\overline{R\left(  B_{\mu,r}\right)  }\oplus\mathbf{R}\frac
{\partial\rho_{\mu}}{\partial\mu}%
\]
which implies
$n^{-}\left(  L_{\mu,r}|_{\overline{R\left(  B_{\mu,r}\right)  }}\right)
=n^{-}\left(  L_{\mu,r}\right)  $.

\end{proof}

\begin{remark}
If $\mu$ belongs to a stable interval, we must have $\frac{d}{d\mu}\left(
\frac{M\left(  \mu\right)  }{R\left(  \mu\right)  }\right)  \neq0$. Indeed,
when $\frac{d}{d\mu}\left(  \frac{M\left(  \mu\right)  }{R\left(  \mu\right)
}\right)  =0$, by (\ref{unstable-dimension-formula-EP}) and Lemma
\ref{lemma-property-A-EP}, we have $n^{u}\left(  \mu\right)  =n^{-}\left(
D_{\mu}^{0}\right)  \geq1$.
\end{remark}

To prove Theorem \ref{Thm1-EP} iii), by Proposition
\ref{prop-discrete-non-radial}, it remains to show that the eigenvalues of the
operator $\mathcal{J}^{\mu}\mathcal{L}^{\mu}$ defined in
(\ref{hamiltonian-radial}) are purely isolated and
\begin{equation}
\ker J^{\mu}L^{\mu}=span\left\{  \left(
\begin{array}
[c]{c}%
\partial_{\mu}\rho_{\mu}\\
0
\end{array}
\right)  \right\}  . \label{ker-radial}%
\end{equation}
We first prove (\ref{ker-radial}) and leave the proof of the discreteness of
eigenvalues of $J^{\mu}L^{\mu}$ to the end of this section.
By \eqref{2nd order variation} we have
\[
L^{\mu}\left(
\begin{array}
[c]{c}%
\partial_{\mu}\rho_{\mu}\\
0
\end{array}
\right)  =-\frac{d}{d\mu}\left(  \frac{M\left(  \mu\right)  }{R\left(
\mu\right)  }\right)  \left(
\begin{array}
[c]{c}%
1\\
0
\end{array}
\right)  ,
\]
thus $span\left\{  \left(
\begin{array}
[c]{c}%
\partial_{\mu}\rho_{\mu}\\
0
\end{array}
\right)  \right\}  \subset\ker J^{\mu}L^{\mu}$. To prove $\ker J^{\mu} L^{\mu
}\subset span\left\{  \left(
\begin{array}
[c]{c}%
\partial_{\mu}\rho_{\mu}\\
0
\end{array}
\right)  \right\}  $, we consider two cases. Suppose $J^{\mu}L^{\mu}\left(
\begin{array}
[c]{c}%
\sigma\\
v
\end{array}
\right)  =0$ for some nonzero $\left(
\begin{array}
[c]{c}%
\sigma\\
v
\end{array}
\right)  \in X_{\mu,r}\times Y_{\mu,r}$. It is easy to check that $v=0$ and
$L_{\mu,r}\sigma=c$ for some constant $c$.

Case 1: $\frac{d}{d\mu}\left(  \frac{M\left(  \mu\right)  }{R\left(
\mu\right)  }\right)  \neq0$. Then
\[
L_{\mu,r}\left(  \sigma+\frac{c}{\frac{d}{d\mu}\left(  \frac{M\left(
\mu\right)  }{R_{\mu}}\right)  }\partial_{\mu}\rho_{\mu}\right)  =0.
\]
This implies that $\sigma=-\frac{c}{\frac{d}{d\mu}\left(  \frac{M\left(
\mu\right)  }{R_{\mu}}\right)  }\partial_{\mu}\rho_{\mu}$, since by Lemma
\ref{lemma-kernel-D-mu}, $\dim L_{\mu,r}=\dim\ker D_{\mu}^{0}=0$.

Case 2: $\frac{d}{d\mu}\left(  \frac{M\left(  \mu\right)  }{R\left(
\mu\right)  }\right)  =0$. Then $\ker L_{\mu,r}=span\left\{  \partial_{\mu
}\rho_{\mu}\right\}  $ and $M^{\prime}\left(  \mu\right)  \neq0$ by Lemma
\ref{lemma-M-R-no critical}. From $L_{\mu,r}\sigma=c$ we have
\[
0=\left\langle L_{\mu,r}\partial_{\mu}\rho_{\mu},\sigma\right\rangle
=\left\langle L_{\mu,r}\sigma,\partial_{\mu}\rho_{\mu}\right\rangle
=cM^{\prime}\left(  \mu\right)  .
\]
Thus $c=0$ and $L_{\mu,r}\sigma=0$, which again imply that $\sigma\in
span\left\{  \partial_{\mu}\rho_{\mu}\right\}  $. This proves
(\ref{ker-radial}).

Next, we prove the turning point principle by using Theorem \ref{Thm1-EP}.

\begin{proof}
[Proof of Theorem \ref{Thm2-TPP}]By Lemma \ref{lemma-small-D-neg}, when $\mu$
is small enough, $n^{-}\left(  D_{\mu}^{0}\right)  =1$. By the proof of Lemma
\ref{lemma-steady-small-density}, when $\mu$ is small we have
\[
\rho_{\mu}=F_{+}\left(  \alpha\theta_{\alpha}\left(  \alpha^{\frac{n_{0}-1}%
{2}}r\right)  \right)  ,\ \ \alpha=\Phi^{\prime}\left(  \mu\right)
,\ n_{0}=\frac{1}{\gamma_{0}-1}.
\]
Here, $\theta_{\alpha}\rightarrow\theta_{0}$ in $C^{1}\left(  0,R\right)  $
for any $R>0$ and $\theta_{0}\,$is the Lane-Emden function satisfying
(\ref{eqn-LE-g-0}). The support radius of $\rho_{\mu}$ is
\[
R_{\mu}=\alpha^{-\frac{n_{0}-1}{2}}S_{\alpha}=\alpha^{-\frac{2-\gamma_{0}%
}{2\left(  \gamma_{0}-1\right)  }}S_{\alpha},
\]
where $S_{\alpha}$ is $C^{1}$ in $\alpha$ and when $\alpha\rightarrow
0,\ S_{\alpha}\rightarrow R_{0}$, the support radius of $\theta_{0}$. The
total mass is
\begin{align*}
M\left(  \mu\right)   &  =4\pi\int_{0}^{R_{\mu}}F_{+}\left(  \alpha
\theta_{\alpha}\left(  \alpha^{\frac{n_{0}-1}{2}}r\right)  \right)  r^{2}dr\\
&  =\alpha^{\frac{1}{2}\frac{\left(  3\gamma_{0}-4\right)  }{\gamma_{0}-1}%
}\int_{0}^{S_{\alpha}}g_{\alpha}\left(  \theta_{\alpha}\left(  s\right)
\right)  s^{2}ds,
\end{align*}
where $g_{\alpha}\rightarrow g_{0}$ in $C^{1}\left(  0,1\right)  $ with
$g_{\alpha},\ g_{0}\ $defined in (\ref{defn-g-alpha}) and (\ref{defn-g-0}).
So
\[
\int_{0}^{S_{\alpha}}g_{\alpha}\left(  \theta_{\alpha}\left(  s\right)
\right)  s^{2}ds\rightarrow\int_{0}^{R_{0}}g_{0}\left(  \theta_{0}\left(
s\right)  \right)  s^{2}ds>0,\ \text{when }\alpha\rightarrow0.
\]
Thus for $\mu$ small, we have: i) $\frac{M\left(  \mu\right)  }{R\left(
\mu\right)  }\thickapprox\alpha=\Phi^{\prime}\left(  \mu\right)  $ and
$\frac{d}{d\mu}\left(  \frac{M\left(  \mu\right)  }{R_{\mu}}\right)  >0$; ii)
$M^{\prime}\left(  \mu\right)  >0$ when $\gamma_{0}\in\left(  \frac{4}%
{3},2\right)  $ and $M^{\prime}\left(  \mu\right)  <0$ when $\gamma_{0}%
\in\left(  \frac{6}{5},\frac{4}{3}\right)  $. Thus when $\mu$ is small the
formula (\ref{formula-unstable-mode-small-mu}) for $n^{u}\left(  \mu\right)
\ $follows from Theorem \ref{Thm1-EP}.

Next, we keep track of the changes of $n^{u}\left(  \mu\right)  $ along the
mass-radius curve by increasing $\mu$. We consider four cases.

Case 1: No critical points of $\frac{M\left(  \mu\right)  }{R\left(
\mu\right)  }$ or $M\left(  \mu\right)  $ are met. Then $\frac{d}{d\mu}\left(
\frac{M\left(  \mu\right)  }{R_{\mu}}\right)  $ and $M^{\prime}\left(
\mu\right)  $ do not change sign. By Lemma \ref{lemma-small-D-neg} and
(\ref{formula-unstable-mode}), $n^{u}\left(  \mu\right)  $ is unchanged.

Case 2: At a critical point $\mu^{\ast}$ of $\frac{M\left(  \mu\right)
}{R_{\mu}}$. The jump formula (\ref{formula-jump-D-mu}) implies that
\[
n^{u}\left(  \mu^{\ast}+\right)  =n^{-}\left(  D_{\mu^{\ast}+}^{0}\right)
-i_{\mu^{\ast}+}=n^{-}\left(  D_{\mu^{\ast}-}^{0}\right)  -i_{\mu^{\ast}%
-}=n^{u}\left(  \mu^{\ast}-\right)  .
\]
That is, the number of unstable modes remains unchanged when crossing
$\mu^{\ast}$.

Case 3: At an extrema (i.e. maximum or minimum) point $\bar{\mu}$ of $M\left(
\mu\right)  $ where $M^{\prime}(\mu)$ changes sign. Then $\frac{d}{d\mu
}\left(  \frac{M\left(  \mu\right)  }{R_{\mu}}\right)  |_{\mu=\bar{\mu}}\neq0$
and $n^{-}\left(  D_{\mu}^{0}\right)  $ is the same in a neighborhood of
$\bar{\mu}$. But $M^{\prime}\left(  \mu\right)  $ changes sign when crossing
$\bar{\mu}$, thus we have
\[
n^{u}\left(  \bar{\mu}+\right)  -n^{u}\left(  \bar{\mu}-\right)  =-(i_{\mu
+}-i_{\mu-})=\pm1,
\]
when $M^{\prime}(\mu)\frac{d}{d\mu}\left(  \frac{M\left(  \mu\right)  }%
{R_{\mu}}\right)  $ changes from $\pm$ to $\mp$ at $\bar{\mu}$. Since
$M^{\prime}\left(  \bar{\mu}\right)  =0$ and $\frac{d}{d\mu}\left(
\frac{M\left(  \mu\right)  }{R_{\mu}}\right)  |_{\mu=\bar{\mu}}\neq0$, when
$\mu\ $is near $\bar{\mu}$, we have $R^{\prime}\left(  \bar{\mu}\right)
\neq0$ and the sign of $M^{\prime}(\mu)\frac{d}{d\mu}\left(  \frac{M\left(
\mu\right)  }{R_{\mu}}\right)  $ is the same as $-M^{\prime}(\mu)R^{\prime
}\left(  \mu\right)  $. Thus $n^{u}\left(  \bar{\mu}+\right)  -n^{u}\left(
\bar{\mu}-\right)  =\pm1$ when $M^{\prime}(\mu)R^{\prime}\left(  \mu\right)  $
changes from $\mp$ to $\pm$ at $\bar{\mu}$, or equivalently the mass-radius
curve bends counterclockwise (clockwise) at $\bar{\mu}$.

Case 4: At a critical but non-extrema point $\tilde{\mu}\ $of $M\left(
\mu\right)  $. Then $\frac{d}{d\mu}\left(  \frac{M\left(  \mu\right)  }%
{R_{\mu}}\right)  |_{\mu=\tilde{\mu}}\neq0$ and $n^{-}\left(  D_{\mu}%
^{0}\right)  $ is the same for $\mu\ $near $\tilde{\mu}$. Since $\tilde{\mu}$
is not an extrema point of $M\left(  \mu\right)  $, the sign of $M^{\prime
}\left(  \mu\right)  $ does not change when crossing $\tilde{\mu}$. Then by
(\ref{formula-unstable-mode}) $n^{u}\left(  \mu\right)  $ does not change when
crossing $\tilde{\mu}$. However, we should note that if $M^{\prime}\left(
\mu\right)  \frac{d}{d\mu}\left(  \frac{M\left(  \mu\right)  }{R_{\mu}%
}\right)  >0$ or equivalently $M^{\prime}\left(  \mu\right)  R^{\prime}\left(
\mu\right)  <0\ $in a neighborhood of $\tilde{\mu}$ excluding $\tilde{\mu}$,
then $n^{u}\left(  \mu\right)  $ has a removable jump discontinuity at
$\tilde{\mu}$ where $n^{u}\left(  \mu\right)  \ $is reduced by one.

Summing up above discussions, we finish the proof of Theorem \ref{Thm2-TPP}.
\end{proof}

Below, we prove Theorem \ref{Thm3-trichotomy} about exponential trichotomy
estimates of (\ref{hamiltonian-radial}).

\begin{proof}
[Proof of Theorem \ref{Thm3-trichotomy}]Conclusion (i) is by Theorem
\ref{Thm1-EP} and \ref{T:main}. Conclusion (ii) and
(\ref{estimate-stable-unstable-EP}) follows directly from Theorem
\ref{T:main}. To prove (\ref{estimate-center-EP-1}) and
(\ref{estimate-center-EP-2}), we consider radial and nonradial perturbations
separately. For nonradial perturbations, Proposition
\ref{prop-discrete-non-radial} implies that all eigenvalues are discrete and
on the imaginary axis. Hence according to the block decomposition and the
anti-self-adjointness of $\mathbf{T}_{3}$ in Theorem \ref{T:main}, the
algebraic growth can only arise from the generalized kernel. By Theorem
\ref{Thm1-EP} i),
we have
\begin{equation}
\left\vert e^{t\mathcal{J}_{\mu}\mathcal{L}_{\mu}}|_{E^{c}\cap\mathbf{X}_{nr}%
}\right\vert =\left\vert e^{t\mathcal{J}_{\mu}\mathcal{L}_{\mu}}%
|_{\mathbf{X}_{nr}}\right\vert \leq C_{0}(1+|t|).
\label{estimate-nonradial-growth}%
\end{equation}
For radial perturbations, when $M^{\prime}(\mu)=0$, by Theorem \ref{T:main2}
i), we have
\[
\left\vert e^{t\mathcal{J}_{\mu}\mathcal{L}_{\mu}}|_{E^{c}\cap\mathbf{X}_{r}%
}\right\vert \leq C_{0}(1+|t|)^{2}%
\]
and (\ref{estimate-center-EP-2}) follows by combining it with
(\ref{estimate-nonradial-growth}). When $M^{\prime}(\mu)\neq0$, we check that
$L_{\mu,r}|_{\overline{R\left(  B_{\mu,r}\right)  }}$ is non-degenerate. Let
$W_{1}=span\left\{  \frac{\partial\rho_{\mu}}{\partial\mu}\right\}  $. Since
$\int\frac{\partial\rho_{\mu}}{\partial\mu}dx=M^{\prime}(\mu)\neq0$, there is
an invariant decomposition $\mathbf{X}_{r}=R\left(  B_{\mu,r}\right)  \oplus
W_{1}$. When $M^{\prime}(\mu)\frac{d}{d\mu}\left(  \frac{M\left(  \mu\right)
}{R_{\mu}}\right)  \neq0$, by the proof of Theorem \ref{Thm1-EP} ii),
$\overline{R\left(  B_{\mu,r}\right)  }$ is the $L_{\mu,r}$-orthogonal
complement space of $W_{1}=span\left\{  \frac{\partial\rho_{\mu}}{\partial\mu
}\right\}  $. The non-degeneracy of $L_{\mu,r}|_{\overline{R\left(  B_{\mu
,r}\right)  }}$ follows since $\ker L_{\mu,r}=\left\{  0\right\}  $ and
$L_{\mu,r}|_{W_{1}}$ is non-degenerate by (\ref{quadratic-form-d-rho-d-mu}).
When $\frac{d}{d\mu}\left(  \frac{M\left(  \mu\right)  }{R_{\mu}}\right)  =0$,
we have $\ker L_{\mu,r}=W_{1}$ and the non-degeneracy of $L_{\mu
,r}|_{\overline{R\left(  B_{\mu,r}\right)  }}$ also follows. Thus by Theorem
\ref{T:main2} iii), we have $\left\vert e^{t\mathcal{J}_{\mu}\mathcal{L}_{\mu
}}|_{E^{c}\cap\mathbf{X}_{r}}\right\vert \leq C_{0}$, which implies Conclusion
(iv) and (\ref{estimate-center-EP-1}).
\end{proof}

It remains to prove that the eigenvalues of the linearized problem
(\ref{hamiltonian-radial}) for radial perturbations are all discrete by
Theorem \ref{T:main}. We need the following Hardy's inequality (\cite{hardy
inequality} \cite{jang-instability}).

\begin{lemma}
[Hardy's inequality]Let $k$ be a real number and $g$ be a function satisfying
\[
\int_{0}^{1}s^{k}\left(  g^{2}+\left\vert g^{\prime}\right\vert ^{2}\right)
ds<\infty.
\]

i) If $k>1$, then we have
\[
\int_{0}^{1}s^{k-2}g^{2}ds\lesssim\int_{0}^{1}s^{k}\left(  g^{2}+\left\vert
g^{\prime}\right\vert ^{2}\right)  ds.
\]

ii) If $k<1$, then $g$ has a trace at $x=0$ and%
\begin{equation}
\int_{0}^{1}s^{k-2}\left(  g-g\left(  0\right)  \right)  ^{2}ds\lesssim
C\int_{0}^{1}s^{k}\left\vert g^{\prime}\right\vert ^{2}ds.
\label{hardy-k-less-1}%
\end{equation}

\end{lemma}

Define the function space $Z_{\mu,r}$ to be the closure of $D\left(  B_{\mu
,r}A_{\mu,r}\right)  \subset Y_{\mu, r}$ under the graph norm
\begin{align*}
\left\Vert v\right\Vert _{Z_{\mu,r}}  &  =\left\Vert v\right\Vert _{Y_{\mu,r}%
}+\left\Vert B_{\mu,r}A_{\mu,r}v\right\Vert _{X_{\mu,r}}\\
&  =\left(  \int_{0}^{R_{\mu}}\rho_{\mu}\left\vert v\right\vert ^{2}%
r^{2}dr\right)  ^{\frac{1}{2}}+\left(  \int_{0}^{R_{\mu}}\Phi^{\prime\prime
}\left(  \rho_{\mu}\right)  \left\vert \frac{1}{r^{2}}\partial_{r}\left(
r^{2}\rho_{\mu}v\right)  \right\vert ^{2}r^{2}dr\right)  ^{\frac{1}{2}}.
\end{align*}
By Theorem \ref{T:main}, to show the discreteness of eigenvalues for radial
perturbations it suffices to show the following compactness lemma.

\begin{lemma}
The embedding $Z_{\mu,r}\hookrightarrow Y_{\mu,r}$ is compact.
\end{lemma}

\begin{proof}
First, near the support radius $R_{\mu}$ we have $\rho_{\mu}\left(  r\right)
\approx\left(  R_{\mu}-r\right)  ^{\frac{1}{\gamma-1}}$. This is well-known
for Lane-Emden stars. To be self-contained, we give a proof for general
equation of states. By (\ref{first-order-ode}), we have
\[
y_{\mu}^{\prime}\left(  R_{\mu}\right)  =-\frac{4\pi}{R_{\mu}^{2}}\int%
_{0}^{R_{\mu}} s^{2} F_{+}\left(  y_{\mu}\left(  s\right)  \right)
ds=-\frac{1 }{R_{\mu}^{2}}M\left(  \mu\right)  <0.
\]
Thus for $r$ near $R_{\mu}$, $y_{\mu}\left(  r\right)  \approx R_{\mu}-r\,$.
Since $\rho_{\mu}\left(  r\right)  =F_{+}\left(  y_{\mu}\left(  r\right)
\right)  $ and $F_{+}\left(  y\right)  \approx y^{\frac{1}{\gamma-1}}$ for
$0<y\ll1$, we deduce that for $r$ near $R_{\mu}$
\begin{equation}
\rho_{\mu}\left(  r\right)  \approx\left(  y_{\mu}\left(  r\right)  \right)
^{\frac{1}{\gamma-1}}\approx\left(  R_{\mu}-r\right)  ^{\frac{1}{\gamma-1}}.
\label{estimate-rho-mu-near-R}%
\end{equation}
Then for $r$ near $R_{\mu},$%
\begin{equation}
\Phi^{\prime\prime}\left(  \rho_{\mu}\left(  r\right)  \right)  \approx
\rho_{\mu}\left(  r\right)  ^{\gamma-2}\approx\left(  R_{\mu}-r\right)
^{\frac{\gamma-2}{\gamma-1}} . \label{estimate-phi-rho-near-R}%
\end{equation}

Let $r_{2}<R_{\mu}$ and $R_{\mu}-r_{2}$ be small enough so that
(\ref{estimate-rho-mu-near-R}) and (\ref{estimate-phi-rho-near-R}) are valid
in $\left(  r_{2},R_{\mu}\right)  $. Then for any $v\in Z_{\mu,r}$, we have
\begin{align*}
&  \int_{r_{2}}^{R_{\mu}}\Phi^{\prime\prime}\left(  \rho_{\mu}\right)
\left\vert \frac{1}{r^{2}}\partial_{r}\left(  r^{2}\rho_{\mu}v\right)
\right\vert ^{2}r^{2}dr\\
&  \gtrsim\int_{r_{2}}^{R_{\mu}}\left(  R_{\mu}-r\right)  ^{\frac{\gamma
-2}{\gamma-1}}\left\vert \partial_{r}\left(  r^{2}\rho_{\mu}v\right)
\right\vert ^{2}dr\text{ }\\
&  \gtrsim\int_{r_{2}}^{R_{\mu}}\left(  R_{\mu}-r\right)  ^{\frac{\gamma
-2}{\gamma-1}-2}\left\vert r^{2}\rho_{\mu}v\right\vert ^{2}dr\text{ (by
Hardy's inequality (\ref{hardy-k-less-1}))}\\
&  \gtrsim\int_{r_{2}}^{R_{\mu}}\left(  R_{\mu}-r\right)  ^{-1}\rho_{\mu}%
v^{2}dr\text{ (by (\ref{estimate-rho-mu-near-R}))}\\
&  \gtrsim\left(  R_{\mu}-r_{2}\right)  ^{-1}\int_{r_{2}}^{R_{\mu}}\rho_{\mu
}v^{2}dr.
\end{align*}
Thus,
\begin{equation}
\int_{r_{2}}^{R_{\mu}}\rho_{\mu}\left\vert v\right\vert ^{2}r^{2}%
dr\lesssim\left(  R_{\mu}-r_{2}\right)  \left\Vert B_{\mu,r}A_{\mu
,r}v\right\Vert _{X_{\mu,r}}^{2}. \label{estimate-norm-near-R}%
\end{equation}

Let $r_{1}\in\left(  0,R_{\mu}\right)  $ be small enough so that
\[
\frac{1}{2}\mu\leq\rho_{\mu}\left(  r\right)  \leq\mu,\ \ \ \forall
r\in\left(  0,r_{1}\right)  ,
\]
then
\[
0<\delta_{1}\left(  \mu\right)  \leq\Phi^{\prime\prime}\left(  \rho_{\mu
}\right)  \leq\delta_{2}\left(  \mu\right)  ,\ \ \forall r\in\left(
0,r_{1}\right)  ,\ \
\]
where $\delta_{1}\left(  \mu\right)  =\min_{\rho\in\left(  \frac{1}{2}\mu
,\mu\right)  }\Phi^{\prime\prime}\left(  \rho\right)  $ and $\delta_{2}\left(
\mu\right)  =\max_{\rho\in\left(  \frac{1}{2}\mu,\mu\right)  }\Phi
^{\prime\prime}\left(  \rho\right)  $. We have
\begin{align*}
&  \int_{0}^{r_{1}}\Phi^{\prime\prime}\left(  \rho_{\mu}\right)  \left\vert
\frac{1}{r^{2}}\partial_{r}\left(  r^{2}\rho_{\mu}v\right)  \right\vert
^{2}r^{2}dr\\
&  \gtrsim\int_{0}^{r_{1}}\frac{1}{r^{2}}\left\vert \partial_{r}\left(
r^{2}\rho_{\mu}v\right)  \right\vert ^{2}dr\text{ }\gtrsim\int_{0}^{r_{1}%
}\left(  \rho_{\mu}v\right)  ^{2}dr\text{ (by (\ref{hardy-k-less-1}))}\\
&  \gtrsim r_{1}^{-2}\int_{0}^{r_{1}}v^{2}r^{2}dr.
\end{align*}
Thus,%
\begin{equation}
\int_{0}^{r_{1}}\rho_{\mu}\left\vert v\right\vert ^{2}r^{2}dr\lesssim
r_{1}^{2}\left\Vert B_{\mu,r}A_{\mu,r}v\right\Vert _{X_{\mu,r}}^{2}.
\label{estimate-norm-near-0}%
\end{equation}
Denote $B_{Z}=\left\{  v\in Z_{\mu,r}\ |\ \left\Vert v\right\Vert _{Z_{\mu,r}%
}\leq1\right\}  \ $to be the unit ball in $Z_{\mu,r}$. Then for any
$\varepsilon>0$, by estimates (\ref{estimate-norm-near-R}) and
(\ref{estimate-norm-near-0}), we can choose $0<r_{1}<r_{2}<R_{\mu}$ such that
\[
\int_{0}^{r_{1}}\rho_{\mu}\left\vert v\right\vert ^{2}r^{2}dr+\int_{r_{2}%
}^{R_{\mu}}\rho_{\mu}\left\vert v\right\vert ^{2}r^{2}dr\leq\varepsilon
,\ \ \forall\ v\in B_{Z}.
\]
The compactness of $Z_{\mu,r}\hookrightarrow Y_{\mu,r}$ follows from above
estimate and the compactness of the embedding $Z_{\mu,r}\hookrightarrow
L^{2}\left(  r_{1},r_{2}\right)  $.
\end{proof}

\begin{remark}
The stability criterion $L_{\mu,r}|_{\overline{R\left(  B_{\mu,r}\right)  }%
}\geq0$ has the following physical meaning. By
(\ref{quadratic-L-2nd-variation}), the quadratic form $\left\langle L_{\mu
,r}\rho,\rho\right\rangle $ is the second order variation of the energy
functional $E_{\mu}\left(  \rho\right)  $ defined in
(\ref{energy functional -rho}). By (\ref{range-mass-constraint}), the space
$\overline{R\left(  B_{\mu,r}\right)  }$ consists of perturbations satisfying
the mass constraint. Thus, our stability criterion verifies Chandrasekhar's
variational principle that stable states should be energy minimizers under the
mass constraint. (see also Remark \ref{rmk-variational-steady})
\end{remark}

\begin{remark}
\label{rm-Eddington-eqn}

In the astrophysical literature, the linear radial oscillations were usually studied
through the singular Sturm-Liouville equation
\begin{equation}
\frac{d}{dr}\left(  \Gamma_{1}P_{\mu}\frac{1}{r^{2}}\frac{d}{dr}\left(
r^{2}\xi\right)  \right)  -\frac{4}{r}\frac{dP_{\mu}}{dr}\xi+\omega^{2}%
\rho_{\mu}\xi=0,\label{ODE-SL-radial-oscilation}%
\end{equation}
with the boundary conditions
\begin{equation}
\xi\left(  0\right)  =0\text{ and }\xi\left(  R_{\mu}\right)  \text{ is
finite. }\label{SL-eqn-BC}%
\end{equation}
Here, $\xi$ is the linearized Lagrangian displacement in the radial direction, $P_{\mu
}=P\left(  \rho_{\mu}\right),\ \Gamma_{1}=\frac{\rho_{\mu}P^{\prime}\left(
\rho_{\mu}\right)  }{P\left(  \rho_{\mu}\right)  }$ is the local Polytropic
index and $i\omega$ is the eigenvalue. The equation
(\ref{ODE-SL-radial-oscilation}) was first derived by Eddington in 1918
(\cite{eddington1919}) and had been widely used in later works (e.g.
\cite{ledoux58} \cite{cox-book} \cite{jang-instability} \cite{lin-ss-1997}
\cite{makino15}). For Polytropic stars $P\left(  \rho\right)  =K\rho^{\gamma}%
$, $\Gamma_{1}=\gamma$, (\ref{ODE-SL-radial-oscilation}) is greatly simplified
and can be used to show $\gamma=\frac{4}{3}$ is the critical index for
stability (\cite{ledoux58} \cite{lin-ss-1997}). However, for general equation
of states it is difficult to get explicit stability criteria such as TPP in
Theorem \ref{Thm2-TPP} by (\ref{ODE-SL-radial-oscilation}). Moreover, since
the Sturm-Liouville problem (\ref{ODE-SL-radial-oscilation}) is singular near
$r=0$ and $R_{\mu}$, it is highly nontrivial (\cite{beyer-JMP-1}
\cite{beyer-JMP-2} \cite{beyer-schmidt} \cite{lin-ss-1997} \cite{makino15}
\cite{jang-makino-LEP19}) to prove self-adjointness and discreteness of
eigenvalues which were taken for granted in the astrophysical literature.

By the separable Hamiltonian formulation (\ref{hamiltonian-radial}), the
eigenvalue equation can be written as (see (\ref{2nd order eqn-v}))
\begin{equation}
B_{\mu,r}^{\prime}L_{\mu,r}B_{\mu,r}A_{\mu,r}v=\omega^{2}%
v,\label{eigenvalue-problem-factorized}%
\end{equation}
which is equivalent to (\ref{ODE-SL-radial-oscilation}) by explicit
calculations. There are several advantages of the factorized form
(\ref{eigenvalue-problem-factorized}) over (\ref{ODE-SL-radial-oscilation}).
First, each factor in (\ref{eigenvalue-problem-factorized}) has a clear
physical meaning related to the variational structures of steady states or the
physical constraint etc. Second, the form in
(\ref{eigenvalue-problem-factorized}) makes it convenient to prove properties
of the operator $B_{\mu,r}^{\prime}L_{\mu,r}B_{\mu,r}A_{\mu,r}\ $such as the
self-adjointness and discreteness of eigenvalues. This approach is rather
flexible and has been used in the recent works on the stability of rotating
stars (\cite{lin-wang-rotating}), relativistic stars and globular clusters
(\cite{mahir-lin-rein} \cite{lin-hadzic-tpp-euler-einstein}).
\end{remark}

\subsection{Examples}

We apply the stability criteria for several examples of gaseous stars.

1. Polytropic stars

For Polytropic stars, $P\left(  \rho\right)  =K\rho^{\gamma}$ with $\gamma
\in\left(  \frac{6}{5},2\right)  $, then by Lemma \ref{lemma-Morse-3-below},
we have $n^{-}\left(  D_{\mu}^{0}\right)  =1$ for any $\mu>0$. The functions
$M\left(  \mu\right)  ,R_{\mu}$ are given by (\ref{relation-M-R-polytropic}).
For any $\gamma>1,\ $we have$\frac{d}{d\mu}\left(  \frac{M\left(  \mu\right)
}{R_{\mu}}\right)  >0$ for all $\mu>0$. When $\gamma\in\left(  \frac{6}%
{5},\frac{4}{3}\right)  $ we have $M^{\prime}\left(  \mu\right)  <0\ $and thus
$i_{\mu}=0$. Then it follows from Theorems \ref{Thm1-EP} and
\ref{Thm3-trichotomy} that for any $\mu>0,\ \rho_{\mu}$ is unstable with
$n^{u}\left(  \mu\right)  =1$ and there is Lyapunov stability on the co-dim
$2$ center space. When $\gamma\in\left(  \frac{4}{3},2\right)  $, we have
$M^{\prime}\left(  \mu\right)  >0\ $and thus $i_{\mu}=1$. By Theorems
\ref{Thm3-trichotomy} (iv), linear Lyapunov stability holds for any $\mu>0$.
The case $\gamma=\frac{4}{3}$ is the critical index for stability. In this
case, we have $M^{\prime}\left(  \mu\right)  =0$. Thus, $i_{\mu}=1$ and we
have spectral stability. In \cite{deng-et-2002}, nonlinear instability was
shown for $\gamma=\frac{4}{3}\ $in the sense that for any small perturbation
with positive total energy of stationary solutions, either the support of the
density will go to infinity or singularity forms in the solution in finite time.

2. White dwarf stars

Next, we consider white dwarf stars (\cite{chandra-book-stellar}) with
$P_{\text{w}}\left(  \rho\right)  =Af\left(  x\right)  $ and $\rho=Bx^{3}$,
where $A,B$ are two constants and
\begin{align}
f\left(  x\right)   &  =x\left(  x^{2}+1\right)  ^{\frac{1}{2}}\left(
2x^{2}-3\right)  +3\ln\left(  x+\sqrt{1+x^{2}}\right)
\label{defn-f-white dwarf}\\
&  =8\int_{0}^{x}\frac{u^{4}du}{\sqrt{1+u^{2}}}.\nonumber
\end{align}
Then $P_{\text{w}}\left(  \rho\right)  $ satisfies (\ref{assumption-P2}) with
$\gamma_{0}=\frac{5}{3}$. Therefore, for any center density $\mu\in\left(
0,\infty\right)  $, there exists a unique non-rotating star $\rho_{\mu}\left(
\left\vert x\right\vert \right)  $ (see Remark \ref{R:equilibria}). It was
shown in \cite{staruss-wu-white dwarf} (see also \cite{uggala-Newtonian-stars}%
) that $M^{\prime}\left(  \mu\right)  >0$ for any $\mu>0$.

\begin{lemma}
\label{lemma-monotone-R-white dwarf}Assume $P\left(  \rho\right)  $ satisfies
(\ref{assumption-P2}) with $\gamma_{0}\in\left(  \frac{4}{3},2\right)  $. Let
$\mu_{0}\in(0,+\infty]$ be the such that $M^{\prime}(\mu) \ge0$ on $[0,
\mu_{0})$.
Then $\frac{d}{d\mu}\left(  \frac{M\left(  \mu\right)  }{R_{\mu}}\right)  >0$
for any $\mu\in\left(  0,\mu_{0}\right)  $.
\end{lemma}

\begin{proof}
By the proof of Theorem \ref{Thm2-TPP} we have $\frac{d}{d\mu}\left(
\frac{M\left(  \mu\right)  }{R_{\mu}}\right)  >0\ $and $M^{\prime}\left(
\mu\right)  >0\ $when $\mu$ is small enough.
Suppose the conclusion of the lemma is not true. Let $\mu_{1}\in\left(
0,\mu_{0}\right)  $ be the first zero of $\frac{d}{d\mu}\left(  \frac{M\left(
\mu\right)  }{R_{\mu}}\right)  $. Then $\frac{d}{d\mu}\left(  \frac{M\left(
\mu\right)  }{R_{\mu}}\right)  >0$ for all $\mu\in\left(  0,\mu_{1}\right)  $.
Consequently, by Lemma \ref{lemma-small-D-neg},
$n^{-}\left(  D_{\mu}^{0}\right)  =1$ for all $\mu\in\left(  0,\mu_{1}\right)
$. At $\mu=\mu_{1}$, we have $n^{-}\left(  D_{\mu_{1}}^{0}\right)  \geq1$ (by
Lemma \ref{lemma-property-A-EP} iii)) and $0$ is an eigenvalue of $D_{\mu_{1}%
}^{0}$. Since $M^{\prime}(\mu_{1}) >0$ due to our assumption and Lemma
\ref{lemma-kernel-D-mu}, when $\mu<\mu_{1}$ and $\left\vert \mu-\mu
_{1}\right\vert $ is small enough, we have $\frac{d}{d\mu}\left(
\frac{M\left(  \mu\right)  }{R_{\mu}}\right) M^{\prime}\left(  \mu\right)
>0$. Therefore $i_{\mu}=1$ according to \eqref{defn-imu} and thus Proposition
\ref{prop-jump-index} implies
$n^{-}\left( D_{\mu}^{0}\right)  =n^{-}\left(  D_{\mu_{1}}^{0}\right)
+1\geq2$. This is in contradiction to that $n^{-}\left(  D_{\mu}^{0}\right)
=1$ for $\mu\in\left(  0,\mu_{1}\right)  $.
\end{proof}


\begin{corollary}
\label{cor-stability-white dwarf}White dwarf stars $\rho_{\mu}\left(
\left\vert x\right\vert \right)  \ $are linearly stable for any center density
$\mu>0$.
\end{corollary}

\begin{proof}
Lemmas \ref{lemma-monotone-R-white dwarf} and \ref{lemma-small-D-neg} imply
that $n^{-}\left(  D_{\mu}^{0}\right)  =1$ for all $\mu>0$. Since $M^{\prime
}\left(  \mu\right)  >0$ and $\frac{d}{d\mu}\left(  \frac{M\left(  \mu\right)
}{R_{\mu}}\right)  >0$ for $\mu\in\left(  0,\infty\right)  $, linear Lyapunov
stability of $\rho_{\mu}$ follows from Theorem \ref{Thm3-trichotomy} (iv).
\end{proof}

\begin{remark}
The mass of white dwarf stars has a finite upper bound $M_{\infty}=\lim
_{\mu\rightarrow\infty}M\left(  \mu\right)  $, which was known as
Chandrasekhar's limit (\cite{chandra-book-stellar} \cite{chandra-nobel}). We
note that for white dwarf stars, $P_{\text{w}}\left(  \rho\right)
\approx2AB^{-\frac{4}{3}}\rho^{\frac{4}{3}}$ when $\rho$ is large. The
Chandrasekhar limit $M_{\infty}$ is exactly the mass of the Polytropic star
with $P\left(  \rho\right)  =2AB^{-\frac{4}{3}}\rho^{\frac{4}{3}}$, which is
independent of $\mu$ by (\ref{relation-M-R-polytropic}).
\end{remark}

3. More general equation of states

Last, we consider general equation of states $P\left(  \rho\right)  $
satisfying (\ref{assumption-P1})-(\ref{assumption-P2}). Assume $\gamma_{0}%
\in\left(  \frac{4}{3},2\right)  $ in (\ref{assumption-P2}). Indeed,
$\gamma_{0}=\frac{5}{3}\ $for most physical equation of states including white
dwarf stars. Then for $\mu$ small, we have
\[
n^{-}\left(  D_{\mu}^{0}\right)  =1,\ \ M^{\prime}\left(  \mu\right)
>0,\ \ \frac{d}{d\mu}\left(  \frac{M\left(  \mu\right)  }{R_{\mu}}\right)
>0.
\]
Let
\[
\mu_{0} = \inf\{ \mu>0\mid M^{\prime}(\mu) <0\} \in(0,+\infty].
\]
If $\mu_{0} <+\infty$ and
$M^{\prime}(\mu) <0$ for $0< \mu-\mu_{0} \ll1$,
we denote
\[
\mu_{1} = \sup\{ \mu> \mu_{0} \mid M^{\prime}(\mu^{\prime}) <0, \forall
\mu^{\prime}\in(\mu_{0}, \mu) \} \in(\mu_{0},+\infty].
\]

\begin{corollary}
\label{cor-general-stars}Assume $P\left(  \rho\right)  $ satisfies
(\ref{assumption-P2}) with $\gamma_{0}\in\left(  \frac{4}{3},2\right)  $. Then
the non-rotating star $\rho_{\mu}\left(  \left\vert x\right\vert \right)  $ is
linearly stable for $\mu\in\left(  0,\mu_{0}\right)  $. If $\mu_{0}<+\infty$,
then $\rho_{\mu}$ is linearly unstable for $\mu\in\left(  \mu_{0},\mu
_{1}\right)  $ and $n^{u}\left(  \mu\right)  =1$.
\end{corollary}

\begin{proof}
Linear stability of $\rho_{\mu}\left(  \left\vert x\right\vert \right)  $ for
$\mu\in\left(  0,\mu_{0}\right)  $ follows as in Corollary
\ref{cor-stability-white dwarf}. When $\mu_{0}<\infty$, linear instability of
$\rho_{\mu}$ for $\mu\in\left(  \mu_{0},\mu_{1}\right)  $ and $n^{u}\left(
\mu\right)  =1\ $follow from Theorem \ref{Thm2-TPP}.
\end{proof}

If $M\left(  \mu\right)  $ has isolated extrema points, then $\mu_{0}%
,\ \mu_{1}$ are the first maximum and minimum points respectively. Below, we
give examples of $P\left(  \rho\right)  \ $for which the maximum of $M\left(
\mu\right)  $ is obtained at a finite center density, which gives the first
transition point of stability. As in \cite{uggala-Newtonian-stars}, we
consider asymptotically polytropic equation of states satisfying that: for
some positive constants $a_{0},a_{1},n_{0},n_{1},c_{-},c_{+},$

i)%

\begin{equation}
P\left(  \rho\right)  =c_{-}\rho^{\frac{n_{0}+1}{n_{0}}}\left(  1+O\left(
\rho^{\frac{a_{0}}{n_{0}}}\right)  \right)  ,\ \ \text{when\ }\rho
\rightarrow0; \label{asymp-P-small}%
\end{equation}

ii)
\begin{equation}
P\left(  \rho\right)  =c_{+}\rho^{\frac{n_{1}+1}{n_{1}}}\left(  1+O\left(
\rho^{-\frac{a_{1}}{n_{1}}}\right)  \right)  ,\ \ \text{when\ }\rho
\rightarrow+\infty. \label{asymp-P-large}%
\end{equation}
Denote $\gamma_{0}=\frac{n_{0}+1}{n_{0}}$ and $\gamma_{\infty}=\frac{n_{1}%
+1}{n_{1}}$. By Theorem 5.5 in \cite{uggala-Newtonian-stars}, when $n_{1}%
\in\left(  0,5\right)  ,$ to first order, the mass--radius relation for high
central pressures is approximated by the mass--radius relation for an exact
polytrope with polytropic index $n_{1}$. That is, when $\mu$ is large enough,
\[
M\left(  \mu\right)  \propto\mu^{\frac{3-n_{1}}{2n_{1}}}=\mu^{\frac{1}%
{2}\left(  3\gamma_{1}-4\right)  },\ \ R_{\mu}\propto\mu^{\frac{1-n_{1}%
}{2n_{1}}}=\mu^{\frac{1}{2}\left(  \gamma_{1}-2\right)  }.
\]
Therefore, when $n_{1}>3$ (i.e. $\gamma_{\infty}<\frac{4}{3}$), $\frac{d}%
{d\mu}\left(  \frac{M\left(  \mu\right)  }{R_{\mu}}\right)  >0$ and
$M^{\prime}\left(  \mu\right)  <0$ for sufficiently large $\mu$. Thus for
large $\mu$, we have $i_{\mu}=0$ and $\rho_{\mu}$ is linearly unstable by
Theorem \ref{Thm1-EP}. When $n_{0}<3$ (i.e. $\gamma_{0}>\frac{4}{3})$, we have
$M^{\prime}\left(  \mu\right)  >0$ for $\mu$ small enough. Thus, the
transition of stability must occur at some $\mu>0$.

By Theorem 5.4 in \cite{uggala-Newtonian-stars}, when $\gamma_{0}>\frac{4}{3}$
and $\gamma_{\infty}<\frac{6}{5}$ (i.e. $n_{1}>5$), the mass--radius relation
for high central pressures possesses a spiral structure, with the spiral given
by%
\begin{equation}
\left(
\begin{array}
[c]{c}%
\ R_{\mu}\\
M\left(  \mu\right)
\end{array}
\right)  =\left(
\begin{array}
[c]{c}%
\ R_{0}\\
M_{0}%
\end{array}
\right)  +\left(  \frac{1}{\alpha}\right)  ^{\gamma_{1}}\mathcal{BJ}\left(
\gamma_{2}\ln\frac{1}{\alpha}\right)  b+o\left(  \left(  \frac{1}{\alpha
}\right)  ^{\gamma_{1}}\right)  ,\ \mu\gg1, \label{mass-radial-spiral}%
\end{equation}
where $\alpha=\Phi^{\prime}\left(  \mu\right)  ,\ \ R_{0}$ and $M_{0}$ are
constants, $\mathcal{B}$ is a non-singular matrix, and $b$ a non-zero vector.
The matrix $\mathcal{J}\left(  \varphi\right)  $ $\in SO\left(  2\right)  $
describes a rotation by an angle $\varphi$, and the constants $\gamma_{1}$ and
$\gamma_{2}\ $are given by%
\[
\gamma_{1}=\frac{1}{4}\left(  n_{1}-5\right)  ,\ \gamma_{2}=\frac{1}{4}%
\sqrt{7n_{1}^{2}-22n_{1}-1}.
\]
Thus, when $\mu\rightarrow\infty$, the mass $M\left(  \mu\right)  $ has
infinitely many extrema points. We claim that at each of these extrema points,
the number of unstable modes $n^{u}\left(  \mu\right)  $ must increase by $1$
and in particular $n^{u}\left(  \mu\right)  \rightarrow\infty$ when
$\mu\rightarrow\infty$. Indeed, for large $\mu\ $the mass-radius curve must
spiral counterclockwise and then by Theorem \ref{Thm2-TPP} $n^{u}\left(
\mu\right)  $ increases by $1$ when crossing any mass extrema of $M\left(
\mu\right)  $ on the spiral. Suppose not, the mass-radius curve spiral
clockwise when $\mu\rightarrow\infty$. Then by Theorem \ref{Thm2-TPP}
$n^{u}\left(  \mu\right)  $ decreases by $1$ when crossing each mass extrema
of $M\left(  \mu\right)  $ on the spiral. Therefore, after crossing finitely
many mass extrema in the spiral, $n^{u}\left(  \mu\right)  $ must become zero.
Let $\mu^{\ast}$ be the first mass extrema in the spiral such that
$n^{u}\left(  \mu\right)  =0$ for $\mu$ slightly less than $\mu^{\ast}$. Then
for $\mu$ slightly less than $\mu^{\ast}$, we have $n^{-}\left(  D_{\mu}%
^{0}\right)  =1$ and $i_{\mu}=1$ which implies that $M^{\prime}(\mu)R^{\prime
}\left(  \mu\right)  <0$. Thus when crossing $\mu^{\ast}$, the sign of
$M^{\prime}(\mu)R^{\prime}\left(  \mu\right)  $ must changed from $-$ to $+$,
which is in contradiction to the assumption that the spiral is clockwise. This
proves that the mass-radius spiral can only be counterclockwise.


We summarize the above discussions in the following.

\begin{corollary}
\label{cor-asymp-poly}Consider asymptotically polytropic $P\left(
\rho\right)  $ satisfying (\ref{asymp-P-small})-(\ref{asymp-P-large}). Assume
$\gamma_{0}\in\left(  \frac{4}{3},2\right)  $ (i.e. $n_{0}\in\left(
1,3\right)  \ $in (\ref{asymp-P-small})). Then when $n_{1}\in\left(
3,5\right)  $ or $n_{1}>5$ with $n_{1}$ defined in (\ref{asymp-P-large}),
there must be transition point of stability in the sense of Corollary
\ref{cor-general-stars}. Moreover, $\rho_{\mu}$ is unstable when $\mu$ is
large enough. When $n_{1}>5$, $n^{u}\left(  \mu\right)  \rightarrow\infty$
when $\mu\rightarrow\infty$.
\end{corollary}

\begin{remark}
White dwarf stars are supported by the pressure due to cold degenerate
electrons, as given by the equation of state (\ref{defn-f-white dwarf}). When
the density is high enough, the pressure due to cold degenerate neutrons
should be taken into account. For such modified equation of states, the
maximal mass (Chandrasekhar's limit) is indeed achieved at a finite center
density $\mu_{0}<\infty$. Then by Corollary \ref{cor-general-stars}$\ \mu_{0}$
is the first transition point of stability and non-rotating stars with center
density slightly larger than $\mu_{0}$ become unstable. We refer to Figure
11.2 and Section 11.4 in \cite{weinberg book} for such mass-radius curve and
physical explanations. If the stars are much more compact than the one with
Chandrasekhar limit, then relativistic effects can not be ignored and
Euler-Einstein model should be used. Similar turning point principle can be
derived for stability of relativistic compact stars modeled by Euler-Einstein
equation (\cite{mahir-lin-rein} \cite{lin-hadzic-tpp-euler-einstein}).
\end{remark}

\section{Appendix: Lagrangian formulation and Hamiltonian structure}\label{A:Lag}

In this appendix we formally outline the Lagrangian formulation of the Euler-Poisson system \eqref{continuity-EP}-\eqref{Poisson-EP} and its linearization. Let $\big(\rho_\mu (|x|), u(x)\equiv0\big)$ be the non-rotating star supported on the ball $S_\mu \subset \BFR^3$ with radius $R(\mu)$, where $\mu = \rho_\mu(0)$. We simply take $S_\mu$ as the reference domain in the Lagrangian framework and define the (abstract) configuration space of Lagrangian maps as
\[
\Lambda = \{ \text{diffeomorphism } \CX: S_\mu \to \CX(S_\mu) \subset \BFR^3\}.
\]
For any reference density $\rho_*: S_\mu \to \BFR^+ \cup\{0\}$ and a path of Lagrangian maps $\CX(t) \in \Lambda$, the action functional $\CA$ is given by
\[
\CA= \int \left( \int_{S_\mu}  \frac 12  |\CX_t|^2 \rho_* dy - \int_{\CX(t, S_\mu)} \Phi (\rho) dx +  \frac 1{8\pi}\int_{\BFR^3} |\nabla V|^2 dx \right) dt,
\]
where the enthalpy $\Phi(\rho)$ is defined in \eqref{defn-enthalpy}, the gravitational potential $V(t, x)$ by  \eqref{Poisson-EP} (or equivalently $V = |x|^{-1} * \rho$),  and the physical density $\rho$ in the Eulerian coordinates is given by
\[
\rho (t, \cdot)  = \left(\frac {\rho_*}{\det D\CX (t, \cdot)} \right) \circ \CX(t, \cdot)^{-1}: \CX(t, S_\mu) \to \BFR^+ \cup\{0\}
\]
and extended as $0$ outside $\CX(t, S_\mu) \subset \BFR^3$. Through standard calculus of variation procedure (with respect to $\CX$), it is straight forward to verify that $\CX(t)$ is a critical path of $\CA$ if and only if $(\rho, u = \CX_t \circ \CX^{-1})$, which are supported on $\CX(t, S_\mu)$, solves the Euler-Poisson system \eqref{continuity-EP}-\eqref{Poisson-EP}. The reference density $\rho_*$ plays the role of a parameter not evolve in $t$. The conserved energy of this Lagrangian system is
\[
E = \int_{S_\mu}  \frac 12  |\CX_t|^2 \rho_* dy + \int_{\CX(t, S_\mu)} \Phi (\rho) dx -  \frac 1{8\pi}\int_{\BFR^3} |\nabla V|^2 dx = \int_{ \BFR^3} \frac 12  \rho |u|^2  +  \Phi (\rho)  -  \frac 1{8\pi} |\nabla V|^2 dx.
\]
One observes that the potential energy consisting of the enthalpy and gravity depends on $\CX$ only through the density $\rho$. Therefore the action functional is invariant under the transformation $\CX(t) \to \CX(t) \circ \CT$, where $\CT$ belongs to the group $\CG$ of diffeomorphism on $S_\mu$ preserving $\rho_*$, namely
\[
\CG= \{ \text{diffeomorphism } \CT: S_\mu \to S_\mu \mid (\rho_* \circ \CT) \det D\CT = \rho_*\}.
\]
The Euler-Poisson system \eqref{continuity-EP}-\eqref{Poisson-EP} in the Eulerian formulation is essentially a reduction of the Lagrangian system due to this relabeling symmetry where $\rho(t, \cdot)$ and $u(t, \cdot)$ are supported on $\CX(t, S_\mu)$.

The non-rotating star $(\rho_\mu, u\equiv0)$ corresponds to the stationary solution $\CX\equiv id$ along with $\rho_* = \rho_\mu(|x|)$ which is a critical point of the potential energy. Let $\CX(t, x, \ep)$ be a family of solutions (parametrized by $\ep$) in the Lagrangian formulation with the reference density $\rho_*(x, \ep)$ such that $\CX(t, x, 0)=x$, $\rho_*(x, 0)= \rho_\mu(x)$, for all $x\in S_\mu$. The linearized system at $\CX=id$ and $\rho_\mu$ governs the dynamics of the leading order variation $\tilde \CX = \p_\ep \CX|_{\ep=0}$, which also involves $\sigma = \p_\ep \rho|_{\ep=0}$.
The corresponding quantities in the Eulerian formulation are
\[
\sigma = \p_\ep \rho |_{\ep=0}  = \big( \p_\ep \rho_* - \nabla \cdot (\rho_\mu \tilde \CX)\big), \quad v = \p_\ep u|_{\ep=0}  = \p_\ep (\CX_t \circ \CX^{-1})|_{\ep=0} =\tilde \CX_t,
\]
which are supported on $S_\mu$. The associated action of the linearized Lagrangian system, which is simply the quadratic part of $\CA$, can be expressed more conveniently using $\sigma$ as
\[
\CA_2= \frac 12 \int_{S_\mu} \rho_\mu  |\tilde \CX_t|^2  - \Phi'' (\rho_\mu) \sigma^2 dy +  \frac 1{8\pi}\int_{\BFR^3} |\nabla (|x|^{-1} * \sigma)|^2 dx.
\]
Using the above formula of $\sigma$, which also implies \eqref{LEP-1}, one obtains the linearized equation through the variation of $\CA_2$ with respect to $\tilde \CX$
\[
- \tilde \CX_{tt} - \nabla \cdot \big(\Phi'' (\rho_\mu) \sigma + |x|^{-1} * \sigma\big) =0,
\]
which is equivalent to \eqref{LEP2}.
The quadratic part
\begin{align*}
E_2
=&  \frac 12 \int_{S_\mu}  \rho_\mu |v|^2  +  \Phi'' (\rho_\mu) \sigma^2 dx -  \frac 1{8\pi}\int_{\BFR^3} |\nabla (|x|^{-1} * \sigma)|^2 dx = H_\mu(\sigma, v),
\end{align*}
of the the nonlinear energy $E$, which is equal to the Hamiltonian $H_\mu(\sigma, v)$ of the linearized Euler-Poisson system defined in \eqref{defn-H-linearized-EP}, is conserved by these linearized solutions.

Through the Legendre transformation $U =\rho_* \CX_t$, the Lagrangian structure with the action $\CA$ induces a natural Hamiltonian structure of the Euler-Poisson system with the Hamiltonian $\CH$ and the standard symplectic structure $J$:
\[
\CH(\CX, U) = \int_{S_\mu}  \frac 1{2 \rho_*}  |U|^2 dy + \int_{\CX(t, S_\mu)} \Phi (\rho) dx -  \frac 1{8\pi}\int_{\BFR^3} |\nabla V|^2 dx, \quad J = \begin{pmatrix} 0 & 1 \\ -1 & 0 \end{pmatrix}.
\]
It might be possible to apply the general results in Section \ref{section-abstract} to analyze the linearized Euler-Poisson system at $(\rho_\mu, 0)$ as a linear Hamiltonian system of the linearized Lagrangian map $\p_\ep \CX$ and momentum $\p_\ep U$. As in the nonlinear case, one could expect such system to be reduced to \eqref{LEP-1}-\eqref{LEP2} through a reduction due to the relabeling symmetry. We carried out the analysis directly on \eqref{LEP-1}-\eqref{LEP2} with the different symplectic structure $\CJ_\mu$, where the large symmetry group (corresponding to additional infinite kernel dimensions) has been reduced and stability/instability is directly on the linearized density and velocity.

\begin{center}
{\Large Acknowledgement}
\end{center}

This work is supported partly by the NSF grants DMS-1715201 and DMS-2007457
(Lin) and DMS-1900083 (Zeng).

\end{document}